\documentclass[12pt]{amsart}
\usepackage{amssymb}
\usepackage{hyperref} 
\usepackage{dsfont}				
\usepackage{mathrsfs}

\usepackage[all]{xy}
\newtheorem{theorem}{Theorem}[section]
\newtheorem{proposition}[theorem]{Proposition}
\newtheorem{coro}[theorem]{Corollary}
\newtheorem{lemma}[theorem]{Lemma}
\newtheorem{rem}[theorem]{Remark}
\newtheorem{definition}[theorem]{Definition}
\newtheorem{example}[theorem]{Example}

\renewcommand{\epsilon}{\varepsilon}
\newcommand{\eps}{\epsilon}

\newcommand\PP{D}

\newcommand\C{\mathbb{C}}
\newcommand\N{\mathbb{N}}
\newcommand\R{\mathbb{R}}
\newcommand\calS{\mathcal{S}}

\newcommand\Z{\mathbb{Z}}

\newcommand{\calM}{\ensuremath{\mathcal{M}}}

\newcommand{\scrC}{\ensuremath{\mathscr{C}}}

\newcommand{\norm}[1]{\left\|#1\right\|}
\newcommand{\abs}[1]{\left|#1\right|}

\newcommand{\Ran}{\textsf{R}}

\DeclareMathOperator{\supp}{supp}
\DeclareMathOperator{\loc}{loc}
\newcommand{\Eins}{\ensuremath{\mathds{1}}}
\DeclareMathOperator{\osc}{Osc}
\DeclareMathOperator*{\esssup}{ess\,sup}

\let\div=\relax
\DeclareMathOperator{\div}{div}

\def\Xint#1{\mathchoice
   {\XXint\displaystyle\textstyle{#1}}%
   {\XXint\textstyle\scriptstyle{#1}}%
   {\XXint\scriptstyle\scriptscriptstyle{#1}}%
   {\XXint\scriptscriptstyle\scriptscriptstyle{#1}}%
   \!\int}
\def\XXint#1#2#3{{\setbox0=\hbox{$#1{#2#3}{\int}$}
     \vcenter{\hbox{$#2#3$}}\kern-.5\wd0}}

\def\aver#1{\Xint-_{#1}}

\def \vsp {\vspace{6pt}}

\newcommand\restr[2]{{% we make the whole thing an ordinary symbol
  \left.\kern-\nulldelimiterspace % automatically resize the bar with \right
  #1 % the function
  \vphantom{\big|} % pretend it's a little taller at normal size
  \right|_{#2} % this is the delimiter
  }}

\numberwithin{theorem}{section}
\numberwithin{equation}{section}

\title[Sobolev algebras]{Sobolev algebras through heat kernel estimates}

\author{Fr\'ed\'eric Bernicot}
\address{Fr\'ed\'eric Bernicot, CNRS - Universit\'e de Nantes, Laboratoire Jean Leray, 2 rue de la Houssini\`ere, 44322 Nantes cedex 3.  France}
\email{frederic.bernicot@univ-nantes.fr}

\author{Thierry Coulhon}
\address{Thierry Coulhon, Paris Sciences et Lettres Research University, France}
\email{thierry.coulhon@univ-psl.fr}

\author{Dorothee Frey}
\address{Dorothee Frey, Mathematical Sciences Institute, The Australian National University, Canberra ACT 0200, Australia}
\curraddr{Universit\'e Paris-Sud, Laboratoire de Math\'ematiques, UMR 8628 du CNRS, 91405 Orsay, France}
\email{dorothee.frey@univ-nantes.fr}

\thanks{FB's research was supported by the ANR projects AFoMEN no. 2011-JS01-001-01 and HAB no. ANR-12-BS01-0013.\\
TC's research was  done while he was employed by the Australian National University and was supported by the 
Australian Research Council (ARC) grant DP  130101302. \\DF's research was supported by the 
Australian Research Council (ARC) grants DP 110102488 and DP 120103692.}
\date{\today}

\let\oldtocsection=\tocsection
\let\oldtocsubsection=\tocsubsection
\let\oldtocsubsubsection=\tocsubsubsection

\renewcommand{\tocsection}[2]{\hspace{0em}\oldtocsection{#1}{#2}}
\renewcommand{\tocsubsection}[2]{\hspace{1em}\oldtocsubsection{#1}{#2}}
\renewcommand{\tocsubsubsection}[2]{\hspace{2em}\oldtocsubsubsection{#1}{#2}}

\begin{document}
\begin{abstract} On a doubling metric measure space $(M,d,\mu)$ endowed with a ``carr\'e du
champ", let  $\mathcal{L}$ be  the associated Markov generator and $\dot L^{p}_\alpha(M,\mathcal{L},\mu)$  the corresponding homogeneous Sobolev space of order $0<\alpha<1$ in $L^p$, $1<p<+\infty$,  with norm $\left\|\mathcal{L}^{\alpha/2}f\right\|_p$.  We give sufficient  conditions on the  heat semigroup $(e^{-t\mathcal{L}})_{t>0}$  for the  spaces $\dot L^{p}_\alpha(M,\mathcal{L},\mu) \cap L^\infty(M,\mu)$ to be  algebras for the pointwise product. Two approaches are developed, one using paraproducts (relying on extrapolation to prove their boundedness) and a second one through geometrical square functionals (relying on sharp estimates involving oscillations). A chain rule and a paralinearisation result are also given. 
In comparison with previous results (\cite{CRT,BBR}), the main improvements consist in the fact that we  neither require any Poincar\'e inequalities nor $L^p$-boundedness of Riesz transforms, but only $L^p$-boundedness of the gradient of the semigroup. As a consequence, in the range $p\in(1,2]$, the Sobolev algebra property is shown  under  Gaussian upper estimates of the heat kernel only. 
\end{abstract}

\maketitle

\begin{quote}
\footnotesize\tableofcontents
\end{quote}

\section{Introduction}

It is well-known that in the Euclidean space $\mathbb{R}^n$ (endowed with its canonical non-negative Laplace operator $\Delta$),  the Bessel-type Sobolev space 
$$L^p_{\alpha}(\mathbb{R}^n)=\left\lbrace f \in L^p; \, \Delta^{\alpha/2}f \in L^p  \right\rbrace,
$$
 is an algebra for the pointwise product for all $1< p<+\infty$ and $\alpha >0$ such that $\alpha p>n$. This result is due to Strichartz in \cite{St}, where the Sobolev norm was shown to be equivalent to the $L^p$-norm of a suitable quadratic functional.

Twenty years after Strichartz's work, Kato and Ponce \cite{KP} gave a stronger result, still in the Euclidean space. They proved that for all $p\in(1,+\infty)$ and  $\alpha >0$, $L^p_{\alpha}(\mathbb{R}^n)\cap L^{\infty}(\mathbb{R}^n)$ is an algebra for the pointwise product. Later on Gulisashvili and Kon \cite{GK}  considered the homogeneous Sobolev spaces $\dot L^p_{\alpha}(\mathbb{R}^n)$ and  proved  the even stronger result that under the same conditions, $\dot L^p_{\alpha}(\mathbb{R}^n)\cap L^{\infty}(\mathbb{R}^n)$ is an algebra for the pointwise product. 
These results come with the associated Leibniz rules.

One way to obtain these properties and more general Leibniz rules  in the Euclidean setting is to use paraproducts (introduced by Bony in \cite{bony} and later used by Coifman and Meyer \cite{cm, Meyer}, see also \cite{taylor}) and the boundedness of these bilinear operators on $L^{\infty}(\mathbb{R}^n)\times \dot L^p_{\alpha}(\mathbb{R}^n)$. This powerful tool allows one to split the pointwise product into two terms, the regularity of which can be easily computed from the regularity of the two factors in the product. Moreover, paraproducts also yield a paralinearisation formula, which allows one to linearise a nonlinearity in Sobolev spaces. 

The main motivation of the inequalities deriving from such Leibniz rules and algebra properties comes from the study of nonlinear PDEs. In particular, to obtain well-posedness results in Sobolev spaces for some semi-linear PDEs, one has to understand how the nonlinearity acts on Sobolev spaces. This topic, the action of a nonlinearity on Sobolev spaces (and more generally on Besov spaces), has given rise to numerous works in the Euclidean setting where the authors attempt to obtain the minimal regularity on a nonlinearity $F$ such that the following property holds
$$ f\in B^{s,p} \Longrightarrow F(f) \in B^{\alpha,p},$$
where $B^{\alpha,p}$ can be  Sobolev or  Besov spaces (see for example \cite{Sickel}, \cite{RS} or  \cite{Bourdaud}).

It is  natural to look for an extension of these results beyond Euclidean geometry, as was pioneered in \cite{bohnke}.  In \cite{CRT}, Coulhon, Russ and Tardivel-Nachef extended the Strichartz approach,
in the case $0<\alpha<1$, to the case of Lie groups with polynomial volume growth and  Riemannian manifolds with non-negative Ricci curvature. The proof works as soon as one has the volume  doubling property as well as a   pointwise Gaussian upper bound for the gradient of the heat kernel. More recently, on a doubling Riemannian manifold equipped with an operator satisfying suitable heat kernel bounds, Badr, Bernicot and Russ \cite{BBR} have shown similar results under  Poincar\'e inequalities and  boundedness of the Riesz transform, but without assuming  pointwise bounds on the gradient of the heat kernel (note that the latter imply the boundedness of the Riesz transform, see \cite{ACDH}).  See also \cite{BS} for further developments and  \cite{GS}, with a quite different approach, for the case of Besov spaces on  Lie groups with polynomial volume growth.

Our aim in the present work is to improve these  results while working in the general setting of a Dirichlet metric measure space. Our standing assumptions will be the volume doubling property and a Gaussian upper estimate for the heat kernel. We show in particular that the algebra property always  holds for $1<p<+\infty$ (which is reminiscent of the  results in \cite{CD1} and \cite{ACDH}) under  $L^q$-bounds on the gradient of the heat semigroup for some $q\in (p, +\infty]$, which is much weaker than what is assumed in \cite{CRT,BBR} (mainly boundedness of Riesz transform and some Poincar\'e inequalities). The precise results are stated in Theorems \ref{thm:summary} and \ref{thm:summary2} below. 

\subsection{The Dirichlet form setting}

Let $M$ be a locally compact separable metrisable space, equipped with a Borel measure $\mu$, finite on compact sets and strictly positive on any non-empty open set. 
%We assume that $M$ is non-bounded and as a consequence, we know that $\mu(M)=\infty$.
For $\Omega$ a measurable subset of $M$, we shall  denote $\mu\left(\Omega\right)$ by $\left|\Omega\right|$.

Let $\mathcal{L}$ be a non-negative self-adjoint operator on $L^2(M,\mu)$ with dense domain ${\mathcal D}\subset L^2(M,\mu)$. Denote by $\mathcal{E}$ the associated quadratic form, that is
$${\mathcal E}(f,g)=\int_Mf\mathcal{L}g\,d\mu,$$  and by  $\mathcal{F}$ its domain, which contains ${\mathcal D}$. 
If $\mathcal{E}$ is a  Dirichlet form (see \cite{FOT} for a definition), it follows (see \cite[Theorem 1.4.2]{FOT}) that the space $L^\infty(M,\mu) \cap {\mathcal F}$ is an algebra and
\begin{equation} \sqrt{{\mathcal E}(fg,fg)} \leq \|f\|_\infty \sqrt{{\mathcal E}(g,g)} + \sqrt{{\mathcal E}(f,f)} \|g\|_\infty, \quad \forall f,g\in L^\infty(M,\mu) \cap {\mathcal F}. \label{eq:coeur} \end{equation}

The operator $\mathcal{L}$ generates a strongly continuous  semigroup  $(e^{-t\mathcal{L}})_{t>0}$ of self-adjoint contractions on $L^2(M,\mu)$. In addition $(e^{-t\mathcal{L}})_{t>0}$  is submarkovian, that is
$0\leq e^{-t\mathcal{L}}f\leq 1$ if  $0\le f\leq 1$. It follows that the semigroup $(e^{-t\mathcal{L}})_{t>0}$ is uniformly  bounded  on $L^p(M,\mu)$ for $p\in[1,+\infty]$ and strongly continuous for $p\in[1,+\infty)$.
Also, $(e^{-t\mathcal{L}})_{t>0}$ is   bounded analytic on $L^p(M,\mu)$ for $1<p<+\infty$  (see \cite{topics}), which means  that  $(t\mathcal{L}e^{-t\mathcal{L}})_{t>0}$ is bounded on $L^p(M,\mu)$ uniformly in $t>0$.
 
\bigskip

Let $\mathcal{C}_0(M)$ denote the space of continuous functions on $M$ which vanish at infinity. For  $1\le p<+\infty$ and  $\alpha>0$, define
 $\dot{L}^p_{\alpha}(M,\mathcal{L},\mu)$ as the completion of
$$
\left\{f\in \mathcal{C}_0(M);\mathcal{L}^{\alpha / 2}f\in L^p(M,\mu)\right\}.
$$
for the norm $\|f\|_{p,\alpha}:=\left\|\mathcal{L}^{\alpha/2}f\right\|_p$.
 The space $\dot{L}^p_{\alpha}(M,\mathcal{L},\mu)$ may not be a Banach space of functions, but
 $\dot{L}^p_{\alpha}(M,\mathcal{L},\mu) \cap L^{\infty}(M,\mu)$, equipped with the 
norm $\left\|\mathcal{L}^{\alpha/2}f\right\|_p+\left\|f\right\|_\infty$, obviously is.

\begin{definition} For $\alpha>0$ and $p\in(1,+\infty)$ we say that property $A(\alpha,p)$ holds if: 
\begin{itemize}
\item the space $\dot{L}^p_{\alpha}(M,\mathcal{L},\mu) \cap L^{\infty}(M,\mu)$ is an algebra for the pointwise product;
\item and the Leibniz rule inequality is valid:
$$ \| fg \|_{p,\alpha} \lesssim \|f\|_{p,\alpha} \|g\|_{\infty} +  \|f\|_{\infty} \|g\|_{p,\alpha}, \quad \forall\,f,g\in \dot{L}^p_{\alpha}(M,\mathcal{L},\mu)\cap L^\infty(M,\mu).$$
\end{itemize}
\end{definition}

One could also consider local versions of $A(\alpha,p)$ as in \cite{CRT} and \cite{BBR}; we leave this to the reader.

In the present paper we restrict ourselves to the range $\alpha\in (0,1)$. 
We shall see below that the case $\alpha=1$ is very much connected to the Riesz transform problem (see \cite{CD1}, \cite{ACDH} and references therein).

Note that, as in the Riesz transform problem, the case $p=2$ is trivial. Indeed, \eqref{eq:coeur} and the identity
$\mathcal{E}(f,f)=\left\|\mathcal{L}^{1/2}f\right\|_2^2$ for $f\in\mathcal{D}$ 
obviously imply $A(1,2)$. 
Now, since  $\mathcal{E}_\alpha(f,g)=\int_M (\mathcal{L}^\alpha f)  \,g \,d\mu$ is also a Dirichlet form for $0<\alpha<1$, it follows that
for the same reason $A(\alpha,2)$  holds for $0<\alpha\le 1$.

Assume from now on that the Dirichlet form $\mathcal{E}$ is  strongly local and regular  (see \cite{FOT, GSC} for precise definitions).
There exists  an energy measure $d\Gamma$, that is a signed measure depending in a bilinear way on $f,g\in{\mathcal F}$ such that
\begin{equation}\label{energy}
\mathcal{E}(f,g)=\int_M d\Gamma(f,g)
\end{equation}
 for all $f,g\in\mathcal{F}$. 
According to Beurling-Deny and Le Jan formula, the energy measure encodes a kind of Leibniz rule, which is (see \cite[Section 3.2]{FOT})
\begin{equation}
 d\Gamma(fg,h) = fd\Gamma(g,h) + g d\Gamma(f,h), \quad \forall f,g,h \in L^\infty \cap {\mathcal F}. \label{eq:leibniz}
\end{equation}

One can define a pseudo-distance $d$ associated with $\mathcal{E}$ by
 \begin{equation}\label{defd}
d(x,y):=\sup\{f(x)-f(y) ; f\in \mathcal{F}\cap \mathcal{C}_0(M) \mbox{ s.t. }d\Gamma(f,f)\le d\mu\}.
\end{equation}
Throughout the whole paper, we assume that the pseudo-distance $d$ separates points, is finite everywhere, continuous and defines the initial topology of $M$, and that $(M,d)$ is complete (see \cite{ST1} and \cite[Section 2.2.3]{GSC} for details).

When we are in the above situation, we shall say that $(M,d,\mu, {\mathcal E})$ is a metric  measure (strongly local and regular)  Dirichlet space. This is slightly abusive, in the sense that in the above presentation $d$ follows from ${\mathcal E}$.

 For all $x \in M$ and all $r>0$, denote by $B(x,r)$ the open ball for the metric $d$ with centre $x$ and radius $r$, and  by $V(x,r)$ its measure $|B(x,r)|$.  For a ball $B$ of radius $r$ and a real $\lambda>0$, denote by $\lambda B$   the ball concentric  with $B$ and with radius $\lambda r$. We shall sometimes denote by $r(B)$ the radius of a ball $B$. We will use $u\lesssim v$ to say that there exists a constant $C$ (independent of the important parameters) such that $u\leq Cv$, and $u\simeq v$ to say that $u\lesssim v$ and $v\lesssim u$. Moreover, for $\Omega\subset M$ a subset of finite and non-vanishing measure and $f\in L^1_{loc}(M,\mu)$, $\aver{\Omega} f \, d\mu=\frac{1}{|\Omega|} \int f \, d\mu$ denotes the average of $f$ on $\Omega$. 

We shall assume  that   $(M,d,\mu)$ satisfies the volume doubling property, that is
  \begin{equation}\label{d}\tag{$V\!D$}
     V(x,2r)\lesssim  V(x,r),\quad \forall~x \in M,~r > 0.
    \end{equation}
As a consequence, there exists  $\nu>0$  such that
     \begin{equation*}\label{dnu}\tag{$V\!D_\nu$}
      V(x,r)\lesssim \left(\frac{r}{s}\right)^{\nu} V(x,s),\quad \forall~r \ge s>0,~ x \in M,
    \end{equation*}
which implies
     \begin{equation*}
      V(x,r)\lesssim \left(\frac{d(x,y)+r}{s}\right)^{\nu} V(y,s),\quad \forall~r \ge s>0,~ x,y \in M.
    \end{equation*}
Another easy consequence of \eqref{d} is that  balls with a non-empty intersection and comparable radii have comparable measures. Finally, \eqref{d} implies that the semigroup $(e^{-t\mathcal{L}})_{t>0}$ has the conservation property (see \cite{Grigo, ST1}), which means that
\begin{equation}\label{cons}
 e^{-t\mathcal{L}}1=1 ,\qquad \forall t>0.
 \end{equation}
Indeed, in a rather subtle way, the above assumptions exclude the case of a non-empty boundary with Dirichlet boundary conditions, see the comments
in \cite[pp. 13--14]{GS}.

We shall say that $(M,d,\mu, {\mathcal E})$ is a doubling metric  measure   Dirichlet space if it is a metric measure space  endowed with a strongly local and regular Dirichlet form  and satisfying \eqref{d}.

\subsection{Heat kernel and regularity estimates}

As in  \cite{CRT} and \cite{BBR}, a major role in our assumptions will be played by heat kernel estimates.

The semigroup $(e^{-t\mathcal{L}})_{t>0}$ may or may not  have a  kernel, that is for all $t>0$ a measurable function $p_t:M\times M\to\R_+$ such that
$$e^{-t\mathcal{L}}f(x)=\int_Mp_t(x,y)f(y)\,d\mu(y), \quad \mbox{a.e. }x\in M.$$  If it does, $p_t$ is called the heat kernel associated with $\mathcal{L}$ (or rather with $(M,d,\mu,\mathcal{E}))$. Then $p_t(x,y)$ is non-negative and symmetric in $x,y$,  since $e^{-t\mathcal{L}}$ is positivity preserving and self-adjoint for all $t>0$.
One may  naturally ask for upper estimates of $p_t$ (see for instance the recent article \cite{BCS} and the many relevant references therein). A typical upper estimate is
  \begin{equation}\tag{$DU\!E$}
 p_{t}(x,y)\lesssim
\frac{1}{\sqrt{V(x,\sqrt{t})V(y,\sqrt{t})}}, \quad \forall~t>0,\,\mbox{a.e. }x,y\in
 M.\label{due}
\end{equation}
This estimate is called on-diagonal because if $p_t$ happens to be continuous then \eqref{due} can be rewritten as
  \begin{equation}
p_{t}(x,x)\lesssim
\frac{1}{V(x,\sqrt{t})}, \quad \forall~t>0,\,\forall\,x\in
 M.
\end{equation}

Under  \eqref{d}, \eqref{due} self-improves into a Gaussian upper estimate (see \cite[Theorem 1.1]{Gr1} for the Riemannian case,   \cite[Section 4.2]{CS} for a metric measure space setting):
\begin{equation}\tag{$U\!E$}
p_{t}(x,y)\lesssim
\frac{1}{V(x,\sqrt{t})}\exp
\left(-\frac{d^{2}(x,y)}{Ct}\right), \quad \forall~t>0,\, \mbox{a.e. }x,y\in
 M.\label{UE}
\end{equation}

To formulate some other assumptions, we will need a notion of pointwise length of the gradient. The Dirichlet form $\mathcal{E}$ admits a ``carr\'e du champ'' (see for instance \cite{GSC} and the references therein) if for all $f,g\in\mathcal{F}$ the energy measure $d\Gamma(f,g)$ is absolutely continuous with respect to $\mu$. Then the  density $\Upsilon(f,g) \in L^1(M,\mu)$ of $d\Gamma(f,g)$ is called the ``carr\'e du champ'' and satisfies the following inequality
 \begin{equation}   |\Upsilon(f,g)|^2 \leq \Upsilon(f,f) \Upsilon(g,g).  \label{eq:carre} \end{equation} In the sequel, when we assume that $(M,d,\mu,\mathcal{E})$ admits
 a ``carr\'e du champ'',  we shall abusively denote $\left[\Upsilon(f,f)\right]^{1/2}$ by $ |\nabla f|$.
 This has the advantage to stick to the more intuitive and classical Riemannian notation, but one should not forget that one works 
 in a much more general setting (see for instance \cite{GSC} for examples), and that one never uses differential calculus in the classical sense.

We will also use estimates on the gradient (or ``carr\'e du champ") of the semigroup, which were introduced in  \cite{ACDH}: for $p\in[1,+\infty]$, consider
\begin{equation} \label{Gp}
\sup_{t>0} \|\sqrt{t}|\nabla e^{-t\mathcal{L}} |\|_{p\to p} <+\infty \tag{$G_p$},
\end{equation}
which is equivalent to the interpolation inequality
\begin{equation} \label{mult}  \||\nabla f|\|_p^2 \lesssim \|\mathcal{L} f\|_p\| f\|_p, \qquad \forall\,f\in {\mathcal D}
\end{equation}
(see \cite[Proposition 3.6]{CS2}). Note that $(G_p)$ always holds  for $1<p\leq 2$.
For more about $(G_p)$, we refer  to \cite{ACDH}, to the introduction of \cite{BCF1},  and to the references therein.
This notion was introduced in \cite{ACDH} to understand the stronger notion of boundedness of the Riesz transform $|\nabla {\mathcal L}^{-1/2}|$ (we refer the reader to \cite{ACDH}  for more details about these two notions and how they are related and to \cite{BF2} for recent results in this area). Given $p\in(1,+\infty)$, one says  the Riesz transform is bounded on $L^p(M,\mu)$ if
\begin{equation} \label{rp}
 \||\nabla f |\|_{p} \lesssim \| \mathcal{L}^{1/2}f \|_{p},  \quad \forall\,f\in {\mathcal D}, \tag{$R_p$}
\end{equation}
and that the reverse Riesz transform is bounded on $L^p(M,\mu)$ if
\begin{equation} \label{rrp}
\| \mathcal{L}^{1/2}f \|_{p}  \lesssim  \||\nabla f |\|_{p},  \quad \forall\,f\in {\mathcal D}. \tag{$RR_p$}
\end{equation}
If both estimates hold true, then
\begin{equation} \label{ep}
 \||\nabla f |\|_{p} \simeq \|  \mathcal{L}^{1/2}f \|_{p},  \quad \forall\,f\in {\mathcal D}. \tag{$E_p$}
\end{equation}

It is then clear, using \eqref{eq:leibniz}  and \eqref{eq:carre}, that  $(E_p)$ implies $A(1,p)$. One of the main objectives of this work is to prove $A(\alpha,p)$ for $0<\alpha<1$ without assuming \eqref{ep} or \eqref{rp}.

We can now formulate the $L^p$ version of the scale-invariant Poincar\'e inequalities, which may or may not be true, depending on $p\in[1,+\infty)$. More precisely, for $p\in[1,+\infty)$, one says that  $(P_p)$ holds if
\begin{equation}\tag{$P_p$}
 \left( \aver{B} \left| f - \aver{B} f \, d\mu \right|^p d\mu \right)^{1/p} \lesssim r \left(\aver{B} |\nabla f|^p \, d\mu \right)^{1/p}, \qquad \forall\, f\in {\mathcal F}, \label{Pp}
\end{equation}
where $B$ ranges over balls in $M$ of radius $r$.
Recall that $(P_p)$ is weaker and weaker as $p$ increases, that is $(P_p)$ implies $(P_q)$ for $p<q<+\infty$.
Also, under \eqref{d}, $(P_2)$ is equivalent to  the Gaussian  lower bound matching \eqref{UE}, see \cite{BCF1} and the references therein.
For more about $(P_p)$, we refer to \cite{HaKo} and to the introduction of \cite{BCF1}.

\subsection{Main results}

The original approach by Strichartz to  the Sobolev algebra property in \cite{St}, and later also used  in  \cite{CRT,BBR},  relies on the    functional
$$S_\alpha f(x)=\left(\int_0^{+\infty}\left[\aver{B(x,r)}|f-f(x)|\,d\mu\right]^2\frac{dr}{r^{2\alpha+1}}\right)^{1/2},$$
which measures the regularity of the function $f$ by averaging its oscillations at all  scales (see Section \ref{sec:osci} for more details). 
If one proves
\begin{equation}
\label{quad} \|S_\alpha f\|_p\simeq \|\mathcal{L}^{\alpha/2}f\|_p,\quad \forall\,f\in\mathcal{F},\tag*{$E(\alpha,p)$}
\end{equation}
then it is easy to see that $A(\alpha,p)$ follows.

In the present paper, we shall rather rely on  the paraproduct approach, using a notion of paraproduct associated with the underlying operator $\mathcal{L}$ and the corresponding semigroup that was recently introduced in \cite{B}, \cite{F},  \cite{BF}. 
This requires slightly weaker assumptions. On the other hand, Strichartz's approach yields a stronger chain rule (requiring less regularity on the nonlinearity). This is why we shall also study property $E(\alpha,p)$ in Section \ref{quadra}. Note also that $E(\alpha,p)$ may be considered as a fractional version of $(E_p)$.

 Let us now recall some tools  that have been studied in \cite{BCF1} (and previously, see references therein), namely an inhomogeneous $L^2$ version of the De Giorgi property, as well as some H\"older regularity estimates for the heat semigroup. 
  
\begin{definition}[$L^2$ De Giorgi property] For $\kappa\in(0,1)$, we say that $({DG}_{2,\kappa})$ holds if the following is satisfied:  for all $r\leq R$, every pair of concentric balls $B_r,B_R$ with  respective radii $r$ and $R$, and for every function $f\in {\mathcal D}$, one has
$$ \left(\aver{B_r} |\nabla f|^2 \, d\mu \right)^{1/2} \lesssim    \left(\frac{R}{r}\right)^\kappa \left[ \left(\aver{B_R} |\nabla f|^2 \, d\mu \right)^{1/2}
+ R  \|\mathcal{L}f\|_{L^\infty(B_R)}\right].$$
We sometimes omit the parameter $\kappa$, and write $({DG}_{2})$ if $({DG}_{2,\kappa})$ is satisfied for some $\kappa \in (0,1)$.
\end{definition}

For more details and background, see \cite{BCF1}. We just point out that $(DG_2)$ is  implied by the Poincar\'e inequality $(P_2)$.

\begin{definition} For $p,q\in[1,+\infty]$ and $\eta\in(0,1]$, we shall say that property \eqref{hp} holds if for every   $0<r\leq \sqrt{t}$, every pair of  concentric balls $B_r$, $B_{\sqrt{t}}$ with respective radii $r$ and $\sqrt{t}$, and every function $f\in L^{p}(M,\mu)$,
\begin{equation} \label{hp}
q\text{-}\osc_{B_r}(e^{-t\mathcal{L}} f):=\left(\aver{B_r} \left| e^{-t\mathcal{L}}f - \aver{B_r}e^{-t\mathcal{L}}f \, d\mu \right|^{q} \, d\mu \right)^{1/q} \lesssim \left( \frac{r}{\sqrt{t}} \right)^\eta  \left|B_{\sqrt{t}}\right|^{-1/p} \|f\|_p
 \tag{$H_{p,q}^\eta$},
\end{equation}
with the obvious modification for $p=\infty$.

We shall say that \eqref{gplocal} is satisfied if, for some $(\gamma_\ell)$ exponentially decreasing coefficients and for all $0<r\leq \sqrt{t}$, every ball $B_r$ of radius, and every function $f\in L^{p}_{\textrm{loc}}(M,\mu)$, 
\begin{equation} \label{gplocal}
q\text{-}\osc_{B_r}(e^{-t\mathcal{L}} f) \lesssim \left( \frac{r}{\sqrt{t}} \right)^\eta  \sum_{\ell \geq 0}  \gamma_{\ell} \left( \aver{2^\ell B_{\sqrt{t}}} |f|^p\, d\mu \right)^{1/p}.
 \tag{$\overline{H}_{p,q}^\eta$}
\end{equation}
\end{definition}

Then  the following  holds.

\begin{proposition} \label{prop:bcf1} Let $(M,d,\mu,\mathcal{E})$ be a metric measure Dirichlet space with a ``carr\'e du champ" satisfying \eqref{d} and \eqref{due}. We have
\begin{itemize}
\item The lower Gaussian estimates for the heat kernel $(LE)$ are equivalent to the existence of some $p\in(1,+\infty)$ and some $\eta>0$ such that $(H_{p,p}^\eta)$ holds; 
\item ($H_{p,p}^\eta$) implies ($H_{p,\infty}^\eta$) and ($\overline{H}_{1,\infty}^\lambda$) for every $\lambda\in[0,\eta)$;
\item Moreover, for every $\lambda\in(0,1]$ the property $\bigcap_{\eta<\lambda} (H_{p,p}^\eta) $ is independent on $p\in[1,+\infty]$ and will be called 
$$ (H^\lambda) := \bigcup_{p\in[1,+\infty]} \bigcap_{\eta<\lambda} (H_{p,p}^\eta)= \bigcap_{\eta<\lambda} \bigcup_{p\in[1,+\infty]} (H_{p,p}^\eta).$$
\end{itemize}
\end{proposition}

We refer to \cite[Theorem 3.4]{BCF1} for the first part and to Appendix \ref{AppA} for the  last  two statements. 

We can now state our main results.

\begin{theorem} \label{thm:summary} Let $(M,d,\mu, {\mathcal E})$  be a doubling metric  measure   Dirichlet space with a ``carr\'e du champ'' satisfying \eqref{dnu} and \eqref{due}. Then
\begin{itemize}
 \item[(a)]   $A(\alpha,p)$ holds for every $p\in(1,2]$ and $\alpha\in(0,1)$, and for every $p\in(2,+\infty)$ and $\alpha\in\left(0, 1-\nu\left(\frac{1}{2} -\frac{1}{p}\right)\right)$; 
 \item[(b)] Under   $(G_{p_0})$ for some $p_0\in(2,+\infty)$, $A(\alpha,p)$ holds for every $p\in(1,p_0]$ and $\alpha\in(0,1)$, and for every $p\in(p_0,+\infty)$ and $\alpha\in\left(0, 1-\nu\left(\frac{1}{p_0} -\frac{1}{p}\right)\right)$;
 \item[(c)] Under   $(G_{p_0})$ and $({DG}_{2,\kappa})$ for some $2<p_0\leq +\infty$ and $\kappa\in(0,1)$, $A(\alpha,p)$ holds for every $p\in(1,p_0]$ and $\alpha\in(0,1)$, and for every $p>p_0$ and $\alpha\in\left(0, 1-\kappa \left(1-\frac{p_0}{p}\right)\right)$;
\item[(d)] Under  $(H^\eta)$ for some $\eta\in(0,1]$, $A(\alpha,p)$ holds for every $\alpha\in(0,\eta)$ and $p\in(1,+\infty)$.
\end{itemize}
\end{theorem}

Since $(G_{2})$ always holds,  (a) is nothing but  (b) in the case $p_0=2$.
Statement (a) is proven in Theorem \ref{thm:tent-extrap-small} (for $p\leq 2$) and in Theorem \ref{thm:extrap>22} (for $p> 2$), statement (b) in Theorem \ref{gaza} (for $p<p_0$) and Theorem \ref{thm:extrap>2} (for $p\geq p_0$), statement (c)  in Theorem \ref{thm:tent-extrap-large-3},
and statement (d) in Theorem \ref{thm:osci}. Statement (d) had been announced in \cite[p.333]{CRT}.

\begin{rem}
An alternative method of proof for Theorem $\ref{thm:summary}$ $(a)$ - $(b)$ is the following: Instead of using extrapolation methods on Lebesgue spaces (see \cite{BK},\cite{AM},\cite{BZ}) as we do here, it is also possible to use extrapolation methods on tent spaces. This amounts to considering the boundedness of singular integral operators of the form
\[
	T: T^{p,2}(M,\mu) \to T^{p,2}(M,\mu), \qquad
	TF(t,\,.\,) = \int_0^{+\infty} K_{\alpha}(t,s) F(s,\,.\,) \,\frac{ds}{s},
\]
with an operator-valued kernel $K_{\alpha}(t,s)$ as defined in  \eqref{eq:K}. 
 We refer to \cite{AKMP} and \cite{FMP} and the references therein for results of this kind. Combining this with the fact that, under \eqref{due},  the Hardy spaces $H^p_\mathcal{L}(M,\mu)$ associated with $\mathcal{L}$ are equal to $L^p(M,\mu)$, for $p \in (1,+\infty)$ (cf. \cite{AMR} for  Riemannian manifolds; the proof extends to our setting,
 see for instance \cite{Chen2}), one obtains the desired results. 
\end{rem}

\begin{example} Let $n\geq 2$. Consider $M:=\R^n \sharp \R^n$ the connected sum of two copies of  $\R^n$, that is the manifold consisting of two copies of $\R^n \setminus B(0,1)$ with the Euclidean metric, glued smoothly along the unit spheres. Then it is known that \eqref{due} is satisfied and the Riesz transform is bounded on $L^p$ for every $p\in(1,n)$ (and unbounded for $p\geq n$), see \cite{CCH}. It follows from Theorem $\ref{thm:summary}$  that $A(\alpha,p)$ holds for $p\in(1,n)$ with $\alpha\in(0,1)$ and for $p>n$ with $\alpha\in(0,\frac{n}{p})$.
\end{example}

\begin{example} Let $(M,d,\mu)$ be a doubling Riemannian manifold supporting the Poincar\'e inequality $(P_2)$, and $\mathcal{L}=\Delta$ its non-negative Laplace Beltrami operator. It is well-known that \eqref{due} holds (see for instance \cite{SA}). Then one knows from \cite{AC} that $(P_2)$ yields $(R_p)$ hence $(G_p)$  for every $p\in(2,2+\varepsilon)$ for some    $\varepsilon>0$,   and from \cite{BCF1} that  $(P_2)$ yields  $({DG}_{2,\kappa})$  for some    $\kappa\in (0,1)$. So we conclude that $A(\alpha,p)$ holds for $p\in(1,2]$ with $\alpha\in(0,1)$ and for $p>2$ with $\alpha\in\left(0,\frac{2}{p}+\eta\right)$, for some $\eta>0$. 
\end{example}

We now state our results concerning the characterization of  the Sobolev space $\dot L^p_\alpha$ in terms of a quadratic functional.

\begin{theorem} \label{thm:summary2} Let $(M,d,\mu, {\mathcal E})$  be a doubling metric  measure   Dirichlet space with a ``carr\'e du champ'' satisfying \eqref{dnu}.
Then
\begin{itemize}
\item[(e)] Under \eqref{due} and $(H^\eta)$, $E(\alpha,p)$ holds  for every $p\in(1,+\infty)$ and $\alpha\in(0,\eta)$;
\item[(f)]  Under the combination $(G_{p_0})$, $(P_{p_0})$ for some  $p_0>2$, $E(\alpha,p)$ holds for every $p\in(2,p_0)$ and $\alpha\in(0,1)$.
\end{itemize}
\end{theorem}

Statement (e) is proven in Theorem \ref{thm:osci} and statement (f) in Theorem \ref{thm:osci2}.

In  statement (f), one does not need to assume explicitely   \eqref{due} but, according to \cite[Proposition 2.1]{BCF1},  the combination $(G_{p_0})+(P_{p_0})$ for  $p_0>2$  does imply \eqref{due}. 

Note that, similarly to the Riesz transform problem (see \cite{CD1,ACDH}), the case $1<p<2$ is substantially easier in the above results than the case $p>2$ .

\begin{example} Let us mention that our results are not bound to self-adjoint setting. Consider ${\mathbb R}^n$, equipped with its Euclidean structure, and a second order divergence form operator $L=-\div(A \nabla)$, where $A \in L^\infty(\R^n; \mathcal{B}(\C^n))$ and for some $\lambda>0$, $\Re(A(x)) \geq \lambda I > 0$ for a.e. $x \in \R^n$.
Then $L$ is a sectorial operator on $L^2(M,\mu)$, and $-L$ generates an analytic semigroup $(e^{-tL})_{t>0}$ on $L^2(M,\mu)$.
It is known (see \cite{A}) that the semigroup $(e^{-tL})_{t>0}$ and its gradient satisfy $L^2$ Davies-Gaffney estimates. From the solution of the Kato square root problem \cite{AHLMcT}, we know that the Riesz transform $\nabla L^{-1/2}$ is bounded on $L^2(M,\mu)$. Let us assume that $(e^{-tL})_{t>0}$ has a (complex-valued) kernel $p_t$ which satisfies Gaussian estimates, that is, $|p_t|$  satisfies \eqref{UE} (which is for example the case if $A$ has real-valued coefficients, see \cite{AT}). Then there exists $q_+=q_+(L) \in(2,\infty]$ such that for every $p\in(1,q_+)$,  $(G_p)$ and, equivalently, $(R_p)$ holds. See \cite{A}. In dimension $n=1$, it is known that $q_+=\infty$. Moreover, for every $p \in (1,+\infty)$, $(RR_p)$ holds.
The kernel $p_t$ satisfies a H\"older regularity estimate (see \cite{AT}), so property $(H^\eta)$ holds for some $\eta\in(0,1]$.

We leave it to the reader to check that, even if the operator $L$ is not self-adjoint, our proofs still hold in this situation.
We deduce that $A(\alpha,p)$ holds (as well as a chain rule property) for every $p\in(1,q_+]$ and $\alpha\in(0,1)$, and for every $p>q_+$ and $\alpha\in(0,1)$ with $0<\alpha < \kappa+\frac{q_+}{p}(1-\kappa)$ and $\kappa=\max(1-\frac{n}{q_+},\eta)$. Moreover if $p\leq q_+$ or $\alpha<\eta$, then $E(\alpha,p)$ holds.
\end{example}

Section \ref{sec:chainrule} is devoted to the proof of a chain rule inequality (which enables one to control the stability of Sobolev spaces via the composition by a nonlinearity). In particular it is proved (see Corollary \ref{cor:chainrule} and Theorem \ref{thm:chainrule}):

\begin{theorem}[Chain rule] \label{thm:cr}  Let  $(M,d,\mu, {\mathcal E})$  be a doubling metric  measure   Dirichlet space with a ``carr\'e du champ''. 
\begin{itemize}
\item Under the assumptions of $(e)$ and $(f)$ in Theorem $\ref{thm:summary2}$, we have the optimal chain rule: for $F$ a Lipschitz function with $F(0)=0$, the map
$ f\rightarrow F(f)$ is bounded in $\dot L^p_\alpha$ and
$$ \|F(f)\|_{\dot L^p_\alpha} \lesssim \|F\|_{\textrm{Lip}} \|f\|_{\dot L^p_\alpha};$$
\item Under the assumptions of Theorem \ref{thm:summary}, for $F$ a $C^2$ function with $F(0)=0$, the map
$ f\rightarrow F(f)$ is bounded in $\dot L^p_\alpha$.
\end{itemize}
\end{theorem}

Similarly, a paralinearisation formula (also called Bony's formula) is also obtained in this setting and we refer the reader to Theorem \ref{thm:paralinearisation} for a precise statement.

\section{Preliminaries, definitions and toolbox}

In this section, $(M,d,\mu, {\mathcal E})$ will  be a doubling metric  measure   Dirichlet space with a ``carr\'e du champ''.

\subsection{Functional calculus}

Since $\mathcal{L}$ is a self-adjoint operator on $L^2(M,\mu)$, it admits a bounded Borel functional calculus on $L^2(M,\mu)$. Under the additional assumption of \eqref{dnu} and \eqref{due}, it is known that $\mathcal{L}$ can be extended to an unbounded operator acting on $L^p(M,\mu)$, for $p \in (1,+\infty)$, with a bounded $H^\infty$ functional calculus on $L^p(M,\mu)$ as shown in \cite[Theorem 3.1]{DR}. It also admits a bounded H\"ormander-type functional calculus on $L^p(M,\mu)$, see  \cite{DR} and \cite[Theorem 3.1]{DOS}.  We refer to \cite{Mc,ADM} and references in \cite{ADM} for more details on functional calculus. In the sequel, we will mostly make use of $H^\infty$ functional calculus rather than H\"ormander-type functional calculus. 

Moreover, gathering Theorem 3.1, Remark 1 p.451 and (1.8) from  \cite{DOS}, one obtains the following estimate on imaginary powers of the operator $\mathcal{L}$ (see also \cite{SW}).

\begin{proposition}\label{prop:imaginary} Under \eqref{dnu} and \eqref{due}, for every $p\in(1,+\infty)$ and $s>\nu/2$, one has
$$  \| \mathcal{L}^{i\beta} \|_{p \to p} \lesssim (1+|\beta|)^{s},$$
 for  $\beta\in\R$.
\end{proposition}

\subsection{Operator estimates}

The building blocks of our analysis will be the following operators derived from the semigroup $(e^{-t\mathcal{L}})_{t>0}$.

\begin{definition} \label{def:Qt-Pt}
Let $N > 0$, and set $c_N=\int_0^{+\infty} s^{N} e^{-s} \,\frac{ds}{s}$.
 For $t>0$, 
define
\begin{equation} \label{def:Qt}
	Q_t^{(N)}:=c_N^{-1}(t\mathcal{L})^{N} e^{-t\mathcal{L}}
\end{equation}
and 
\begin{equation} \label{def:Pt}
	P_t^{(N)}:=\phi_N(t\mathcal{L}),
\end{equation}
with $\phi_N(x):= c_N^{-1}\int_x^{+\infty} s^{N} e^{-s} \,\frac{ds}{s}$,  $x\ge 0$.
\end{definition}

\begin{rem} \label{Pt-rem}
Let $p \in (1,\infty)$ and $N>0$.
\begin{enumerate}
\item[(i)]
As a consequence of the bounded functional calculus for ${\mathcal L}$ in $L^p(M,\mu)$, the operators $P_t^{(N)}$ and $Q_t^{(N)}$ are bounded in $L^p(M,\mu)$, uniformly in $t>0$.
\item[(ii)] 
Note that $P_t^{(1)}=e^{-t\mathcal{L}}$ and $Q_t^{(1)}=t\mathcal{L}e^{-t\mathcal{L}}$. The two families of operators $(P_t^{(N)})_{t>0}$ and $(Q_t^{(N)})_{t>0}$ are related by $$ t\partial_t P_t^{(N)} =  t\mathcal{L}\phi'_N(t\mathcal{L})= - Q_t^{(N)}. $$
Since $P_t^{(N)}f\to f$ as $t\to 0^+$ in $L^p(M,\mu)$ (see the proof of Proposition \ref{prop:reproducing} below), it follows that
\begin{equation}\label{along}
P_t^{(N)} = \textrm{Id}+ \int_0^t Q_s^{(N)} \, \frac{ds}{s}.
\end{equation}
\item[(iii)]
One can write  $P_t^{(N)}= R_t^{(N)} e^{-t/2{\mathcal L}}$, with 
\begin{equation} 
R_t^{(N)}:=c_N^{-1}\int_t^{+\infty} (s{\mathcal L})^{N} e^{-(s-t/2){\mathcal L}} \,\frac{ds}{s}. \label{eq:Ro} 
\end{equation}
By functional calculus, $R_t^{(N)}$ is again a bounded operator in $L^p(M,\mu)$, uniformly in $t>0$.
\item[(iv)]
If $N$ is an integer, then $Q_t^{(N)}=(-1)^{N}c_N^{-1} t^{N} \partial_t^{N} e^{-t\mathcal{L}}$, and $P_t^{(N)}=p(t\mathcal{L})e^{-t\mathcal{L}}$, $p$ being a polynomial of degree $N-1$ with 
$p(0)=1$.
\end{enumerate}
\end{rem}

\begin{definition}
Let $p,q \in [1,\infty]$ with $p\leq q$, and let $r>0$. A linear operator $T$  acting on $L^p(M,\mu)$ is said to have $L^p$-$L^q$ off-diagonal bounds of order $N>0$ at scale $r$, if there exists $C_N > 0$ such that for every pair of balls $B_1,B_2$ of radius $r$ and every $f \in L^p(M,\mu)$ supported in $B_1$, we have
$$ \left(\aver{B_2}|T f|^q\,d\mu \right)^{1/q} \leq C_N \left(1+ \frac{d^2(B_1,B_2)}{r^2}\right)^{-N} \left(\aver{B_1}|f|^p\,d\mu\right)^{1/p}.$$
\end{definition}

Let us recall that we may compose off-diagonal estimates:

\begin{lemma} \label{lem:comp-OD}
Let $p,q,r\in [1,\infty]$ with $p\leq q\leq r$. Let $S$ (resp. $T$) be two linear operators satisfying $L^p$-$L^q$ (resp. $L^q$-$L^r$) off-diagonal estimates of order $N_1>\frac{\nu}{2}$ (resp. $N_2>\frac{\nu}{2}$) at scale $\sqrt{s}$ (resp. $\sqrt{t}$).
If $s=t$, then $TS$ satisfies $L^p$-$L^r$ off-diagonal estimates of order $N:=\min(N_1,N_2)>0$ at scale $\sqrt{s}=\sqrt{t}$.
If $p=q=r$ with $N>\nu$ (and $s\neq t$), then $TS$ satisfies $L^p$-$L^r$ off-diagonal estimates of order $N-\frac{\nu}{2}$ at scale $\max(\sqrt{s},\sqrt{t})$.
\end{lemma}

\begin{proof} 
If $s=t$, consider balls $B_1,B_2$ of radius $\sqrt{s}$ and $(B^j)_j$ a collection of balls of radius $\sqrt{s}$ which covers the whole space and satisfies a bounded overlap property. Then we have for every $f \in L^p$ supported on $B_1$
\begin{align*}
\left(\aver{B_2} |TSf|^{r} \, d\mu \right)^{1/r} & \lesssim \sum_j  \left( 1+ \frac{d^2(B_2,B^j)}{s}\right)^{-N_2} \left(\aver{B_2} |Sf|^{q} \, d\mu \right)^{1/q} \\
& \lesssim \sum_j  \left( 1+ \frac{d^2(B_2,B^j)}{s}\right)^{-N_2} \left( 1+ \frac{d^2(B_1,B^j)}{s}\right)^{-N_1}
 \left(\aver{B_1} |f|^{p} \, d\mu \right)^{1/p} \\
 & \lesssim \left( 1+ \frac{d^2(B_2,B^j)}{s}\right)^{-N}  \left(\aver{B_1} |f|^{p} \, d\mu \right)^{1/p},
\end{align*}
where we used that $N>\nu/2$ to sum over the covering as detailed in \cite[Lemma 3.6]{FK}.

Let us now consider the case $p=q=r$. Consider the case $s\geq t$ (the other one can be treated similarly). We are first going to check that $T$ satisfies $L^p$-$L^p$ off-diagonal estimates at the largest scale $\sqrt{s}$. Since $N_2>\frac{\nu}{2}$, by decomposing the whole space with a bounded covering at scale $\sqrt{s}$, we deduce that $T$ is $L^p$-bounded.  So the on-diagonal case of the off-diagonal estimates for $T$ directly hold. Then fix two balls $B_1,B_2$ of radius $\sqrt{s}$ with $d(B_1,B_2)\geq \sqrt{s}$ and $f\in L^p$ supported on $B_1$. Consider $(B_2^j)_j$  (resp. $(B_1^j)_j$) a bounded covering of $B_2$ (resp. $B_1$) with balls of radius $\sqrt{t}$. Then for every index $j$, we have
\begin{align*}
\left(\aver{B_2^j} |Tf|^{p} \, d\mu \right)^{1/p} & \lesssim \sum_k  \left( 1+ \frac{d^2(B_2^j,B_1^k)}{t}\right)^{-N_2} \left(\aver{B_1^k} |f|^{p} \, d\mu \right)^{1/p} \\
& \lesssim  \left( 1+ \frac{d^2(B_2,B_1)}{t}\right)^{-N_2+\nu/2 }  \left(\frac{s}{t}\right)^{\nu/(2p)} \left(\aver{B_1} |f|^{p} \, d\mu \right)^{1/p},
\end{align*}
where we used the doubling property.
Then by summing over $j$, we get
\begin{align*}
\left(\aver{B_2} |Tf|^{p} \, d\mu \right)^{1/p} & \lesssim \left( |B_2|^{-1} \sum_j |B_2^j| \left(\aver{B_2^j} |Tf|^{p} \, d\mu \right)\right)^{1/p} \\
 & \lesssim  \left( 1+ \frac{d^2(B_2,B_1)}{t}\right)^{-N_2+\nu/2 }  \left(\frac{s}{t}\right)^{\nu/(2p)} \left(\aver{B_1} |f|^{p} \, d\mu \right)^{1/p}.
\end{align*}
Since $d(B_1,B_2)\geq \sqrt{s}\geq \sqrt{t}$ and $N_2\geq \nu$ we have
\begin{align*}
\left(\aver{B_2} |Tf|^{p} \, d\mu \right)^{1/p} & \lesssim  \left( 1+ \frac{d^2(B_2,B_1)}{s}\right)^{-N_2+\nu/2 } \left(\aver{B_1} |f|^{p} \, d\mu \right)^{1/p},
\end{align*}
which concludes the proof of the fact that $T$ admits $L^p$-$L^p$ off-diagonal estimates at the larger scale $\sqrt{s}$.
Then we may repeat the first statement of the Lemma and conclude that $TS$ admits $L^p$-$L^p$ off-diagonal estimates at the scale $\sqrt{s}$.
\end{proof}

\begin{lemma} \label{prop:kernel-est}
Assume  \eqref{due}. Let  $N> 0$. For every $t>0$,  $Q_t^{(N)}$  is an integral operator with kernel $k_t^{(N)}$ such that for all  $t>0$, all $\theta\in[0,1]$ and a.e. $x,y \in M$,
\begin{equation} \label{kernel-est}
	\abs{k_t^{(N)}(x,y)}\lesssim \frac{1}{V(x,\sqrt{t})^\theta V(y,\sqrt{t})^{1-\theta} } \left(1+\frac{d^2(x,y)}{t}\right)^{-N}.
\end{equation}
Consequently, for every $p,q\in[1,{+\infty}]$ with $p\leq q$, $Q_t^{(N)}$ satisfies $L^p$-$L^q$ off-diagonal bounds of order $N$ at scale $\sqrt{t}$.\\
Let $N>\frac{\nu}{2}$. For every $t>0$, $P_t^{(N)}$ is an integral operator with kernel $\tilde k_t^{(N)}$ such that for all $t>0$, all $\theta\in[0,1]$ and a.e. $x,y\in M$,
\begin{equation*} 
	\abs{\tilde k_t^{(N)}(x,y)} \lesssim \frac{1}{V(x,\sqrt{t})^\theta V(y,\sqrt{t})^{1-\theta} } \left(1+\frac{d^2(x,y)}{t}\right)^{-N}.
\end{equation*}
Consequently, for every $p,q\in[1,{+\infty}]$ with $p\leq q$, $P_t^{(N)}$ satisfies $L^p$-$L^q$ off-diagonal bounds of order $N$ at scale $\sqrt{t}$.
\end{lemma}

\begin{rem} \label{rem:kr} 
Let $N>\frac{\nu}{2}$. The operator $R_t^{(N)}$ introduced in Remark  \ref{Pt-rem} is an integral operator as well, with its kernel $r_t^{(N)}$ satisfying \eqref{kernel-est}. Moreover, for all $p \in [1,+\infty]$, $R_t^{(N)}$ has $L^p$-$L^p$ off-diagonal bounds of order $N$. 
\end{rem}

\begin{proof}  
Observe first that by \eqref{dnu}, one has
 for $\theta\in[0,1]$ and every $x,y\in M$ 
\begin{equation} \frac{1}{V(x,\sqrt{t})} e^{-c\frac{d^2(x,y)}{t}} \lesssim \frac{1}{V(x,\sqrt{t})^\theta V(y,\sqrt{t})^{1-\theta}} e^{-\frac{c}{2}\frac{d^2(x,y)}{t}}.\label{eq:gauss} \end{equation}
As we already said, if $N$ is an integer, then $Q_t^{(N)}=(-1)^{N}c_N^{-1} t^{N} \partial_t^{N} e^{-t\mathcal{L}}$. 
By \cite[Corollary 2.7]{ST2}, its kernel admits Gaussian bounds  and therefore in particular \eqref{kernel-est}. In the general case, consider an integer $K>N$. Then 
$$  Q_t^{(N)}  =c_N^{-1}t^{N}\mathcal{L}^K \mathcal{L}^{N-K}e^{-t\mathcal{L}},$$
and by the integral representation
$\mathcal{L}^{N-K}= c\int_0^{+\infty} s^{K-N} e^{-s\mathcal{L}}  \frac{ds}{s}$
for some constant $c>0$, one may write
$$  Q_t^{(N)}  = c'\int_0^{+\infty} (s\mathcal{L})^K e^{-(s+t)\mathcal{L}} \left(\frac{t}{s}\right)^{N} \frac{ds}{s}.$$
Gaussian upper estimates for  $\left(t\mathcal{L}\right)^K e^{-t\mathcal{L}}$ and \eqref{d} then yield a bound of the form \eqref{eq:gauss} for $\left((t+s)\mathcal{L}\right)^K e^{-(s+t)\mathcal{L}}$  at the scale $\max(\sqrt{s},\sqrt{t})$, hence
\begin{align*}
 \abs{k_t^{(N)}(x,y)} & \lesssim \frac{1}{V(x,\sqrt{t})^\theta V(y,\sqrt{t})^{1-\theta}} \left[\int_0^t  \left(\frac{s}{t}\right)^K  e^{-\frac{c}{2}\frac{d^2(x,y)}{t}} \left(\frac{t}{s}\right)^{N} \,\frac{ds}{s} + \int_t^{+\infty}   e^{-\frac{c}{2}\frac{d^2(x,y)}{s}} \left(\frac{t}{s}\right)^{N} \,\frac{ds}{s} \right] \\
& \lesssim \frac{1}{V(x,\sqrt{t})^\theta V(y,\sqrt{t})^{1-\theta}} \left[e^{-\frac{c}{2}\frac{d^2(x,y)}{t}} \int_0^t  \left(\frac{s}{t}\right)^{K-N} \, \frac{ds}{s}+  \int_t^{+\infty}    e^{-\frac{c}{2}\frac{d^2(x,y)}{s}} \left(\frac{t}{s}\right)^{N} \,\frac{ds}{s}\right].
\end{align*}
Thus
\begin{align*}
 \abs{k_t^{(N)}(x,y)} & \lesssim \frac{1}{V(x,\sqrt{t})^\theta V(y,\sqrt{t})^{1-\theta}} \left[e^{-\frac{c}{2}\frac{d^2(x,y)}{t}} + \left(1+\frac{d^2(x,y)}{t}\right)^{-N}\right],
\end{align*}
which concludes the proof of \eqref{kernel-est} for $k_t^{(N)}$. Integrating over the bound in \eqref{kernel-est} then gives the second claim for $Q_t^{(N)}$. 

In order to obtain the assertions on $P_t^{(N)}$, we use Remark \ref{Pt-rem} (iii), which yields $P_t^{(N)}=e^{-t{\mathcal L}/4}R_t^{(N)} e^{-t{\mathcal L}/4}$ and so for every $x,y\in M$
\begin{align*}
\abs{\tilde k_t^{(N)}(x,y)} & \lesssim \int \left| p_{t/4}(x,z) \right| \left| R_t^{(N)}[p_{t/4}(y,\cdot)](z) \right| \, d\mu(z) \\
 & \lesssim \left(\int \left| p_{t/4}(x,z) \right|^2 \, d\mu(z)\right)^ {1/2} \left( \int \left| R_t^{(N)}[p_{t/4}(y,\cdot)](z) \right|^2 \, d\mu(z) \right)^{1/2} \\
 & \lesssim V(x, \sqrt{t})^{-1/2} \|R_t^{(N)}\|_{2\to 2}   V(y, \sqrt{t})^{-1/2},
\end{align*} 
where we used \eqref{due} to estimate the $L^2$ norm of the heat semigroup. Consequently, since $R_t^{(N)}$ is bounded in $L^2(M,\mu)$ uniformly in $t>0$, we obtain that
$$ \abs{\tilde k_t^{(N)}(x,y)} \lesssim V(x, \sqrt{t})^{-1/2}  V(y, \sqrt{t})^{-1/2}.$$
For the diagonal part, when $d(x,y) \lesssim \sqrt{t}$,  we have by doubling  $V(x, \sqrt{t}) \simeq  V(y, \sqrt{t})$ and so the previous estimate implies the desired inequality.

For the off-diagonl part, when $d(x,y)\geq \sqrt{t}$, we use the representation \eqref{along} and integrate the previous estimate on $k_t^{(N)}$ (the kernel of $Q_t^{(N)}$) in time. This gives
\begin{align*}
\abs{\tilde k_t^{(N)}(x,y)} & \leq \int_0^t \abs{k_s^{(N)}(x,y)} \, \frac{ds}{s} \\
 & \lesssim \int_0^t \frac{1}{V(x,\sqrt{s})^\theta V(y,\sqrt{s})^{1-\theta}} \left(1+\frac{d^2(x,y)}{s}\right)^{-N} \, \frac{ds}{s} \\
 & \lesssim \frac{1}{V(x,\sqrt{t})^\theta V(y,\sqrt{t})^{1-\theta}} \int_0^t  \left(\frac{t}{s}\right)^{\nu/2} \left(1+\frac{d^2(x,y)}{s}\right)^{-N} \, \frac{ds}{s} \\
 & \lesssim \frac{1}{V(x,\sqrt{t})^\theta V(y,\sqrt{t})^{1-\theta}}  \left(1+\frac{d^2(x,y)}{t}\right)^{-N}, 
\end{align*}
where we have used \eqref{d} and $N>\nu/2$. 

The second statement for $P_t^{(N)}$ follows by combining the previous estimate with the global $L^p$ boundedness of $P_t^{(N)}$.

\end{proof}

\begin{proposition}[Davies-Gaffney estimates] \label{prop:Davies-Gaffney} Let $N \in \N$. There exists a constant $c>0$ such that for all Borel sets $E,F \subset M$ and every $t>0$
\begin{align*}\label{DG1}
 \| P_t^{(N)}\|_{L^2(E) \to L^2(F)} + \|\sqrt{t} |\nabla P_t^{(N)}| \|_{L^2(E) \to L^2(F)} \lesssim  e^{- c\frac{d^2(E,F)}{t}},\\
  \| Q_t^{(N)}\|_{L^2(E) \to L^2(F)} + \|\sqrt{t} |\nabla Q_t^{(N)}| \|_{L^2(E) \to L^2(F)} \lesssim  e^{- c\frac{d^2(E,F)}{t}}.
 \end{align*}
If $N>\nu/2$ is not an integer, then for all balls $B_1,B_2$ of radius $\sqrt{t}$
\begin{equation*}\label{DG2} \|\sqrt{t} |\nabla P_t^{(N)}| \|_{L^2(B_1) \to L^2(B_2)} +  \|\sqrt{t} |\nabla Q_t^{(N)}| \|_{L^2(B_1) \to L^2(B_2)} \lesssim  \left(1+ \frac{d^2(B_1,B_2)}{t}\right)^{-N}.\end{equation*}
\end{proposition}

\begin{proof} The  first estimate  is classical for $P_t^{(1)}=e^{-t\mathcal{L}}$ (see for instance \cite{ST2}, except  for the term with the gradient, which was introduced in
\cite[Section 3.1]{ACDH} in the Riemannian setting. For an adaptation to the present setting, see \cite[Section 2]{BCF1}). The generalisation to $P_t^{(N)}$ and $Q_t^{(N)}$ with arbitrary $N \in \N^*$ is a consequence of the analyticity of $(e^{-t\mathcal{L}})_{t>0}$ in $L^2(M,\mu)$, and the particular form of $P_t^{(N)}$, see Remark \ref{Pt-rem}. 
Now for the second estimate. Lemma \ref{prop:kernel-est} yields that $P_t^{(N)}$ and $Q_t^{(N)}$ satisfy $L^2$-$L^2$ off-diagonal estimates of order $N$. 
Since $\sqrt{t} \nabla Q_t^{(N)} = 2^N\sqrt{t} \nabla e^{-t/2 \mathcal{L}} Q_{t/2}^{(N)}$, and $\sqrt{t}\nabla e^{-t\mathcal{L}}$ satisfies $L^2$-$L^2$ off-diagonal estimates of any order, we may compose these off-diagonal estimates and Lemma \ref{lem:comp-OD} implies the desired result for $\nabla Q_t^{(N)}$. For $\nabla P_t^{(N)}$, we use the representation $\sqrt{t}\nabla P_t^{(N)}=\sqrt{t}\nabla e^{-t/2\mathcal{L}} R_t^{(N)}$ of Remark \ref{Pt-rem}, together with Remark \ref{rem:kr}.
\end{proof}

\begin{lemma}[Off-diagonal estimates] \label{lem:off} 
Assume  \eqref{due}. Let $N\geq 1$ be an integer and consider the operators $P_t^{(N)}$, $Q_t^{(N)}$ as defined in \eqref{def:Qt} and \eqref{def:Pt}. For every $t>0$, every ball $B$ of radius $r$ and every $p\in[1,{+\infty}]$, we have
\begin{itemize}
\item if $r\leq \sqrt{t}$ with $\tilde B:=\frac{\sqrt{t}}{r}B$ the dilated ball,
$$ \left( \aver{B} |P_t^{(N)}f|^p + |Q_t^{(N)}f|^p \, d\mu\right)^{1/ p} \lesssim  \sum_{\ell \geq 0} \gamma(\ell) \aver{2^\ell \tilde B} |f| \, d\mu,$$
%where $\gamma(\ell)$ are exponentially decreasing coefficients
\item if $r\geq \sqrt{t}$,
$$ \left( \aver{B} |P_t^{(N)}f|^p + |Q_t^{(N)}f|^p \, d\mu\right)^{1/ p} \lesssim  \sum_{\ell \geq 0} \gamma(\ell) \left(\aver{2^\ell B} |f|^p \, d\mu\right)^{1/ p},$$
\item more generally, if $r\geq \sqrt{t}$ with $p_0,p_1\in[1,{+\infty}]$ satisfying $p_1\geq p_0$ then
$$ \left( \aver{B} |P_t^{(N)}f|^{p_1} + |Q_t^{(N)}f|^{p_1} \, d\mu\right)^{1/ p_1} \lesssim  \left(\frac{r}{\sqrt{t}}\right)^{\nu(\frac{1}{p_0}-\frac{1}{p_1})} \sum_{\ell \geq 0} \gamma(\ell) \left(\aver{2^\ell B} |f|^{p_0} \, d\mu\right)^{1/ p_0},$$
\end{itemize}
where $\gamma(\ell)$ are exponentially decreasing coefficients.\\
For $N>0$ not an integer, $p \in [1,\infty]$, $t>0$ and $B$ a ball of radius $\sqrt{t}$, we have
\begin{align*}
	\|Q_t^{(N)}f\|_{L^\infty(B)} \lesssim \sum_{\ell \geq 0} 2^{-2\ell(N-\frac{\nu}{2})} \left(\aver{2^\ell B} |f|^{p} \, d\mu\right)^{1/ p}.
\end{align*}
\end{lemma}

\begin{proof} For the first part, we use (since $B\subset \tilde B$)
 \begin{align*} 
\left( \aver{B} |P_t^{(N)}f|^p + |Q_t^{(N)}f|^p \, d\mu\right)^{1/ p}  & \leq \|P_t^{(N)}f\|_{L^\infty(B)} + \|Q_t^{(N)}f\|_{L^\infty(B)} \\
 & \leq  \|P_t^{(N)}f\|_{L^\infty(\tilde B)} + \|Q_t^{(N)}f\|_{L^\infty(\tilde B)}
 \end{align*}
and then the proof follows from the pointwise Gaussian estimates of the kernel for both operators $P_t^{(N)}$ and $Q_t^{(N)}$, see \cite[Corollary 2.7]{ST2}). \\
For the second part, the ball $B$ may be covered by a collection of balls of radius $\sqrt{t}$, with a bounded overlap property. Then by using the $L^p$ off-diagonal estimates at the scale $\sqrt{t}$ for operators $P_t^{(N)}$ and $Q_t^{(N)}$, we obtain the stated inequality by summing over this covering.
The third part can be proved by interpolating between the second part and the $L^1$-$L^\infty$ estimates (which corresponds to the case $p_0=1$ and $p_1=\infty$) which comes from \eqref{due} with doubling.\\
The last statement is a consequence of the kernel estimates for $Q_t^{(N)}$ shown in Lemma \ref{prop:kernel-est}. 
\end{proof}

\subsection{Quadratic functionals}

Combining Corollary 1 with Lemma 2 from \cite{ST1} yields the following statement, which does not even require \eqref{d}.

\begin{proposition} \label{prop:kernel} For every $p\in(1,{+\infty})$, consider a function $f\in L^p(M,\mu) \cap {\mathcal D}$ solution of $\mathcal{L}f=0$ on $M$. We have
\begin{itemize}
\item if $|M|={+\infty}$ then $f=0$;
\item if  $|M|<{+\infty}$ then $f$ is constant.
\end{itemize}
In other words, if we denote
$N_p(\mathcal{L}):=\left\{f\in L^p \cap {\mathcal D},\ \mathcal{L}f=0\right\}$,
then $N_p(\mathcal{L})=\{0\}$ or $N_p(\mathcal{L})\simeq \R$ and in particular, it does not depend on $p$ and so will be sometimes denoted $N(\mathcal{L})$.

\end{proposition}

Note that, under \eqref{d}, $|M|<{+\infty}$ if and only if $M$ is bounded.

\begin{proposition}[Calder\'on reproducing formula] \label{prop:reproducing} Let $p \in (1,{+\infty})$.  Let $N>0$,  and consider the operators $P_t^{(N)}$, $Q_t^{(N)}$ as defined in \eqref{def:Qt} and \eqref{def:Pt}. Under \eqref{d} and \eqref{due}, we have the decomposition $L^p(M,\mu) = \overline{\Ran_p(\mathcal{L})} \oplus N_p(\mathcal{L})$. Moreover, for every $f\in L^p(M,\mu)$,
\begin{align} \label{limzero}
\lim_{t\to {0^+}} P_t^{(N)}f & =  f \quad & \textrm{in $L^p(M,\mu)$},\\ \label{liminfty}
\lim_{t\to {+\infty}} P_t^{(N)}f & =  \mathsf{P}_{N_p(\mathcal{L})}f \quad & \textrm{in $L^p(M,\mu)$},
\end{align}
and
for every $f\in\overline{\Ran_p(\mathcal{L})}$,
\begin{equation} \label{calde}
 f = \int_0^{+\infty} Q_t^{(N)}f \frac{ d t}{t} \quad \text{in}\ L^p(M,\mu).
\end{equation}
For every $f\in\overline{\Ran_2(\mathcal{L})}$, one has
\begin{equation}\label{ortho}
 \|f\|_{2}^2 \simeq \int_0^{+\infty} \| Q_t^{(N)} f\|_2^2\ \frac{ d t}{t}.
\end{equation}
\end{proposition}

\begin{proof}
Under \eqref{d} and \eqref{due},  $\mathcal{L}$ has a bounded $H^\infty$ functional calculus in $L^p(M,\mu)$ according to \cite[Theorem 3.1]{DR}. 
Since this in particular implies sectoriality of  $\mathcal{L}$ in $L^p(M,\mu)$, \cite[Theorem 3.8]{CDMY} yields the decomposition of $L^p(M,\mu)$ into nullspace and range of $\mathcal{L}$. 
Using this decomposition, and noting that $P_t^{(N)}f=f$ for every $f \in N_p(\mathcal{L})$ and all $t>0$, 
 the Convergence Lemma (see e.g. \cite[Theorem D]{ADM} or \cite[Lemma 9.13]{KW}) implies for every $f\in L^p(M,\mu)$
\begin{align*}
f  & = \lim_{t\to 0} P_t^{(N)}f = \lim_{t\to 0} P_t^{(N)}f - \lim_{t\to \infty} P_t^{(N)}f + \mathsf{P}_{N(\mathcal{L})}f \\
 & = \int_0^{+\infty} Q_t^{(N)}f \, \frac{dt}{t} + \mathsf{P}_{N(\mathcal{L})}f,
\end{align*}
where the limit is taken in $L^p(M,\mu)$. The last equivalence then follows from the self-adjointness of $Q_t^{(N)}$ and Fubini, as for $f\in\overline{\Ran_2(\mathcal{L})}$
\begin{align*}
	\int_0^{+\infty} \| Q_t^{(N)} f\|_2^2 \, \frac{dt}{t}
	=  \langle \int_0^{+\infty}(Q_t^{(N)})^2 f\,\frac{dt}{t},f \rangle \simeq \|f\|_{2}^2.
\end{align*}
\end{proof}

\begin{definition} \label{def:calS} For $p\in(1,{+\infty})$, we define the set of test functions 
\begin{align*}
\calS^p & =\calS^p(M,\mathcal{L}):= {\mathcal D}_p(\mathcal{L}) \cap \Ran_p(\mathcal{L}) \\
 & = \{f\in L^p,\ \exists g,h\in L^p ,\ f=\mathcal{L}g \textrm{ and } h=\mathcal{L}f \},
\end{align*}
and
$$ \calS = \cup_{p\in(1,{+\infty})} \calS^p.$$
\end{definition}
For every $p\in(1,{+\infty})$ and $\alpha\in(0,1)$, under \eqref{due} the set $(\calS^p + N(\mathcal{L})) \subset  \dot L^p_\alpha$ is dense into $ \dot L^p_\alpha$, due to the previous Calder\'on reproducing formula (see also \cite[Theorem 15.8]{KW}). Indeed, for $f\in \dot L^p_\alpha$, Proposition \ref{prop:reproducing} yields that for $N\geq 1>\alpha$
$$ f_\epsilon := \int_{\epsilon}^{\epsilon^{-1}} Q_t^{(N)}f \, \frac{dt}{t} + \mathsf{P}_{N(\mathcal{L})}f $$
is convergent to $f$ in $\dot L^p_\alpha$ and for every $\epsilon>0$, we easily see that $f_\epsilon\in \calS^p + N(\mathcal{L})$.\\

We state some results on square functions that we will need in the following.

\begin{proposition} \label{prop:square-function}
Let $N>0$,  and consider the operators $P_t^{(N)}$, $Q_t^{(N)}$ as defined in \eqref{def:Qt} and \eqref{def:Pt}. Assume  \eqref{due}. \\
(i) Let $p\in(1,{+\infty})$, and let $\alpha>0$. The horizontal square functions, defined by 
$$g_N(f):= \left( \int_0^{+\infty} \left| Q_t^{(N)}f \right|^2 \,\frac{dt}{t}\right)^{{1/2}},  \qquad f \in L^p(M,\mu), $$
and
$$\tilde{g}_{N,\alpha}(f):= \left( \int_0^{+\infty} \left| (t\mathcal{L})^\alpha P_t^{(N)}f \right|^2 \,\frac{dt}{t}\right)^{{1/2}},  \qquad f \in L^p(M,\mu), $$
are bounded on $L^p(M,\mu)$.\\
(ii)  Let $p \in (1,2]$. 
The vertical square functions, defined by 
\begin{equation}	G_N f:=\left(\int_0^{+\infty} \abs{\sqrt{t}\nabla P_t^{(N)} f}^2 \,\frac{dt}{t}\right)^{1/2}, \qquad f \in L^p(M,\mu), \label{eq:GN} 
\end{equation}
and 
\begin{equation}	\tilde{G}_N f:=\left(\int_0^{+\infty} \abs{\sqrt{t}\nabla Q_t^{(N)} f}^2 \,\frac{dt}{t}\right)^{1/2}, \qquad f \in L^p(M,\mu), \label{eq:GN2} 
\end{equation}
are bounded on $L^p(M,\mu)$. \\
(iii)  Assume in addition $(G_{p_0})$ and $(P_{p_0})$ for some $p_0\in(2,{+\infty})$. Then $G_N$ is bounded on $L^p(M,\mu)$ for every $p \in (1,p_0]$. \\
(iv)  Let $p \in (1,{+\infty})$. The conical square function, defined by 
$$\mathcal{G}_N f(x):= \left( \int_{\Gamma(x)} \left| Q_t^{(N)}f(y) \right|^2 \, \frac{dt d\mu(y)}{tV(y,\sqrt{t})}\right)^{{1/2}}, \qquad f \in L^p(M,\mu),$$
is bounded on $L^p(M,\mu)$.
Here, $\Gamma(x)$ denotes the parabolic cone
$$ \Gamma(x):=\{(y,t)\in M\times (0,{+\infty}),\ d(x,y) \leq \sqrt{t} \}.$$
\end{proposition}

\begin{proof}
For the result on the horizontal square function $g_N$, see \cite{meda} and references therein. The result on $\tilde{g}_{N,\alpha}$ with $N$ an integer also follows from \cite{meda}. For arbitrary $N>0$, see e.g. \cite[Theorem 6.6]{CDMY}.
 
The result on vertical square functions in $L^2(M,\mu)$ is a consequence of integration by parts and \eqref{ortho}. 
For $p\neq 2$, we refer to \cite[Theorem 3.6]{BF2}, where indeed the combination $(G_{p_0})$ and $(P_{p_0})$ is shown to imply the boundedness of the Riesz transform in $L^p$ for every $p\in(1,p_0]$ (which is stronger than the boundedness of $G_N$).

For results on conical square functions of this kind, we refer to \cite[Lemma 5.2, Theorem 8.5]{AMR} for the case $p \in (1,2]$. In the present paper we only use the case $p \in [2,{+\infty})$ which is easier and can be proven as in \cite[Section 3.2]{AHM}, that is, by using Lemma  \ref{lem:nontang-Carleson} below and interpolating with $L^2$, where one can reduce the problem to the horizontal one.
\end{proof}

In fact, the Poincar\'e inequality $(P_{p_0})$ is not necessary in (iii) if one allows a loss on the Lebesgue exponent.

\begin{proposition} \label{prop:K} Let $N>0$,  and consider the operators $P_t^{(N)}$, $Q_t^{(N)}$ as defined in \eqref{def:Qt} and \eqref{def:Pt}. Assume  \eqref{due} and $(G_{p_0})$ for some $p_0\in(2,{+\infty}]$. Then for every $p\in(2,p_0)$ and every $f\in L^p(M,\mu)$,
$$ \| G_N f\|_p \lesssim  \| \tilde{G}_N f\|_p
%\left\| \left(\int_0^{+\infty} \abs{\sqrt{t}\nabla Q_t^{(N)} f}^2 \,\frac{dt}{t}\right)^{1/2} \right\|_p 
\lesssim \|f\|_p.$$
\end{proposition}

\begin{proof} By writing
$$ P_t^{(N)}f = \int_t^{+\infty} Q_s^{(N)}f \, \frac{ds}{s} + \mathsf{P}_{N_p(\mathcal{L})}f,$$
one obtains
$$  \abs{\sqrt{t}\nabla P_t^{(N)} f} \leq \int_t^{+\infty} \left(\frac{t}{s}\right)^{1/2} \abs{\sqrt{s}\nabla Q_s^{(N)} f} \, \frac{ds}{s}.$$
Then Hardy's inequality implies the pointwise inequality
$$ G_N f \lesssim \left(\int_0^{+\infty} \abs{\sqrt{t}\nabla Q_t^{(N)} f}^2 \,\frac{dt}{t}\right)^{1/2},$$
which gives the first desired estimate.

Interpolating $(G_{p_0})$ with the $L^2$ Davies-Gaffney estimates stated in Proposition \ref{prop:Davies-Gaffney} yields, for $p\in(2,p_0)$, that there exists constants such that for every $t>0$ and every pair of balls $B_1,B_2$ of radius $\sqrt{t}$,
$$	 \| |\nabla e^{-t\mathcal{L}} | \|_{L^p(B_1) \to L^p(B_2)}  \lesssim e^{-c \frac{d^2(B_1,B_2)}{t}}. $$
By combining this with \eqref{due}, which self-improves in \eqref{UE}, we deduce that
$$	 \| |\nabla e^{-t\mathcal{L}} | \|_{L^1(B_1) \to L^p(B_2)}  \lesssim |B_1|^{\frac{1}{p}-1}e^{-c \frac{d^2(B_1,B_2)}{t}}. $$

In particular, from \cite[Theorem 2.2]{K} we deduce that the family $(\sqrt{t}\nabla e^{-t\mathcal{L}})_{t>0}$ is $R_2$-bounded in $L^p$, for every $p\in(2,p_0)$.
Since $Q_t^{(N)}= 2^N e^{-t\mathcal{L}/2} Q_{t/2}^{(N)}$, and using the $L^p$ boundedness of the horizontal square function $g_N$, this yields 
\begin{align*}
\left\| \left(\int_0^{+\infty} \abs{\sqrt{t}\nabla Q_t^{(N)} f}^2 \,\frac{dt}{t}\right)^{1/2} \right\|_p & \lesssim \left\| \left(\int_0^{+\infty} \abs{\sqrt{t}\nabla e^{-t\mathcal{L}/2} Q_{t/2}^{(N)} f}^2 \,\frac{dt}{t}\right)^{1/2} \right\|_p \\
& \lesssim \left\| \left(\int_0^{+\infty} \abs{Q_{t/2}^{(N)} f}^2 \,\frac{dt}{t}\right)^{1/2} \right\|_p \\
& \lesssim \|f\|_p,
\end{align*}
which concludes the proof.
\end{proof}

We shall also need the following orthogonality lemma, for instance in the proof of Lemma  \ref{lemme}.

\begin{lemma}\label{lem:orthogonality} Let $N>0$. Consider $Q_t^{(N)}$  and $\tilde Q_t:= (t\mathcal{L})^{N / 2} e^{-\frac{t}{2}\mathcal{L}}$ so that $Q_t^{(N)}=\tilde Q_t^2$. Assume  \eqref{due}. 
Then for every $p\in(1,{+\infty})$ one has
$$ \left\| \int_0^{+\infty} Q_t^{(N)}F_t \, \frac{dt}{t} \right\|_{p} \lesssim \left\| \left( \int_0^{+\infty} |\tilde Q_tF_t|^2 \, \frac{dt}{t}\right)^{1/ 2} \right\|_{p},$$
where $F_t(x);=F(t,x)$,  $F: (0,{+\infty})\times M\to\R$ being a measurable function such that the RHS has a meaning and is finite.
\end{lemma}

\begin{proof} Let $g\in L^{p'}(M,\mu)$. Then, by Fubini, Cauchy-Schwarz and H\"older,
\begin{align*}
\left|  \langle \int_0^{+\infty} Q_t^{(N)}F_t \, \frac{dt}{t},g\rangle \right| & = \left| \int_0^{+\infty} \langle \tilde Q_tF_t , \tilde Q_tg\rangle \,\frac{dt}{t} \right| \\
 & \leq \left\| \left( \int_0^{+\infty} |\tilde Q_tF_t|^2 \, \frac{dt}{t}\right)^{1/ 2} \right\|_{p} \left\| \left( \int_0^{+\infty} |\tilde Q_tg|^2 \, \frac{dt}{t}\right)^{1/ 2} \right\|_{{p'}} \\
 & \lesssim  \left\| \left( \int_0^{+\infty} |\tilde Q_tF_t|^2 \, \frac{dt}{t}\right)^{1/ 2} \right\|_{p} \|g\|_{{p'}},
 \end{align*}
 where in the last inequality we have used the fact that $\tilde Q_t=2^{N/2} Q_{t/2}^{(N/2)}$ and the second assertion in Proposition \ref{prop:square-function}.
\end{proof}

We will also need the Fefferman-Stein inequalities for the Hardy-Littlewood maximal operator (see \cite{feffstein} for the discrete version and \cite[Proposition 4.5.11]{Grafakos1} for the transfer method from discrete to continuous versions):
\begin{proposition} \label{prop:FS} Let $1<p<+\infty$ and $1\le q<\min (p,2)$. Then the $L^q$-Hardy-Littlewood maximal function satisfies the following discrete $L^2$-valued inequalities
$$ \left\| \left( \sum_{n \in \Z} |\calM[|f_n|^q]|^{2\over q} \right)^{1/2} \right\|_p \lesssim \left\| \left( \sum_{n \in \Z} | f_n |^2 \right)^{1/2} \right\|_p,$$
for $(f_n)_n \in L^p(M, \ell^2(\Z))$,
and the continuous version
$$ \left\| \left( \int_0^{+\infty} |\calM [|F_t|^q] |^{2q} \, \frac{dt}{t} \right)^{1/2} \right\|_p \lesssim \left\| \left( \int_0^{+\infty} | F_t |^2 \, \frac{dt}{t} \right)^{1/2} \right\|_p,$$
 for $(F_t)_t \in L^p\left(M, L^2[(0,{+\infty}); \frac{dt}{t}]\right)$.
\end{proposition}

\subsection{Carleson duality}

For every $x \in M$, denote by $\Gamma(x)$ the parabolic cone of aperture $1$ with vertex $x$, i.e.
\[
		\Gamma(x):=\{(y,t) \in M \times (0,{+\infty}) \,:\, d(y,x) < \sqrt{t}\}.
\]
For every measurable function $F$ on $M \times (0,{+\infty})$ and an exponent $p\in(1,{+\infty})$, 
the $L^p$-Carleson function $\scrC_p (F)$ is defined by
\[
 	\scrC_p (F)(x):= \sup_{B \ni x} \left(\aver{B} \left( \int_0^{r(B)} \abs{F(y,t)}^2 \, \frac{dt}{t}\right)^{p/2}\, d\mu(y) \right)^{1/p}, \qquad x \in M,
\]
where the supremum is taken over  all balls $B$ in $M$ that contain $x$. 
Let us point out that the case $p=2$ corresponds to the classical maximal function over Carleson boxes. 
For every measurable function $F: M \times (0,\infty) \to \C$, we denote by $N_\ast(F)$ its non-tangential maximal function, which is defined as
$$N_\ast(F)(x):= \sup_{(y,t)\in \Gamma(x)} \abs{F(y,t)}, \qquad x \in M.$$

We will need the following Carleson duality (see \cite{CMS} for the original proof in the Euclidean setting and $p=2$).

\begin{theorem} \label{CarlesonDuality} Let $(M,d,\mu)$  be a doubling metric  measure  space.
Suppose $p\in[2,{+\infty})$. For every $\epsilon>0$ (with $\epsilon=0$ if $p=2$), there exists a constant $C>0$ such that for all measurable functions $F, G: M \times (0,\infty) \to \C$,
\begin{align*}
	\left(\int_M \left(\int_0^{+\infty} \abs{F(x,t)}^2 \abs{G(x,t)}^2 \,\frac{dt}{t}\right)^{p/2}\, d\mu(x) \right)^{1/p}
	&\leq C \norm{N_\ast (F)}_{p}\norm{\scrC_{p+\epsilon} (G)}_{\infty}.
\end{align*}
\end{theorem}

The original proof  for $p\neq 2$ was developed in a Banach space valued setting in \cite[Section 8]{HMP}, see also \cite{KH}. We give a proof in the scalar-valued setting.

\begin{proof} We first recall the existence of a dyadic system, see \cite{Christ}: there exists a family of points $(x^k_\alpha)_{\alpha\in I(k)} \subset M$ with the property
\begin{equation*}
  M=\bigcup_{\alpha\in I(k)} B(x^k_\alpha,2^k) \quad\text{(bounded overlap)}\quad\forall k\in\Z.
\end{equation*}
For every $x\in M$ and $k\in \Z$, we define $\Delta_k(x)$ the set of indices $\alpha$ such that $x\in B(x^k_\alpha,2^k)$.
Without loss of generality, assume that $N_\ast(F)<{+\infty}$ almost everywhere. Denote for $k \in \N$ and almost every $x\in M$
$$ \tau_k(x):=\sup \{t\,:\, \sup_{\ell: 2^ \ell \leq t^{1/2}} \ \sup_{\genfrac{}{}{0pt}{}{y \in B(x^\ell_\alpha,2^\ell)}{ \alpha\in \Delta_\ell(x)}} \ \sup_{2^ {2\ell}\leq s \leq 2^ {2(\ell+1)}} |F(y,s)| > 2^k\},
$$
and set
$$ J_k(x)=[\tau_k(x),\tau_{k+1}(x)).$$
Since for almost every $x$, $\tau_k(x)$ tends to $0$ for $k\to -\infty$, and $\tau_k(x)$ tends to ${+\infty}$ for $k\to {+\infty}$, we deduce that
$$ (0,{+\infty})=\bigcup J_k(x).$$ 
We therefore have for almost every $x \in M$, 
\begin{align*}
	&\left(\int_0^{+\infty} |F(x,t)|^2 |G(x,t)|^2 \,\frac{dt}{t}\right)^{1/2}
	\leq \left(\sum_{k=-\infty}^{+\infty} \int_{J_k(x)}|F(x,t)|^2 |G(x,t)|^2 \,\frac{dt}{t}\right)^{1/2}\\
	& \qquad \lesssim \left(\sum_{k=-\infty}^{+\infty} 2^{2(k+1)} \int_{J_k(x)} |G(x,t)|^2 \,\frac{dt}{t}\right)^{1/2}.
\end{align*}
For fixed $k \in \Z$, define
\begin{align*}
	A_l:=\{x \in M\,;\, 2^{2l} \leq \tau_k(x) <2^{2(l+1)}\}.
\end{align*}
Then $M=\bigcup_{l=-\infty}^{+\infty} A_l$, and
\begin{align*}
	&\left\|\left(\int_{J_k(x)} |G(x,t)|^2 \,\frac{dt}{t}\right)^{1/2}\right\|_p
	 \leq \left(\sum_l \int_{A_l} \left(\int_0^{2^{2(l+1)}} |G(x,t)|^2\,\frac{dt}{t}\right)^{p/2} \,d\mu(x)\right)^{1/p}.
\end{align*}
By definition of $A_l$, it is clear that if $x\in A_l$ then for $\alpha\in \Delta_l(x)$, $B(x^l_\alpha,2^l)$ is also included into $A_l$, which means that $A_l$ can be covered by a union of balls $(B(x^l_\alpha,2^l))_{\alpha\subset U}$ for  a subset $U \subset I(l)$, with a finite overlap.
We thus have
\begin{align*}
\int_{A_l} \left(\int_0^{2^{2(l+1)}} |G(x,t)|^2\,\frac{dt}{t}\right)^{p/2} \,d\mu(x) 
%\leq \sum_{\alpha\in U} \int_{Q^l_\alpha} \left(\int_0^{2^{2(l+1)}} |G(x,t)|^2\,\frac{dt}{t}\right)^{p/2} \,d\mu(x) \\
&  \leq \sum_{\alpha\in U} \int_{B(x^l_\alpha,2^l)} \left(\int_0^{2^{2(l+1)}} |G(x,t)|^2\,\frac{dt}{t}\right)^{p/2} \,d\mu(x) \\
& \lesssim \norm{\scrC_{p} (G)}_{\infty}^p \sum_{\alpha\in U} |B(x^l_\alpha)| \\
& \lesssim \norm{\scrC_{p} (G)}_{\infty}^p |A_l|.
\end{align*}
As a consequence, we deduce that
\begin{align*}
	\left\|\left(\int_{J_k(x)} |G(x,t)|^2 \,\frac{dt}{t}\right)^{1/2}\right\|_p & 
	 \lesssim \norm{\scrC_{p} (G)}_{\infty} \left(\sum_l |A_l| \right)^{1/p} \\
	 & \lesssim \norm{\scrC_{p} (G)}_{\infty} |\{x\in M\,:\, \tau_k(x)<{+\infty}\}|^{1/p} \\
	 	 & \lesssim \norm{\scrC_{p} (G)}_{\infty} |\{x\in M\,:\, N_\ast(F)(x) >2^k\}|^{1/p}.
\end{align*}
Using the assumption $p\geq 2$, we conclude that
\begin{align*}
	& \left\|\left(\int_0^{+\infty} |F(x,t)|^2 |G(x,t)|^2 \,\frac{dt}{t}\right)^{1/2}\right\|_p\\
	& \qquad \lesssim \norm{\scrC_{p} (G)}_{\infty} \left(\sum_{k=-\infty}^{+\infty} 2^{2k} |\{x\in M,\ N_\ast(F) >2^k\}|^{1/p} \right)^{1/2} \\
	& \qquad \lesssim \norm{\scrC_{p} (G)}_{\infty}  \|N_\ast(F)\|_{L^{p,2}},
\end{align*}
where $L^{p,2}(M,\mu)$ is the classical Lorentz space.
Now since there is a small interval $(p-\epsilon,p+\epsilon)$ in which we can apply this previous inequality: indeed for every $q\in(p-\epsilon,p+\epsilon)$
$$ \norm{\scrC_{q} (G)}_{\infty}  \leq \norm{\scrC_{p+\epsilon} (G)}_{\infty} ,$$
we then conclude by real interpolation.
\end{proof}

\section{Paraproducts}\label{para}

We define paraproducts associated with the operator $\mathcal{L}$. Some versions of such paraproducts have already been introduced and studied in \cite{B, F, BS, BF}.  We are going to use here a slightly different version that is more adapted to our purpose.

From now on, let  $\PP$ be a large enough integer ($\PP\geq 4(1+\nu)$ for example should be sufficient for this section,  where $\nu$ is as in \eqref{dnu}; the choice of $\PP$ may depend on other parameters as well in the following, but this is of no real importance), and denote $P_t=P_t^{(\PP)}$ and $Q_t=Q_t^{(\PP)}$ from Definition \ref{def:Qt-Pt}. 
For $g \in L^\infty(M,\mu)$, define the paraproduct $\Pi_g^{(\PP)}$  on $\calS$ by 
\begin{align} \label{def:paraproduct}
	\Pi_g^{(\PP)}(f)=\Pi_g (f) := \int_0^{+\infty} Q_t f \cdot P_t g \,\frac{dt}{t}, \qquad f \in \calS.
\end{align}
For every $p\in(1,{+\infty})$ and $f\in \calS^p$, the integral is absolutely convergent in $L^p(M,\mu)$:  for $f\in {\mathcal D}_p(\mathcal{L}) \cap \Ran_p(\mathcal{L})$, we have $Q_t f = 2^\PP Q_{t/2}^{(\PP-1)} Q_{t/2}^{(1)}f$ which yields
$$ \|Q_tf\|_{p}\lesssim \|Q_{t/2}^{(1)}f\|_{p} \lesssim (t+t^{-1})^{-1}  \|f\|_{{\mathcal D}_p(\mathcal{L}) \cap \Ran_p(\mathcal{L})}.$$
Combining this estimate with the uniform boundedness of $(P_t)_{t>0}$ in $L^\infty(M,\mu)$ gives the absolute convergence.

\begin{lemma}[Product decomposition] 
Let $p\in(1,{+\infty})$ and $\alpha\in(0,1)$. For every $f,g \in (\calS^p + N(\mathcal{L})) \cap L^\infty(M,\mu)$, we have the product decomposition
\begin{align} \label{eq:paraproduit-decomposition}
	f \cdot g & = \Pi_g(f) + \Pi_f(g) +\mathsf{P}_{N(\mathcal{L})}(f) \mathsf{P}_{N(\mathcal{L})}(g) \qquad \text{in} \ L^p(M,\mu).
\end{align}
Note also that $\Pi_g(f) = \Pi_g(f- \mathsf{P}_{N(\mathcal{L})}(f))$. 
\end{lemma}

\begin{proof} 
By  writing
\begin{align*}
f\cdot g - P_t f \cdot P_t g = (f-P_tf) \cdot g + P_tf \cdot (g-P_t g)
\end{align*}
it follows from \eqref{limzero} and \eqref{liminfty} that in the $L^p$ sense
\begin{align*}
 f \cdot g &= \lim_{t\to 0} \, P_t f \cdot P_t g,\\
 \mathsf{P}_{N(\mathcal{L})} f \cdot \mathsf{P}_{N(\mathcal{L})}g &= \lim_{t\to {+\infty}} \, P_t f \cdot P_t g.
\end{align*}
By definition of $P_t$ and $Q_t$, and using the fact that $t\partial_t P_t=-Q_t$, we then have
\begin{align*}
	f \cdot g & = \lim_{t\to 0} \, (P_t f \cdot P_t g) -  \lim_{t\to {+\infty}} \, (P_t f \cdot P_t g)  +\mathsf{P}_{N(\mathcal{L})}f \cdot \mathsf{P}_{N(\mathcal{L})}g \\
	& = -\int_0^{+\infty} \partial_t \left(P_t f \cdot  P_t g\right) \,dt  + \mathsf{P}_{N(\mathcal{L})}f \cdot \mathsf{P}_{N(\mathcal{L})}g  \\
	& = \Pi_g(f) +\Pi_f(g) + \mathsf{P}_{N(\mathcal{L})}f \cdot \mathsf{P}_{N(\mathcal{L})}g.
\end{align*}
\end{proof}

\begin{coro} \label{cor:A} 
From the nature of $N(\mathcal{L})$ (see Proposition $\ref{prop:kernel}$), the function $\mathsf{P}_{N(\mathcal{L})}(f) \cdot \mathsf{P}_{N(\mathcal{L})}(g)$ (is equal to $0$ or is a constant function) always belongs to $N(\mathcal{L})$. So if the bilinear map $(f,g) \to \Pi_g(f)$ is bounded from $(\calS_p, \| \ \|_{\dot L^p_\alpha}) \times L^\infty$ to $\dot L^p_\alpha$, then by Definition $\ref{def:calS}$ and density, $\Pi_g$ admits a continuous extension on $\dot L^p_\alpha$ and the previous product decomposition yields
 $A(\alpha,p)$.
\end{coro}

Let $\alpha \in (0,1)$ and $g \in L^\infty(M,\mu)$ be fixed. The boundedness of $\Pi_g$ in $\dot L^p_\alpha$ is equivalent to the $L^p$-boundedness of the operator $\mathcal{L}^{{\alpha/2}} \Pi_g( \mathcal{L}^{-{\alpha/2}}\cdot )$.
 Using the definition of the paraproduct, Definition \ref{def:paraproduct}, and the reproducing formula, one may write
$$	\mathcal{L}^{{\alpha/2}} \Pi_g( \mathcal{L}^{-{\alpha/2}} f ) =  \int_0^{+\infty}\int_0^{+\infty} K_{\alpha,g}(s,t) [f] \,\frac{ds}{s}\frac{dt}{t},$$
where
the operator-valued kernel  $K_{\alpha,g} (s,t)$ is given by
\begin{equation} K_{\alpha,g}(s,t)(.):=Q_s \mathcal{L}^{{\alpha/2}} (Q_t \mathcal{L}^{-{\alpha/2}} (\, . \,) \cdot P_t g), \label{eq:K} \end{equation}
and $P_t$ and $Q_t$ are defined in Section \ref{para}. 

We split the paraproduct into the two terms $\Pi_g = \Pi^1_g + \Pi^2_g$,
with
\begin{align*} \Pi_g^1(f) & := \int_0^{+\infty}  (I-P_t) \left[Q_t f \cdot P_t g \right] \, \frac{dt}{t} \\
 & =  \int_0^{+\infty} \int_0^t  Q_s \left[Q_t f \cdot P_t g \right] \, \frac{ds}{s} \frac{dt}{t},
 \end{align*}
 and
\begin{align*}
 \Pi_g^2(f) & := \int_0^{+\infty}  P_t \left[Q_t f \cdot P_t g \right] \, \frac{dt}{t}\\
 & =  \int_0^{+\infty} \int_t^{+\infty}  Q_s \left[Q_t f \cdot P_t g \right] \, \frac{ds}{s} \frac{dt}{t}.
 \end{align*}

An important fact for our study is that under \eqref{due} the second term  $\Pi_g^2$ is bounded on every Sobolev space $\dot L^p_\alpha(M,\mathcal{L},\mu)$ with $\alpha\in(0,1)$ and $p\in(1,{+\infty})$.

\begin{proposition} \label{prop:para2} Let $(M,d,\mu, {\mathcal E})$  be a doubling metric  measure   Dirichlet space satisfying \eqref{due}.
% with a ``carr\'e du champ'' and . 
Let $\alpha\in(0,1)$, $p\in(1,{+\infty})$ and $g\in L^\infty(M,\mu)$. Then $\Pi_g^2$ is bounded on $\dot L^p_\alpha(M,\mathcal{L},\mu)$ with
$$ \| \Pi_g^2 (f) \|_{\dot L^p_\alpha} \lesssim \|f\|_{\dot L^p_\alpha} \|g\|_\infty.$$
\end{proposition}

\begin{proof} The $\dot L^p_\alpha$-boundedness of $\Pi_g^2$ is equivalent to the $L^p$-boundedness of $\mathcal{L}^{{\alpha/2}} \Pi_g^2( \mathcal{L}^{-{\alpha/2}}\cdot )$. Let $f \in L^p(M,\mu)$ and $h \in L^{p'}(M,\mu)$. Then
\begin{align*}
\left| \langle  \mathcal{L}^{\alpha / 2} \Pi_g^2( \mathcal{L}^{-\alpha/2} f ), h \rangle \right|&=\left| \langle  \mathcal{L}^{\alpha / 2} \int_0^{+\infty}  P_t \left[Q_t \mathcal{L}^{-\alpha/2} f \cdot P_t g \right] \, \frac{dt}{t}, h \rangle \right|\\
&=\left| \int_0^{+\infty}  \langle  \mathcal{L}^{\alpha / 2} P_t \left[Q_t \mathcal{L}^{-\alpha/2} f \cdot P_t g \right], h \rangle \, \frac{dt}{t} \right|\\  
&= \left| \int_0^{+\infty}  \langle   (t\mathcal{L})^{-\alpha/2} Q_tf\cdot P_t g,  (t\mathcal{L})^{\alpha / 2} P_t h \rangle \, \frac{dt}{t} \right| \\
&= \left| \int_0^{+\infty}  \int_M   (t\mathcal{L})^{-\alpha/2} Q_tf(x)\cdot P_t g(x) \cdot (t\mathcal{L})^{\alpha / 2} P_t h(x) \,d\mu(x)\, \frac{dt}{t} \right|\\
&\le \|g\|_\infty\left| \int_0^{+\infty}  \int_M   (t\mathcal{L})^{-\alpha/2} Q_tf(x)\cdot  (t\mathcal{L})^{\alpha / 2} P_t h(x) \,d\mu(x)\, \frac{dt}{t} \right|,
\end{align*}
where we have used the uniform boundedness of $P_t$ on $L^\infty(M,\mu)$. Now, by Fubini and Cauchy-Schwarz, 
\begin{align*}
&\left| \langle  \mathcal{L}^{\alpha / 2} \Pi_g^2( \mathcal{L}^{-\alpha/2} f ), h \rangle \right| \\
&\le \|g\|_\infty\  \int_M\int_0^{+\infty}   |(t\mathcal{L})^{-\alpha/2} Q_tf(x)|\cdot  |(t\mathcal{L})^{\alpha / 2} P_t h(x)|\, \frac{dt}{t} \,d\mu(x) \\
&\le \|g\|_\infty\  \int_M\left(\int_0^{+\infty}   |(t\mathcal{L})^{-\alpha/2} Q_tf(x)|^2\, \frac{dt}{t}\right)^{1/2}  \left(\int_0^{+\infty}|(t\mathcal{L})^{\alpha / 2} P_t h(x)|^2\, \frac{dt}{t}\right)^{1/2} \,d\mu(x) \\
& =c\|g\|_\infty \langle g_{\PP-\frac{\alpha}{2}}(f), \tilde g_{D,\frac{\alpha}{2}}(h)\rangle,
\end{align*}
for some $c>0$, where  $g_{\PP-\frac{\alpha}{2}}$ and $\tilde g_{D,\frac{\alpha}{2}}$ are the horizontal square functions from Proposition \ref{prop:square-function}.
 Proposition \ref{prop:square-function} yields that both $g_{\PP-\frac{\alpha}{2}}$ and $\tilde  g_{D,\frac{\alpha}{2}}$ are bounded on $L^p(M,\mu)$ for every $p\in(1,{+\infty})$.

By H\"older's inequality, we then conclude that
\begin{align*}
\left| \langle  \mathcal{L}^{\alpha / 2} \Pi_g^2( \mathcal{L}^{-\alpha/2} f ), h \rangle \right| \lesssim \|g\|_\infty  \|f\|_p \|h\|_{p'},
\end{align*}
which by duality gives the $L^p$-boundedness of $\mathcal{L}^{{\alpha/2}} \Pi_g^2( \mathcal{L}^{-{\alpha/2}}\cdot )$.
\end{proof}

So from now on, to study the $\dot L^p_\alpha$-boundedness of the paraproduct $\Pi_g$, we only have to focus on the first part of the paraproduct and prove the $L^p$-boundedness of
$$	\mathcal{L}^{{\alpha/2}} \Pi^1_g( \mathcal{L}^{-{\alpha/2}} f ) = c_\PP^2 \int_0^{+\infty}\left(\int_0^t K_{\alpha,g}(s,t) [f] \,\frac{ds}{s}\right)\frac{dt}{t}.$$
That means that we may restrict our attention to the study of the operator-valued kernel $K_\alpha(s,t)$ in the range $s\leq t$,
which requires extra assumptions in order to get suitable bounds.

\section{Boundedness of the paraproducts for $2\leq p < p_0$ under $(G_{p_0})$}

Let us introduce an  $L^2$-valued version of  $(R_p)$,  which we will denote by $(\overline{R_p})$:
for every  measurable function $(F_t)_{t>0}$  with values in  $L^2(M,\mu)$,
$$ \left\| \left( \int_0^{+\infty} | {\mathcal R} F_t |^2  \, \frac{dt}{t} \right)^{1/2} \right\|_{p} \lesssim \left\| \left( \int_0^{+\infty} | F_t |^2  \, \frac{dt}{t} \right)^{1/2} \right\|_{p},$$
where ${\mathcal R}:=|\nabla \mathcal{L}^{-1/2}|$ is the Riesz transform.
By applying $(\overline{R_p})$ to $F_t=\sqrt{t\mathcal{L}}P_t^{(N)}f$, for $f\in L^2(M,\mu)$, one sees that, for any $p\in(1,+\infty)$, $(\overline{R_p})$ implies the $L^p$-boundedness of 
the vertical square function $G_N$ for any $N>0$. 
In turn, the $L^p$ boundedness of $G_N$ implies $(G_q)$, for $2<q<p$ (see \cite[step 7 of Theorem 6.1]{A}).
On the other hand, applying $(\overline{R_p})$ to $F_t=f\Eins_{[1,2]}(t)$, for $f\in C_0(M)$  yields $(R_p)$. In the Riemannian context, where $\mathcal{L}$ is given by the Laplace-Beltrami operator and $\nabla$ is the Riemannian gradient,  ${\mathcal R}$ derives from the linear operator $\nabla \mathcal{L}^{-1/2}$. Therefore for any $p\in(1,+\infty)$,  $(R_{p})$ implies back $(\overline{R_p})$ by a general and well-known argument, see for instance \cite[Thm 4.5.11]{Grafakos1}.

However, in our  Dirichlet form setting, the Riesz transform is defined as a sublinear operator (since we only have a notion of length of the gradient), so  it is not clear that $(R_{p})$ implies $(\overline{R_p})$ in this generality. 

We first remark that the $L^p$-boundedness of the Riesz transform for $p\in(1,2]$ (obtained in \cite{CD1}) can be extended to a vector-valued setting:

\begin{proposition}\label{prop:Lp<2}  Let $(M,d,\mu, {\mathcal E})$  be a doubling metric  measure   Dirichlet space with a ``carr\'e du champ'' satisfying \eqref{due}. Then  $(\overline{R_{p}})$ holds for every $p\in (1,2]$. 
\end{proposition}

\begin{proof} We refer the reader to \cite{CD1} for the proof in the scalar case, showing $(R_p)$ for every $p\in(1,2]$ by using a Calder\'on-Zygmund decomposition. By repeating this proof with a vector-valued Calder\'on-Zygmund decomposition (see \cite{RRT}), it then yields that the Riesz transform ${\mathcal R}:=|\nabla \mathcal{L}^{-1/2}|$ is an operator bounded on $L^p(M,L^2[(0,{+\infty});\,\frac{dt}{t}])$ (which is $(\overline{R_{p}})$) for every $p\in(1,2)$.
\end{proof}

 Let us then observe that $(\overline{R_p})$ can be dualised.

\begin{lemma} \label{lemma:RRquadratic} Let $(M,d,\mu, {\mathcal E})$  be a doubling metric  measure   Dirichlet space with a ``carr\'e du champ'' satisfying \eqref{due}. Let $p\in(1,+\infty)$. Assume that  $(\overline{R_p})$ holds.  Then the following $L^2$-valued $(RR_{p'})$ inequality, which we denote by $(\overline{RR_{p'}})$,  is valid:
for every $F:M\times(0,+\infty)\to \R$ such that $F_t=F(.,t)\in \mathcal{D}$ for $t>0$,
$$ \left\| \left(\int_0^{+\infty}|\mathcal{L}^{1/2} F(.,t)|^2 \, \frac{dt}{t}\right)^{1/2} \right\|_{{p'}} \lesssim \left\| \left(\int_0^{+\infty}|\nabla F(.,t)|^2 \, \frac{dt}{t}\right)^{1/2} \right\|_{{p'}}.$$
In particular, $(\overline{RR_{q}})$ holds for every $q\in(2,{+\infty})$.
\end{lemma}

\begin{proof} For every $G:M\times(0,+\infty)\to \R$, we have, denoting $G(.,t)$ by $G_t$, 
\begin{align*}
 \int_0^{+\infty} \langle \mathcal{L}^{1/ 2}F_t , G_t \rangle \, \frac{dt}{t} & = \int_0^{+\infty} \langle  \mathcal{L}F_t , \mathcal{L}^{-{1/2}} G_t \rangle \, \frac{dt}{t} \\
  & = \int_0^{+\infty} \langle  \nabla F_t , \nabla \mathcal{L}^{-{1/2}} G_t \rangle \, \frac{dt}{t}\\
 & \leq \int_M \left(\int_0^{+\infty}  |\nabla F_t|^2 \, \frac{dt}{t}\right)^{1/ 2} \left( \int_0^{+\infty} |{\mathcal R} (G_t)|^2 \frac{dt}{t}\right)^{1/ 2} \, d\mu \\
 & \lesssim \left\| \left(\int_0^{+\infty}  |\nabla F_t|^2 \, \frac{dt}{t}\right)^{1/ 2}\right\|_{p } \left\| \left( \int_0^{+\infty} |{\mathcal R}G_t|^2 \, \frac{dt}{t}\right)^{1/ 2} \right\|_{{p'}}.
\end{align*}
By $(\overline{R_p})$, we get
$$ \int_0^{+\infty} \langle \mathcal{L}^{1/ 2}F_t , G_t \rangle \, \frac{dt}{t}   \lesssim \left\| \left(\int_0^{+\infty}  |\nabla F_t|^2 \, \frac{dt}{t}\right)^{1/ 2}\right\|_{p'} \left\| \left( \int_0^{+\infty} |G_t|^2 \, \frac{dt}{t}\right)^{1/ 2} \right\|_{p}.$$
Taking the supremum over all functions $G\in L^{p}(M,\mu;L^2((0,{+\infty}),\frac{dt}{t}))$ with norm $1$ yields the result. The last assertion follows as a combination of the above with Proposition $\ref{prop:Lp<2}$.
\end{proof}

Our main result of this section is the following:

\begin{theorem}\label{gaza} 
Let  $(M,d,\mu, {\mathcal E})$  be a doubling metric  measure   Dirichlet space with a ``carr\'e du champ'' satisfying \eqref{due}. 
Let $\alpha\in(0,1)$. 
\begin{enumerate}
\item[(i)] There exists $\PP_0:=\PP_0(\nu)$ such that for $\PP\geq \PP_0$, the paraproduct $(g,f)\mapsto \Pi_g (f)$ is bounded from $L^\infty(M,\mu) \times  \dot{L}^2_{\alpha}(M,\mathcal{L},\mu)$ to $\dot{L}^2_{\alpha}(M,\mathcal{L},\mu)$, that is
\begin{equation} \label{paral2}
\| \Pi_g(f) \|_{2,\alpha} \lesssim \|f\|_{2,\alpha} \|g\|_{\infty}.
\end{equation}
Moreover, $A(\alpha,2)$ holds.
\item[(ii)]  Assume in addition $(G_{p_0})$ for some $p_0\in(2,{+\infty}]$, and let $p\in [2,p_0)$. Then there exists $\PP_0:=\PP_0(\nu,p)$ such that for $\PP\geq \PP_0$, the paraproduct $(g,f)\mapsto \Pi_g (f)$ is bounded from $L^\infty(M,\mu) \times  \dot{L}^p_{\alpha}(M,\mathcal{L},\mu)$ to $\dot{L}^p_{\alpha}(M,\mathcal{L},\mu)$, that is
\begin{equation} \label{paralp}
 \| \Pi_g(f) \|_{p,\alpha} \lesssim \|f\|_{p,\alpha} \|g\|_{{\infty}}.
 \end{equation} 
Moreover, $A(\alpha,p)$ holds. 
\end{enumerate}
\end{theorem}

$A(\alpha,2)$  and $A(\alpha,p)$ follow directly from the product decomposition \eqref{eq:paraproduit-decomposition} and \eqref{paral2} and \eqref{paralp}, respectively. See Corollary \ref{cor:A}. 
$A(\alpha,2)$  was already known as emphasised in the introduction. However, the more precise estimate \eqref{paral2}
will be used in Sections \ref{sec:<2}, \ref{sec p>2}, and \ref{sec:osci}.

\begin{proof}[Proof of Theorem $\ref{gaza}$]
By Proposition \ref{prop:reproducing} and Lemma \ref{lem:orthogonality} (for $F_t$ independent of $t$) we can write
\begin{align*}\norm{\mathcal{L}^{\alpha/2} \Pi_g(f)}_{p}&=c \left\|\int_0^{+\infty} Q_t\mathcal{L}^{\alpha/2} \Pi_g(f) \, \frac{dt}{t} \right\|_p\\
&\lesssim \left\|\left(\int_0^{+\infty} |\tilde Q_t\mathcal{L}^{\alpha/2} \Pi_g(f)|^2 \, \frac{dt}{t}\right)^{1/2} \right\|_p\end{align*}
where $\tilde Q_t$ is as in Lemma \ref{lem:orthogonality}.

Then by the definition of the paraproduct
\begin{align*}
	\norm{\mathcal{L}^{\alpha/2} \Pi_g(f)}_{p}
	& \lesssim \left\|\left(\int_0^{+\infty} \left|\tilde Q_t\mathcal{L}^{\alpha/2} \int_0^{+\infty} Q_s f \cdot P_s g \,\frac{ds}{s}\right|^2 \, \frac{dt}{t}\right)^{1/2} \right\|_p\\
		& \leq I_1+I_2,
\end{align*}
where $$I_1:=\left\|\left(\int_0^{+\infty} \left|\tilde Q_t\mathcal{L}^{\alpha/2} \int_0^tQ_s f \cdot P_s g \,\frac{ds}{s}\right|^2 \, \frac{dt}{t}\right)^{1/2} \right\|_p$$
and $$I_2:=\left\|\left(\int_0^{+\infty} \left|\tilde Q_t\mathcal{L}^{\alpha/2} \int_t^{+\infty} Q_s f \cdot P_s g \,\frac{ds}{s}\right|^2 \, \frac{dt}{t}\right)^{1/2} \right\|_p.$$

For $I_1$, we use that, thanks  to \eqref{due},  $\tilde Q_t(t\mathcal{L})^{\alpha / 2}$ is bounded by the Hardy-Littlewood maximal function which satisfies a Fefferman-Stein inequality (see Proposition \ref{prop:FS}), therefore
\begin{align*}I_1 &\lesssim\left\|\left(\int_0^{+\infty} \left| t^{-\alpha/2} \int_0^tQ_s f \cdot P_s g \,\frac{ds}{s}\right|^2 \, \frac{dt}{t}\right)^{1/2} \right\|_p\\&\lesssim\left\|\left(\int_0^{+\infty} \left( t^{-\alpha/2} \int_0^t|Q_s f \cdot P_s g| \,\frac{ds}{s}\right)^2 \, \frac{dt}{t}\right)^{1/2} \right\|_p.\end{align*}
Since by Hardy's inequality we have the pointwise inequality
$$ \left(\int_0^{+\infty} \left( t^{-\alpha/2} \int_0^t|Q_s f \cdot P_s g| \,\frac{ds}{s}\right)^2 \, \frac{dt}{t}\right)^{1/2} \lesssim \left(\int_0^{+\infty} \left( s^{-\alpha/2} |Q_s f \cdot P_s g|\right)^2 \,\frac{ds}{s}\right)^{1/2},$$ 
 we deduce that
\begin{align*}
	I_1 
		&\lesssim\left\|\left(\int_0^{+\infty} \left( s^{-\alpha/2} |Q_s f \cdot P_s g|\right)^2 \,\frac{ds}{s}\right)^{1/2}\right\|_p\\
		&=\left\|\left(\int_0^{+\infty}|Q_s(s\mathcal{L})^{-\alpha/2}\mathcal{L}^{\alpha/2} f \cdot P_s g|^2 \,\frac{ds}{s}\right)^{1/2}\right\|_p.
	\end{align*}
Since $P_s$ is uniformly bounded  on $L^\infty$,
\begin{align*}I_1 &\lesssim\left\|\left(\int_0^{+\infty}|Q_s(s\mathcal{L})^{-\alpha/2}\mathcal{L}^{\alpha/2} f |^2 \,\frac{ds}{s}\right)^{1/2}\right\|_p\|g\|_{\infty}\\
&=\left\|\left(\int_0^{+\infty}|Q_s^{(D-\frac{\alpha}{2})}\mathcal{L}^{\alpha/2} f |^2 \,\frac{ds}{s}\right)^{1/2}\right\|_p\|g\|_{\infty},\end{align*}
 and by the second assertion in Proposition \ref{prop:square-function}, 
\begin{align*}
	I_1 \lesssim \|\mathcal{L}^{\alpha / 2}f\|_{p} \|g\|_{\infty}.
\end{align*}	

As for $I_2$, write
$$I_2=\left\|\left(\int_0^{+\infty} t^{1-\alpha}\left|\tilde Q_t (t\mathcal{L})^{-\frac{1-\alpha}{2}}  \int_t^{+\infty} \mathcal{L}^{1/ 2}\left(Q_s f \cdot P_s g \right)\,\frac{ds}{s}\right|^2 \, \frac{dt}{t}\right)^{1/2} \right\|_p.$$

The Fefferman-Stein inequality for $\tilde Q_t (t\mathcal{L})^{-\frac{1-\alpha}{2}}$ and Hardy's inequality again yield
\begin{align*}
	I_2
	&\lesssim\left\|\left(\int_0^{+\infty} t^{1-\alpha}\left( \int_t^{+\infty} \left|\mathcal{L}^{1/ 2}\left(Q_s f \cdot P_s g\right)\right| \,\frac{ds}{s}\right)^2 \, \frac{dt}{t}\right)^{1/2} \right\|_p\\
	&\lesssim\left\|\left(\int_0^{+\infty} s^{1-\alpha} \left|\mathcal{L}^{1/ 2}\left(Q_s f \cdot P_s g\right)\right|^2 \,\frac{ds}{s}  \right)^{1/2} \right\|_p.	
%	&=\left\|\left(\int_0^\infty \left| L^{1/ 2}\int_t^\infty t^{\frac{1-\alpha}{2}}Q_s (sL)^{-\alpha/2} L^{\alpha / 2}f \cdot P_s g \,\frac{ds}{s}\right|^2 \, \frac{dt}{t}\right)^{1/2} \right\|_p\\	& \lesssim \left\| \left\|  L^{1/ 2} \int_t^\infty \left(\frac{t}{s}\right)^{\frac{1-\alpha}{2}} s^{1/ 2} (Q_s (sL)^{-\alpha/2} L^{\alpha / 2}f \cdot P_s g) \,\frac{ds}{s} \right\|_{L^2(\frac{dt}{t})} \right\|_{p}.
\end{align*}
Then by  Lemma \ref{lemma:RRquadratic} and $p\in(2,{+\infty})$, $(\overline{RR_{p}})$ holds so for $F_s=s^{\frac{1-\alpha}{2}} \left(Q_s f \cdot P_s g\right)$, one obtains
\begin{align*}
	I_2&\lesssim\left\|\left(\int_0^{+\infty} s^{1-\alpha} \left|\nabla\left(Q_s f \cdot P_s g\right)\right|^2 \,\frac{ds}{s}  \right)^{1/2} \right\|_p\\
%	& \lesssim \left\| \left\|   \int_t^\infty \left(\frac{t}{s}\right)^{\frac{1-\alpha}{2}} s^{1/ 2} \nabla (Q_s (sL)^{-\alpha/2} L^{\alpha / 2}f \cdot P_s g) \,\frac{ds}{s} \right\|_{L^2(\frac{dt}{t})} \right\|_{p} \\
	& =\left\| \left\|  s^{\frac{1-\alpha}{2}} |\nabla (Q_s f \cdot P_s g)| \right\|_{L^2(\frac{ds}{s})} \right\|_{p}.
\end{align*}
This splits into two terms $I_{2,1}$ and $I_{2,2}$, according to whether the gradient acts on $Q_s$ or $P_s$.
For the first term, using the uniform boundedness of $P_sg$ on $L^\infty$ and, in the last step, the boundedness of $\tilde{G}^{(D+\frac{1-\alpha}{2})}$ on $L^p$ stated in Proposition \ref{prop:square-function} (ii) and Proposition \ref{prop:K},
one obtains
\begin{align*}
	I_{2,1}& = \left\| \left\|  s^{1/ 2} |\nabla Q_s (s\mathcal{L})^{-\alpha/2} \mathcal{L}^{\alpha / 2}f | \cdot |P_s g| \right\|_{L^2(\frac{ds}{s})} \right\|_{p} \\
	& = \left\| \left\|  s^{1/ 2} |\nabla Q_s^{(D-\frac{\alpha}{2})}  \mathcal{L}^{\alpha / 2}f |\cdot |P_s g |\right\|_{L^2(\frac{ds}{s})} \right\|_{p} \\
	 & \lesssim \left\| \left\|  s^{1/ 2} |\nabla Q_s^{(D-\frac{\alpha}{2})} \mathcal{L}^{\alpha / 2} f| \right\|_{L^2(\frac{ds}{s})} \right\|_{p} \|g\|_{\infty}\\
 & = \left\|   \tilde{G}^{(D+\frac{1-\alpha}{2})}\mathcal{L}^{\alpha/2} f  \right\|_{p} \|g\|_{\infty} 
% & = \left\| \left\| Q_s^{(D+\frac{1-\alpha}{2})}\mathcal{L}^{\alpha/2} f \right\|_{L^2(\frac{ds}{s})} \right\|_{p} \|g\|_{\infty} \\
 \lesssim \| \mathcal{L}^{\alpha/2} f \|_{p} \|g\|_{\infty}.
\end{align*}
As for $I_{2,2}$, using the Carleson duality stated in Theorem \ref{CarlesonDuality}, we have for every $\eps>0$ (with $\eps=0$ if $p=2$)
\begin{align*}	 \left\| \left(\int_0^{+\infty} s^{1-\alpha}  |Q_s f|^2 |\nabla P_s g|^2 \,\frac{ds}{s}\right)^{1/2} \right\|_p &= \left\| \left(\int_0^{+\infty} |s^{-\alpha/2}Q_s   f|^2 \cdot |\sqrt{s} \nabla P_s g|^2 \,\frac{ds}{s}\right)^{1/2}\right\|_p\\
%&=\left(\int_{M\times (0,\infty)} |s^{-\alpha/2}Q_s   f(x)|^2 \cdot |\sqrt{s} \nabla P_s g(x)|^2 \,\frac{d\mu(x)ds}{s}\right)^{1/2}\\
	& \lesssim 	\norm{N_\ast(s^{-\alpha/2}Q_s f)}_{p} \norm{\scrC_{p+\epsilon}( s^{1/ 2} \nabla P_s g)}_{\infty}.
%	\\
%	& \lesssim \norm{\mathcal{L}^{\alpha/2} f}_p \norm{g}_\infty,
\end{align*}
We apply Lemma \ref{lem:nontang-Carleson} below (choosing $q=p+\eps<p_0$) to show that the last expression can be bounded by a constant times $\norm{\mathcal{L}^{\alpha/2} f}_p \norm{g}_\infty$. 
Finally, we have shown that
$$ 	\norm{\mathcal{L}^{\alpha/2} \Pi_g(f)}_{p} \lesssim \| \mathcal{L}^{\alpha/2} f \|_{p} \|g\|_{\infty}.$$
\end{proof}

It remains to show the following.

\begin{lemma} \label{lem:nontang-Carleson}
Let  $(M,d,\mu, {\mathcal E})$  be a doubling metric  measure   Dirichlet space with a ``carr\'e du champ'' satisfying \eqref{due}. 
Let $\alpha \in (0,1)$.
\begin{enumerate}
\item
Let $p \in (1,+\infty)$. Then 
\[
	\|N_\ast((s\mathcal{L})^{-{\alpha/2}}Q_s f)\|_p \lesssim  \|f\|_p
\]
for all $f \in L^p(M,\mu)$.
\item Let $q \in [2,+\infty)$. If $q>2$, assume in addition $(G_{p_0})$ for some $p_0>q$. Then 
\[
	\|\scrC_{q} (\sqrt{s}|\nabla P_s g|)\|_{\infty} \lesssim \|g\|_\infty
\]
for all $g \in L^\infty(M,\mu)$.
\end{enumerate}
\end{lemma}

\begin{proof}
(a) According to Lemma \ref{prop:kernel-est}, the kernel $k_s$ of the  operator $(s\mathcal{L})^{-{\alpha/2}}Q_s=Q_s^{(D-\frac{\alpha}{2})}$ satisfies estimates of the form \eqref{kernel-est} of order $\PP-\frac{\alpha}{2}$. Thus, for $x \in M$,
\begin{align*}
	N_\ast ((s\mathcal{L})^{-{\alpha/2}}Q_s f)(x)
		&=\sup_{(y,s) \in \Gamma(x)} \abs{(s\mathcal{L})^{-{\alpha/2}}Q_s f(y)}\\
		&\leq \sup_{(y,s) \in \Gamma(x)} \int_M \abs{k_s(y,z)}\abs{f(z)}\,d\mu(z)\\
		& \lesssim \sup_{(y,s) \in \Gamma(x)} \sum_{j=0}^{+\infty} \frac{1}{V(y,\sqrt{s})} \left(\frac{s}{(2^{j-1}\sqrt{s})^2}\right)^{\PP-\frac{\alpha}{2}} \int_{S_j\left(B(y,\sqrt{s})\right)} \abs{f(z)} \,d\mu(z)\\
		& \lesssim \sup_{(y,s) \in \Gamma(x)} \sum_{j=0}^{+\infty} \frac{2^{j\nu}}{V(y,2^j\sqrt{s})} \left(\frac{s}{(2^{j-1}\sqrt{s})^2}\right)^{\PP-\frac{\alpha}{2}} \int_{B(y,2^j\sqrt{s})} \abs{f(z)} \,d\mu(z)\\
		& \lesssim \sup_{(y,s) \in \Gamma(x)} \sup_{j \in \N} \frac{1}{V(y,2^j\sqrt{s})} \int_{B(y,2^j\sqrt{s})} \abs{f(z)} \,d\mu(z)
		\lesssim \mathcal{M}f(x),
\end{align*}
where $S_j\left(B\right)=2^jB\setminus 2^{j-1}B$ if $j\ge 1$ and $S_0(B)=B$. Here we have used \eqref{dnu} and $\PP-\frac{\alpha}{2}>\frac{\nu}{2}$.
The assertion in (a) follows from the boundedness of the uncentred Hardy-Littlewood maximal operator $\mathcal{M}$ on $L^p(M,\mu)$. \\
(b) Fix a ball $B \subseteq M$. We have to estimate
$$A(g):=\left( \aver{B} \left(\int_0^{r^2(B)} \abs{\sqrt{s}\nabla P_s g(x)}^2 \,\frac{ds}{s}\right)^{q/2} \, d\mu(x) \right)^{1/q}.$$
To this aim, we split
$$ g = g \Eins_{4B} + \sum_{j\geq 3} g\Eins_{S_j(B)}.$$
First using the $L^{q}$-boundedness of the square function $G_{\PP+(1-\alpha)/2}$ stated in Proposition \ref{prop:K}, we have
\begin{align*}
A(g \Eins_{4B}) & \leq \left( \aver{B} (G_{\PP+(1-\alpha)/2}(g \Eins_{4B}))^{q} \, d\mu \right)^{1/q} \\
 & \lesssim |B|^{-1/q} \norm{ G_{\PP+(1-\alpha)/2}(g \Eins_{4B})}_{q} \\
& \lesssim |B|^{-1/q} \norm{g}_{L^{q}(4B)} \lesssim \norm{g}_{\infty}.
\end{align*}

On the other hand, interpolating $(G_{p_0})$ with the Davies-Gaffney estimates from Proposition \ref{prop:Davies-Gaffney} yields $L^{q}$ off-diagonal estimates, 
therefore for $j \geq 3$ and every integer $N\geq 1$
\begin{align*}
	\left(\aver{B} \abs{\sqrt{s}\nabla P_s \Eins_{S_j(B)} g(x)}^{q}\,d\mu(x)\right)^{1/q}
	& \lesssim \left(1+\frac{(2^j r(B))^2}{s}\right)^{-N} |B|^{-1/q} \norm{g}_{L^{q}(2^jB)}\\
	& \lesssim 2^{-2jN}\left(\frac{s}{r^2(B)}\right)^N  2^{j\nu/q} \norm{g}_\infty,
\end{align*}
for $s\le r^2(B)$. Hence, choosing $N>\nu/4$, for $q\geq 2$
\begin{align*}
A(\Eins_{S_j(B)}  g) &\leq \left(\int_0^{r^2(B)} \left(\aver{B} \abs{\sqrt{s}\nabla P_s\Eins_{S_j(B)}  g(x)}^{q} \, d\mu(x)\right)^{2/q} \, \frac{ds}{s}\right)^{1/2}\\
	&\lesssim   2^{-(2N-\frac{\nu}{q})j} \left(\int_0^{r^2(B)} \left(\frac{s}{r^2(B)}\right)^{2N} \,\frac{ds}{s}\right)^{1/2} \norm{g}_\infty \\
	&	\lesssim 2^{-(2N-\frac{\nu}{q})j} \norm{g}_\infty.
\end{align*}
Gathering the above estimates, uniformly with respect to any ball $B$, we have
\begin{align*}
	\left( \aver{B} \left(\int_0^{r^2(B)} \abs{\sqrt{s}\nabla P_s g(x)}^2 \,\frac{ds}{s}\right)^{q/2} \, d\mu(x) \right)^{1/q} \lesssim \norm{g}_\infty,
\end{align*}
which yields the claim.
\end{proof}

\section{Off-diagonal estimates on the kernel of paraproducts} \label{sec:off}

We recall that $K_{\alpha,g}$ denotes the operator-valued kernel of the paraproduct, and that this kernel depends on a parameter $\PP$, see \eqref{def:paraproduct} and \eqref{eq:K}.\\

In order to derive off-diagonal estimates on the kernel $K_{\alpha,g}$, we are going to assume $L^{p_2}$-$L^{p_2}$ off-diagonal estimates on the gradient of the semigroup for some $p_2\in (2,+\infty)$: for every pair of balls $B_1,B_2$ of radius $\sqrt{t}$ and every $f \in L^{p_2}(M,\mu)$ with $\supp f \subseteq B_2$,
\begin{equation}
	\left(\aver{B_1} | \sqrt{t} \nabla e^{-t\mathcal{L}} f | ^{p_2} %+ |\nabla P_t f|^{p_2} 
	\, d\mu \right)^{1/ {p_2}} \lesssim e^{-c \frac{d^2(B_1,B_2)}{t}} \left(\aver{B_2} | f | ^{p_2} \, d\mu \right)^{1/ {p_2}}.
\label{eq:gradient}
\end{equation}
Note that this estimate can be obtained by interpolating between $(G_p)$ for $p>p_2$ and  the Davies-Gaffney estimate from Proposition
\ref{prop:Davies-Gaffney}.

\begin{theorem} \label{thm:KLpLqII} Let $(M,d,\mu, {\mathcal E})$  be a doubling metric  measure   Dirichlet space with a ``carr\'e du champ'' satisfying \eqref{due}. Let $1\leq p_1 \leq 2 \leq p_2< +\infty$, $\alpha\in(0,1)$ and $g\in L^\infty(M,\mu)$. 
Assume $\eqref{eq:gradient}$.
Then for $s\leq t$, the kernel $K_{\alpha,g}$ satisfies the following $L^{p_1}$-$L^{p_2}$ off-diagonal estimates: given $\tilde{N}>\frac{\nu}{2}$, there exists $\PP_0=\PP_0(\tilde{N},\nu)>0$ such that for every integer $\PP\geq \PP_0$ we have  
$$	\left(\aver{B_1} | K_{\alpha,g}(s,t)  h|^{p_2} \, d\mu \right)^{1/p_2} \lesssim \left(\frac{s}{t}\right)^{\frac{1-\alpha}{2}}  \left(1+\frac{d^2(B_1,B_2)}{t}\right)^{-\tilde{N}} \left( \aver{B_2} |h|^{p_1} \, d\mu \right)^{1/p_1} \|g\|_{\infty}$$
for all balls $B_1,B_2$ of radius $\sqrt{t}$.
\end{theorem}

One can obtain a more precise result if one assumes in addition a De Giorgi property.

\begin{theorem} \label{thm:KLpLqIII} Let $(M,d,\mu, {\mathcal E})$  be a doubling metric  measure   Dirichlet space with a ``carr\'e du champ'' satisfying \eqref{due}. Let $1\leq p_1 \leq 2 \leq p_2< +\infty$, $\alpha\in(0,1)$ and $g\in L^\infty(M,\mu)$. Assume $\eqref{eq:gradient}$
and that $({DG}_{2,\kappa})$ holds for some $\kappa\in(0,1)$. Then for $s\leq t$, the kernel $K_{\alpha,g}$ satisfies the following $L^{p_1}$-$L^{\infty}$ off-diagonal estimates: given $\kappa'\in(\kappa,1)$ and $\tilde{N}>\frac{\nu}{2}$, there exists $\PP_0=\PP_0(\tilde{N},\nu,p_2,\kappa)>0$ such that for every integer $\PP\geq \PP_0$ we have  
$$	\|K_{\alpha,g}(s,t)  h\|_{L^\infty(B_1)} \lesssim \left(\frac{s}{t}\right)^{\frac{1-\alpha}{2}}  \left(\frac{t}{s}\right)^{\frac{\kappa'}{2}} \left(1+\frac{d^2(B_1,B_2)}{t}\right)^{-\tilde{N}} \left( \aver{B_2} |h|^{p_1} \, d\mu \right)^{1/p_1} \|g\|_{\infty}$$
for all balls $B_1$ and $B_2$ of radius $\sqrt{t}$.
\end{theorem}

The rest of this section is devoted to the proof of Theorems $\ref{thm:KLpLqII}$ and $\ref{thm:KLpLqIII}$. We will need two lemmas.

The first one is a localised version of the fact that, for $p>2$,   $(RR_p)$ holds under \eqref{due} (see \cite{CD1}). 

\begin{lemma} \label{lemf}  Let $(M,d,\mu, {\mathcal E})$  be a doubling metric  measure   Dirichlet space with a ``carr\'e du champ'' satisfying \eqref{due}. Fix $p\in [2,+\infty)$ and $N>\frac{\nu+1}{2}$. Then, for all $r>0$, every  ball $B_{r}$ of radius $r$, every bounded covering $(B^i_{r})_i$  of  $M$ by balls of radius $r$, and $f\in\mathcal{F}$,
\begin{equation*}
\left(\aver{B_r}|\sqrt{\mathcal{L}} Q_{r^2}^{(N)} f|^p\,d\mu\right)^{1/p}  \lesssim \sum_{i} \left(1+ \frac{d^2(B_r,B^i_{r})}{r^2}\right)^{{-(N-\frac{\nu+1}{2})}} \left(\aver{B^i_r}|\nabla f|^2\,d\mu\right)^{1/2}.
\end{equation*}
\end{lemma}

\begin{proof}
Let $g \in L^{p'}(M,\mu)$ be supported on $B_r$. 
By duality, we have
$$ \langle  \sqrt{\mathcal{L}} Q_{r^2}^{(N)} f , g\rangle  = \langle   f , \mathcal{L} \mathcal{L}^{-{1/2}}  Q_{r^2}^{(N)} g\rangle.$$
By \eqref{energy} and \eqref{eq:carre}, it follows that
\begin{align}\label{dual}
 \left|\langle  \sqrt{\mathcal{L}} Q_{r^2}^{(N)} f , g\rangle\right| & \leq \int | \nabla f|  |\nabla \mathcal{L}^{-{1/2}} Q_{r^2}^{(N)} g|\,  d\mu \\
 & \leq \sum_{i}  \||\nabla f|\|_{L^2(B^i_r)} \||\nabla \mathcal{L}^{-{1/2}} Q_{r^2}^{(N)} g|\|_{L^{2}(B^i_r)}.\nonumber
\end{align}
Write $\nabla \mathcal{L}^{-{1/2}} Q_{r^2}^{(N)}=r\nabla e^{-r^2\mathcal{L}/2}(r^2\mathcal{L})^{N-\frac{1}{2}}e^{-r^2\mathcal{L}/2}$. 
By interpolating $(G_{q})$ for $1<q<p'$, which holds since $p'\leq 2$, with $L^2$ Davies-Gaffney estimates from Proposition \ref{prop:Davies-Gaffney}, we know that $r\nabla e^{-r^2\mathcal{L}/2}$ satisfies $L^{p'}$-$L^{2}$ off-diagonal estimates of exponential order.
Now $$(r^2\mathcal{L})^{N-\frac{1}{2}}e^{-r^2\mathcal{L}/2}=2^{N-\frac{1}{2}} Q_{r^2/2}^{(N-\frac{1}{2})},$$ hence by Lemma \ref{prop:kernel-est}, this operator satisfies $L^{p'}$-$L^{p'}$ off-diagonal estimates of order $N-\frac{1}{2}$. 
By Lemma \ref{lem:comp-OD} and using $N>\frac{\nu+1}{2}$, one obtains 
\begin{equation} \left(\aver{B_r^i}|\nabla \mathcal{L}^{-{1/2}} Q_{r^2}^{(N)} g|^{2}\,d\mu\right)^{1/2}
\lesssim \left(1+ \frac{d(B_r,B_r^i)}{r}\right)^{-2N+1} \left(\aver{B_r}|g|^{p'}\,d\mu\right)^{1/p'}.\label{eq:od2}
\end{equation} 
The claim now follows from \eqref{d} and \eqref{dual}.
\end{proof}

\begin{proof}[Proof of Theorems $\ref{thm:KLpLqII}$ and $\ref{thm:KLpLqIII}$]
Let us start with Theorem \ref{thm:KLpLqIII}, which is slightly more difficult.
First note that it suffices to prove the desired estimate for a ball $B_1$ of radius $\sqrt{s}$, since if $B_1$ is of radius $\sqrt{t}$ then for every ball $\tilde B_1$ of radius $\sqrt{s}$ contained in $B_1$, we have
$$\left(1+\frac{d^2(B_1,B_2)}{t}\right) \simeq \left(1+\frac{d^2(\tilde B_1,B_2)}{t}\right).$$
So consider $B_1$ a ball of radius $\sqrt{s}$ and $B_2$ a ball of radius $\sqrt{t}$. By choosing $\tilde Q_s$ such that $\tilde Q_s^2= Q_s$,  it follows that
$$ K_{\alpha,g}(s,t)  h = \tilde Q_s \tilde K_{\alpha,g}(s,t) \tilde Q_t  h$$
where $\tilde K_{\alpha,g}=\tilde Q_s \mathcal{L}^{{\alpha/2}} (\tilde Q_t \mathcal{L}^{-{\alpha/2}} (\, . \,) \cdot P_t g)$ is of the exact same nature as $K_{\alpha,g}$ (with the intrinsic constant $\PP$ being replaced by $\PP/2$).
Since $\tilde Q_s$ (resp. $\tilde Q_t$) satisfies $L^{p_2}-L^\infty$ (resp. $L^{p_1}$-$L^{p_2}$) off-diagonal estimates at scale $\sqrt{s}$ (resp. $\sqrt{t}$) at order $\PP/2$, by the composition of off-diagonal estimates (see Lemma \ref{lem:comp-OD}), the expected result will follow from the following $L^{p_2}$-$L^{p_2}$ off-diagonal estimates: 
\begin{align}	
& \left( \aver{B_1} |\tilde K_{\alpha,g}(s,t)  h|^{p_2} \, d\mu \right)^{1/p_2}  \nonumber \\
& \qquad \lesssim \left(\frac{s}{t}\right)^{\frac{1-\alpha}{2}}  \left(\frac{t}{s}\right)^{\frac{\kappa}{2}} \left(1+\frac{d^2(B_1,B_2)}{t}\right)^{-\tilde{N}} \left( \aver{B_2} |h|^{p_2} \, d\mu \right)^{1/p_2} \|g\|_{\infty} \label{eq:aa} \end{align}
for all balls $B_1$ and $B_2$ of respective radii $\sqrt{s}$ and $\sqrt{t}$ and every function $h$ supported on $B_2$.

So it remains us to check \eqref{eq:aa}. Fix such balls $B_1,B_2$ and function $h$ supported on $B_2$. By definition
\begin{align*}
	\tilde K_{\alpha,g}(s,t) h = \left(\frac{s}{t}\right)^{\frac{1-\alpha}{2}} (t\mathcal{L})^{1/ 2} \tilde Q_s (s\mathcal{L})^{-\frac{1-\alpha}{2}}   (\tilde Q_t (t\mathcal{L})^{-\alpha/2}  h \cdot P_t g).
\end{align*}
Therefore, with Lemma \ref{lemf} (for $p=p_2\geq 2$ and $N=\tilde \PP:=\frac{\PP}{2}-\frac{1-\alpha}{2}>\frac{\nu+1}{2}$), one has 
\begin{align}
	& \left(\aver{B_1} | \tilde K_{\alpha,g}(s,t) h|^{p_2} \, d\mu \right)^{1/p_2} \nonumber \\
	&	\quad \lesssim    \left(\frac{s}{t}\right)^{\frac{1-\alpha}{2}} \sum_{i} \left(1+ \frac{d^2(B_1,\tilde B_i)}{s}\right)^{{-(\tilde \PP -\frac{\nu+1}{2})}} \left(\aver{\tilde B_i}|\sqrt{t} \nabla  (\tilde Q_t (t\mathcal{L})^{-\alpha/2}  h \cdot P_t g)|^{2}\,d\mu\right)^{1/2}, \label{eq:ois}
\end{align}
where $(\tilde B_i)_i$ is a bounded covering of the whole space with balls of radius $\sqrt{s}$.
Then by distributing the gradient,  two terms appear. First using the property $({DG}_{2,\kappa})$, it follows for every ball $\tilde B_i$ that 
\begin{align*} 
 & \left(\aver{\tilde B_i} |\sqrt{t} \nabla  (\tilde Q_t (t\mathcal{L})^{-\alpha/2} h) |^{2} \, d\mu \right)^{1/2} \\
 & \qquad  \lesssim  \left(\frac{t}{s}\right)^{\frac{\kappa'}{2}}  \left(\aver{\bar B_i} |\sqrt{t} \nabla  (\tilde Q_t (t\mathcal{L})^{-\alpha/2} h) |^{2} \, d\mu \right)^{1/2} + \left(\frac{t}{s}\right)^{\frac{\kappa'}{2}} \norm{\tilde Q_t (t\mathcal{L})^{1-\alpha/2} h}_{L^\infty(\bar B_i)},
%  \sum_{j \geq 0} 2^{-j (\tilde{N}-\frac{\nu}{p_2})}  \left(\aver{2^{\ell+k+j} \tilde B_1} | \tilde Q_t  h |^{p_2}  \, d\mu \right)^{1/p_2} 
% \\ & \qquad + \left(\frac{t}{s}\right)^{\frac{\kappa}{2}} \norm{ tL Q_t (tL)^{-\alpha/2} h}_{L^{\infty}(2^{\ell+k} \tilde B_1)},
\end{align*}
where $\bar B_i=\frac{\sqrt{t}}{\sqrt{s}} \tilde B_i$ is the dilated ball of radius $\sqrt{t}$.
Then by writing $\sqrt{t} \nabla  \tilde Q_t (t\mathcal{L})^{-\alpha/2} = 4^{(\PP-\alpha)/2} \sqrt{t} \nabla e^{-\frac{t}{4}{\mathcal L}} Q_{t/4}^{(\PP/2-\alpha/2)}$, since $\sqrt{t} \nabla e^{-\frac{t}{4}{\mathcal L}}$ satisfies $L^{2}$-$L^{2}$ off-diagonal estimates at scale $\sqrt{t}$ at any order and $Q_{t/4}^{(\PP/2-\alpha/2)}$ satisfies $L^{p_1}$-$L^{2}$ off-diagonal estimates at scale $\sqrt{t}$ at order $(\PP-\alpha)/2$, we deduce by Lemma \ref{lem:comp-OD} that $\sqrt{t} \nabla  \tilde Q_t (t\mathcal{L})^{-\alpha/2}$ also satisfies $L^{p_1}$-$L^{2}$ off-diagonal estimates at scale $\sqrt{t}$ at order $(\PP-\alpha)/2$. Moreover $\tilde Q_t (t\mathcal{L})^{1-\alpha/2}$ satisfies $L^{p_1}$-$L^\infty$ off-diagonal estimates at the scale $\sqrt{t}$ of order $\PP/2+1-\alpha/2\geq \PP-\alpha/2$. So we obtain
\begin{align*} \left(\aver{\tilde B_i} |\sqrt{t} \nabla \tilde Q_t (t\mathcal{L})^{-\alpha/2} h |^{2} \, d\mu \right)^{1/2}\lesssim & \left(\frac{t}{s}\right)^{\frac{\kappa'}{2}} \left(1+\frac{d^2(B_2,\bar B_i)}{t}\right)^{-(\PP-\alpha)/2}  \left(\aver{B_2} | h |^{p_1} \, d\mu \right)^{1/p_1}.
\end{align*}
Similarly, one has
\begin{align*} 
\left(\aver{\tilde B_i} |\sqrt{t} \nabla P_t g  |^{2} \, d\mu \right)^{1/2}\lesssim & \left(\frac{t}{s}\right)^{\frac{\kappa'}{2}} \left(1+\frac{d^2(B_2,\bar B_i)}{t}\right)^{-\PP}  \left(\aver{B_2} | g |^{p_1} \, d\mu \right)^{1/p_1} \\
& \lesssim \left(\frac{t}{s}\right)^{\frac{\kappa'}{2}} \|g\|_\infty.
\end{align*}
So coming back to \eqref{eq:ois}, we obtain that for a large enough parameter $\PP$, it follows
 \begin{align*}
	& \left(\aver{B_1} | \tilde K_{\alpha,g}(s,t)  h|^{p_2} \, d\mu \right)^{1/p_2} \\
	& \quad 	\lesssim    \left(\frac{s}{t}\right)^{\frac{1-\alpha}{2}}\left(\frac{t}{s}\right)^{\frac{\kappa'}{2}} \sum_{i} \left(1+ \frac{d^2(B_1,\tilde B_i)}{s}\right)^{{-(\tilde \PP-\frac{\nu+1}{2})}} \left(1+\frac{d^2(B_2,\bar B_i)}{t}\right)^{-(\PP-\alpha)/2}  \left(\aver{B_2} | h |^{p_1} \, d\mu \right)^{1/p_1} \|g\|_\infty.	
\end{align*}
Since $\bar B_i$ is the dilated ball of radius $\sqrt{t}$ from $\tilde B_i$, we then deduce that
$$ \left(1+ \frac{d^2(B_2, \bar B_i)}{t}\right) \simeq \left(1+ \frac{d^2(B_2,\tilde B_i)}{t}\right) $$
and so since $s\leq t$
$$ \left(1+ \frac{d^2(B_1,B_2)}{t}\right) \lesssim \left(1+ \frac{d^2(B_2, \bar B_i)}{t}\right) \left(1+ \frac{d^2(B_1,\tilde B_i)}{s}\right).$$
Hence as soon as $\PP$ is large enough so that $$C:=C(\PP)=\min\{\tilde \PP - \frac{\nu+1}{2}, (\PP-\alpha)/2\}-(\nu+1)>0,$$ we have 
 \begin{align*}
	& \left(\aver{B_1} | \tilde K_{\alpha,g}(s,t)  h|^{p_2} \, d\mu \right)^{1/p_2} \\
	\lesssim& 	    \left(\frac{s}{t}\right)^{\frac{1-\alpha}{2}}\left(\frac{t}{s}\right)^{\frac{\kappa'}{2}} \left(1+ \frac{d^2(B_1,B_2)}{t}\right)^{-C}	\left(\sum_{i} \left(1+ \frac{d^2(B_1,\tilde B_i)}{s}\right)^{-(\nu+1)}\right)  \\
	&\times\left(\aver{B_2} | h |^{p_1} \, d\mu \right)^{1/p_1} \|g\|_\infty \\
 \lesssim	&     \left(\frac{s}{t}\right)^{\frac{1-\alpha}{2}}\left(\frac{t}{s}\right)^{\frac{\kappa'}{2}} \left(1+ \frac{d^2(B_1,B_2)}{t}\right)^{-C} \left(\aver{B_2} | h |^{p_1} \, d\mu \right)^{1/p_1} \|g\|_\infty,
\end{align*}
where we used that $(B_i)$ is a bounded covering at scale $\sqrt{s}$ (which is also the radius of $B_1$) to bound the sum over the covering.
Since $C=C(\PP)$ can be taken as large as we want according to a large parameter $\PP$, we deduce the statement \eqref{eq:aa}, which as we already have seen, concludes the proof of Theorem \ref{thm:KLpLqIII}.

For Theorem \ref{thm:KLpLqII}, the situation is simpler because we already have the exponent $p_2$ on the left hand side, and balls and operators can be considered at scale $\sqrt{t}$. Indeed, by summing the estimates of Lemma \ref{lemf} along a covering of balls of radius $\sqrt{s}$, we get for $s\leq t$ and $B_1,B_2$ balls of radius $\sqrt{t}$
\begin{align*}
\left(\aver{B_1}|\sqrt{\mathcal{L}} Q_{s}^{(N)} f|^p\,d\mu\right)^{1/p} &  \lesssim  \left(1+ \frac{d^2(B_1,B_2)}{s}\right)^{{-(N-\frac{2\nu+1}{2})}} \left(\aver{B_2}|\nabla f|^p\,d\mu\right)^{1/p} \\
& \lesssim  \left(1+ \frac{d^2(B_1,B_2)}{t}\right)^{{-(N-\frac{2\nu+1}{2})}} \left(\aver{B_2}|\nabla f|^p\,d\mu\right)^{1/p}.
\end{align*}
We then conclude as previously, using the Leibniz rule on the gradient.
The result then follows by composing $L^{p_2}$ off-diagonal estimates at the scale $\sqrt{t}$, see Lemma \ref{lem:comp-OD}.
\end{proof}

\section{The case $1<p<2$} \label{sec:<2}

This section is devoted to the study of $A(\alpha,p)$ with $1<p<2$. Our main result is the following.

\begin{theorem} \label{main} Let $(M,d,\mu, {\mathcal E})$  be a doubling metric  measure   Dirichlet space with a ``carr\'e du champ'' satisfying \eqref{due}. Then property $A(\alpha,p)$ holds for every $p\in(1,2)$ and every $\alpha\in(0,1)$.
\end{theorem}

According to the product decomposition formula \eqref{eq:paraproduit-decomposition} and Corollary $\ref{cor:A}$, Theorem \ref{main} is  a consequence of the following.

\begin{theorem} \label{thm:tent-extrap-small} Let $(M,d,\mu, {\mathcal E})$  be a doubling metric  measure   Dirichlet space with a ``carr\'e du champ'' satisfying \eqref{due}. Let $p \in (1,2)$ and  $\alpha\in (0,1)$. 
There exists $\PP_0=\PP_0(\nu)>0$ such that for every integer $\PP\geq \PP_0$, the paraproduct $(g,f)\mapsto \Pi_g^{(\PP)} (f)$ defined in \eqref{def:paraproduct} is bounded from $L^\infty(M,\mu) \times  \dot{L}^p_{\alpha}(M,\mathcal{L},\mu)$ to $\dot{L}^p_{\alpha}(M,\mathcal{L},\mu)$. We have
\[
	\norm{\Pi_g^{(\PP)}(f)}_{p,\alpha} \lesssim \norm{f}_{p,\alpha} \norm{g}_{\infty},
\]
and $A(\alpha,p)$ holds.
\end{theorem}

Let $\alpha \in (0,1)$ and $g \in L^\infty(M,\mu)$, let $s,t>0$. Recall the operator $K_{\alpha,g}(s,t)$ defined in \eqref{eq:K} by
\begin{equation*} 
K_{\alpha,g}(s,t):=Q_s \mathcal{L}^{\alpha/2} (Q_t \mathcal{L}^{-\alpha/2} (\, . \,) \cdot P_t g), 
\end{equation*}
so that
\[
	\mathcal{L}^{\alpha/2}\Pi_g( \mathcal{L}^{-\alpha/2} f) 
	= \int_0^{+\infty} Q_s \mathcal{L}^{\alpha/2} \Pi_g( \mathcal{L}^{-\alpha/2} f) \,\frac{ds}{s}
			=  \int_0^{+\infty} \int_0^{+\infty} K_{\alpha,g} (s,t) f \,\frac{dt}{t}\frac{ds}{s},
\]
and 
\[
	\mathcal{L}^{\alpha/2}\Pi_g^1( \mathcal{L}^{-\alpha/2} f) 
	= \int_0^{+\infty} Q_s \mathcal{L}^{\alpha/2} \Pi_g^1( \mathcal{L}^{-\alpha/2} f) \,\frac{ds}{s}
			=  \int_0^{+\infty} \int_0^t K_{\alpha,g} (s,t) f \,\frac{ds}{s}\frac{dt}{t}.
\]
We refer the reader to Section \ref{para} for the definition of $\Pi_g^1$, which is the remaining part of the paraproduct that we have to study (see Proposition \ref{prop:para2}).

In the sequel, we describe how the off-diagonal estimates of the kernel $K_{\alpha,g}$ as obtained in Section \ref{sec:off} can be used to obtain boundedness of the paraproducts by means of an extrapolation method.

We recall the extrapolation tool for $p\in(1,2)$.

\begin{proposition} \label{prop:extrap<2} Let $T$ be a bounded linear operator on $L^2(M,\mu)$. Assume that $T$ satisfies the following off-diagonal estimates: there exist integers $N>\frac{\nu}{2}$ and  $\tilde{N}>\frac{\nu}{2}$ such that for every $t>0$ and every pair of balls $B_1,B_2$ of radius $r=\sqrt{t}$
\begin{equation} \label{eq:amontrer}
 \left\| T Q_t^{(N)} \right\|_{L^2(B_1) \to L^2(B_2)} \lesssim \left(1+\frac{d(B_1,B_2)}{r}\right)^{-\tilde{N}}.
  \end{equation}
Then for every $p\in(1,2)$, $T$ is bounded on $L^p(M,\mu)$. 
\end{proposition}

\begin{rem} The same proof yields that $T$ is bounded on the weighted space $L^p(\omega)$ for every weight $\omega \in {\mathbb A}_{p} \cap RH_{(\frac{2}{p})'}$.
\end{rem}

\begin{proof}[Proof of Proposition $\ref{prop:extrap<2}$] We refer the reader to \cite[Theorem 5.11]{BZ} and to \cite[Theorem 6.4]{BZ} (for the weighted part) for a proof of this result. The second assumption of \cite[Theorem 5.11]{BZ} is satisfied as a consequence of the kernel estimates for $P_t^{(N)}$ established in Lemma \ref{prop:kernel-est}. Notice however that instead of \eqref{eq:amontrer}, the first assumption of \cite[Theorem 5.11]{BZ} reads as
\begin{equation} \left\| T (I-P_t^{(N)}) \right\|_{L^2(B_1) \to L^2(B_2)} \lesssim \left(1+\frac{d(B_1,B_2)}{r}\right)^{-\tilde{N}} \label{eq:amontrer2} \end{equation}
for the choice $B_Q=I-P_t^{(N)}$.
Following Step 2 of \cite[Corollary 3.6]{B}, it is known that under the assumption that $T$ is bounded on $L^2(M,\mu)$, \eqref{eq:amontrer} implies \eqref{eq:amontrer2}, thus \eqref{eq:amontrer} is sufficient to conclude. 
Equivalently, the desired result can be obtained as a combination of \cite[Proposition 3.25, Lemma 4.12 and Corollary 4.14]{FK}.
\end{proof}

\begin{proposition} \label{prop:off1}
Let $(M,d,\mu, {\mathcal E})$  be a doubling metric  measure   Dirichlet space with a ``carr\'e du champ'' satisfying \eqref{due}. Let $\alpha\in(0,1)$. Assume \eqref{eq:gradient} for some $p_2\in[2,{+\infty})$.  Then there exists $\PP_0=\PP_0(\nu)$ such that for every $\PP\geq \PP_0$ and every $g\in L^\infty(M,\mu)$, the paraproduct $\Pi^{(\PP),1}_g=\Pi^{1}_g$ satisfies the following off-diagonal estimates: for every $r>0$ and every pair of balls $B_1,B_2$ of radius $r$,
\begin{equation} \label{eq:amontrer0}
 \left\| \mathcal{L}^{\alpha/2} \Pi^1_g[\mathcal{L}^{-\alpha/2} Q_{r^2}^{(N)}] \right\|_{L^{p_2}(B_1) \to L^{p_2}(B_2)} \lesssim \left(1+\frac{d(B_1,B_2)}{r}\right)^{-\nu}. \end{equation}
\end{proposition}

\begin{rem} Up to considering a larger parameter $\PP$, we may have off-diagonal estimates at any order. We chose the order $\nu$ for convenience. Such a proposition also holds for the second part $\Pi_g^2$ of the paraproduct and is indeed easier (as shown by Proposition \ref{prop:para2}, this second part is far more easy to handle with than the first part).
\end{rem}

\begin{proof}
Let $\alpha\in(0,1)$ and  $g\in L^\infty(M,\mu)$. Consider the operator
$$ T:= \mathcal{L}^{\alpha / 2} \Pi_g^1 (\mathcal{L}^{-\alpha/2}).$$
Let us fix balls $B_1,B_2$ of radius $r$, a function $f \in L^2(M,\mu)$ supported in $B_2$, and consider an integer $N\geq 2\nu+1$.

We have
$$ T Q_{r^2}^{(N)}(f) = \int_0^{+\infty}\int _0^t  \mathcal{L}^{\alpha / 2} Q_s^{} \left[Q_t \mathcal{L}^{-\alpha/2} Q_{r^2}^{(N)}(f) \cdot P_t g \right] \, \frac{ds}{s} \frac{dt}{t}.$$
By  the definition \eqref{eq:K} of the kernel $K_{\alpha,g}$,
$$ K_{\alpha,g} (s,t):=Q_s \mathcal{L}^{\alpha/2} (Q_t \mathcal{L}^{-\alpha/2} (\, . \,) \cdot P_t g),$$
we get
$$ TQ_{r^2}^{(N)} f =\iint_{0<s\leq t} K_{\alpha,g} (s,t) Q_{r^2}^{(N)} f \, \frac{ds}{s} \frac{dt}{t}.$$
If $r^2 \leq t$, then write
$$ Q_t Q_{r^2}^{(N)} = Q^{(\PP)}_t Q^{(N)}_{r^2}= \left(\frac{r^2}{t}\right)^N Q^{(\PP+N)}_{t} e^{-r^2 \mathcal{L}},$$
so that
$$ K_{\alpha,g} (s,t) Q_{r^2}^{(N)} f = c\left(\frac{r^2}{t}\right)^N \tilde{K}_{\alpha,g} (s,t) Q_{t/2}^{(N)} e^{-r^2 \mathcal{L}} f,$$
with $\tilde{K}_{\alpha,g} (s,t)= Q_s \mathcal{L}^{\alpha/2} (Q_{t/2} \mathcal{L}^{-\alpha/2} (\, . \,) \cdot P_t g)$.

Let $s\leq t$. Abbreviate  $\eps:=\frac{1-\alpha}{2}>0$. Notice that Theorem \ref{thm:KLpLqII} equally applies to $\tilde{K}_{\alpha,g}$. Thus, for large enough integers $\PP$ and $\tilde{N}$, $\tilde{K}_{\alpha,g} (s,t)$ satisfies $L^{p_2}$-$L^{p_2}$ off-diagonal estimates in $\sqrt{t}$ of order $\tilde{N}$ with extra factor $\left(\frac{s}{t}\right)^{\eps}$. On the other hand, Lemma \ref{lem:off} yields  $L^{p_2}$-$L^{p_2}$ off-diagonal estimates in $\sqrt{t}$ for both $Q_{t/2}^{(N)}$ and $e^{-r^2 \mathcal{L}}$ of arbitrary order. Choose $\tilde{N}>\nu$. By Lemma \ref{lem:comp-OD}, we can combine these off-diagonal estimates and obtain
\begin{align*}
&\norm{ K_{\alpha,g}(s,t)[Q_{r^2}^{(N)} f]}_{L^{p_2}(B_1)} \lesssim \left(\frac{r^2}{t}\right)^N \norm{ \tilde{K}_{\alpha,g}(s,t)[Q_{t/2}^{(N)}e^{-r^2\mathcal{L}} f]}_{L^{p_2}(\tilde B_1)}\\
%		 \lesssim& \left(\frac{r^2}{t}\right)^{N} \left(\frac{s}{t}\right)^\eps  \sum_{\tilde B} \left(1+\frac{d^2(\tilde B_1,\tilde B)}{t}\right)^{-\tilde{N}} \|Q_{t/2}^{(N)}e^{-r^2 \mathcal{L}} f\|_{L^{p_2}(\tilde B)} \|g\|_\infty \\
%\lesssim& \left(\frac{r^2}{t}\right)^{N} \left(\frac{s}{t}\right)^\eps  \sum_{\tilde B} \left(1+\frac{d^2(B_1,\tilde B)}{t}\right)^{-\tilde{N}} \left(1+\frac{d^2(B_2,\tilde B)}{t}\right)^{-\tilde{N}} \| f\|_{L^{p_2}(B_2)} \|g\|_\infty \\
		& \lesssim  \left(\frac{r^2}{t}\right)^{N} \left(\frac{s}{t}\right)^\eps  \left(1+\frac{d^2(B_1, B_2)}{t}\right)^{-\tilde{N}} \|f\|_{L^{p_2}(B_2)} \|g\|_\infty.
		\end{align*}	
By integrating in $s\in(0,t)$ and in $t\geq r^2$, one obtains for $N>\tilde{N}$
$$
\int_{r^2}^{+\infty}  \int_0^t\norm{ K_{\alpha,g}(s,t)[Q_{r^2}^{(N)} f] }_{L^{p_2}(B_1)} \, \frac{ds}{s} \frac{dt}{t}
		\lesssim   \left(1+\frac{d^2(B_1,B_2)}{r^2}\right)^{-\tilde{N}} \|f\|_{L^{p_2}(B_2)} \|g\|_\infty . $$

If otherwise $r^2 \geq t$, then write
$$ Q_t Q_{r^2}^{(N)} =  Q^{(\PP)}_t Q^{(N)}_{r^2} = \left(\frac{t}{r^2}\right)^\PP Q^{(\PP+N)}_{r^2}e^{-t \mathcal{L}},$$
so that
$$ K_{\alpha,g} (s,t) Q_{r^2}^{(N)} f = c\left(\frac{t}{r^2}\right)^N \tilde{K}_{\alpha,g} (s,r^2) Q_{r^2/2}^{(N)} e^{-t \mathcal{L}} f.$$
We therefore apply in this case Theorem \ref{thm:KLpLqII} to $\tilde{K}_{\alpha,g} (s,r^2)$. 
Using the same arguments as above and taking into account $r^2\geq t$, we obtain for large enough integers $\PP$ and $\tilde{N}$,
$$
\norm{K_{\alpha,g} (s,t)[Q_{r^2}^{(N)} f] }_{L^{p_2}(B_1)}
		\lesssim \left(\frac{t}{r^2}\right)^\PP \left(\frac{s}{r^2}\right)^\eps   \left(1+\frac{d^2(B_1,B_2)}{r^2}\right)^{-\tilde{N}} \|f\|_{L^{p_2}(B_2)} \|g\|_\infty. $$
Integrating in $s\in(0,t)$ and then in $t\leq r^2$ yields
$$
\int_0^{r^2}  \int_0^t\norm{K_{\alpha,g} (s,t)[Q_{r^2}^{(N)}f] }_{L^{p_2}(B_1)} \, \frac{ds}{s} \frac{dt}{t}
		\lesssim   \left(1+\frac{d^2(B_1,B_2)}{r^2}\right)^{-\tilde{N}} \|f\|_{L^{p_2}(B_2)} \|g\|_\infty . $$

Summarising the above, we have obtained
\begin{equation} \| T Q_{r^2}^{(N)}(f)\|_{L^{p_2}(B_1)} \lesssim 
\left(1+\frac{d^2(B_1,B_2)}{r^2}\right)^{-\tilde{N}} \|f\|_{L^{p_2}(B_2)} \|g\|_\infty, \label{eq:offpi} \end{equation}
where $\PP,N,\tilde{N}$ are large enough integers depending on $\nu$ and $p_2$.
This ends the proof of \eqref{eq:amontrer0}.
\end{proof}

\begin{proof}[Proof of Theorem $\ref{thm:tent-extrap-small}$]
The boundedness of $(g,f)\mapsto\Pi_g(f)$ from $L^\infty(M,\mu) \times \dot{L}^p_{\alpha}(M,\mathcal{L},\mu)$ to $\dot{L}^p_{\alpha}(M,\mathcal{L},\mu)$  is equivalent to the boundedness of
$(g,f)\mapsto \mathcal{L}^{\alpha / 2} \Pi_g \mathcal{L}^{-\alpha/2}f$ from $L^\infty(M,\mu) \times L^p(M,\mu)$ to $L^p(M,\mu)$.
We  have  already seen in Proposition \ref{prop:para2} that it only remains to study the operator
$$ T:= \mathcal{L}^{\alpha / 2} \Pi_g^1 (\mathcal{L}^{-\alpha/2}),$$
and prove its boundedness in $L^p$ for $p\leq 2$.

This is done by the extrapolation argument from Proposition \ref{prop:extrap<2}: indeed by Theorem \ref{gaza}, we already know that $T$ is $L^2$-bounded and Proposition \ref{prop:off1} with $L^2$ Davies-Gaffney estimates yields that \eqref{eq:amontrer0} holds for $p_2=2$. We may also apply Proposition \ref{prop:extrap<2} to $T$ and obtain its $L^p$-boundedness for $p\in(1,2]$.
\end{proof}

\section{Boundedness of the paraproducts for $p\geq p_0$ under $(G_{p_0})$ via extrapolation} \label{sec=gp}

The main results of this section are the two following ones.

\begin{theorem} \label{thm:extrap>22} Let $(M,d,\mu, {\mathcal E})$  be a doubling metric  measure   Dirichlet space with a ``carr\'e du champ'' satisfying \eqref{due}.
Let $\alpha \in (0,1)$ and let $p\in(2,{+\infty})$ with  $1-\alpha>\nu(\frac{1}{2}-\frac{1}{p})$. Then there exists $\PP_0=\PP_0(\nu,p)>0$ such that for every integer $\PP\geq \PP_0$, the paraproduct defined in \eqref{def:paraproduct} is bounded from $L^\infty(M,\mu) \times  \dot{L}^p_{\alpha}(M,\mathcal{L},\mu)$ to $\dot{L}^p_{\alpha}(M,\mathcal{L},\mu)$. We have
\[
	\norm{\Pi_g(f)}_{p,\alpha} \lesssim \norm{f}_{p,\alpha} \norm{g}_{\infty},
\]
and  $A(\alpha,p)$ holds.
\end{theorem}

\begin{theorem} \label{thm:extrap>2} Let $(M,d,\mu, {\mathcal E})$  be a doubling metric  measure   Dirichlet space with a ``carr\'e du champ'' satisfying \eqref{due} and $(G_{p_0})$ for some $p_0\in(2,{+\infty}]$.
Let $\alpha \in (0,1)$ and let $p\in[p_0,{+\infty})$ with  $1-\alpha>\nu(\frac{1}{p_0}-\frac{1}{p})$. Then there exists $\PP_0=\PP_0(\nu,p)>0$ such that for every integer $\PP\geq \PP_0$, the paraproduct defined in \eqref{def:paraproduct} is bounded from $L^\infty(M,\mu) \times  \dot{L}^{p}_{\alpha}(M,\mathcal{L},\mu)$ to $\dot{L}^{p}_{\alpha}(M,\mathcal{L},\mu)$.
We have
\[
	\norm{\Pi_g(f)}_{p,\alpha} \lesssim \norm{f}_{p,\alpha} \norm{g}_{\infty},
\]
and  $A(\alpha,p)$ holds.
\end{theorem}

Using either $L^2$ Davies-Gaffney estimates (which correspond to \eqref{eq:gradient} for $p_2=2$) in combination with Theorem \ref{gaza}, or the fact that $(G_{p_0})$ implies \eqref{eq:gradient} for every $p_2\in[2,p_0)$ in combination with  Theorem \ref{gaza}, the two previous theorems will be a direct consequence of the following one.

\begin{theorem} \label{thm:extrap>2-bis} Let $(M,d,\mu, {\mathcal E})$  be a doubling metric  measure   Dirichlet space with a ``carr\'e du champ'' satisfying \eqref{due}. Assume \eqref{eq:gradient} for some $p_2 \in [2,{+\infty})$ and  let $p>p_2$ with $1-\alpha>\nu(\frac{1}{p_2}-\frac{1}{p})$. There exists $\PP_0=\PP_0(\nu,p)>0$ such that for every integer $\PP\geq \PP_0$, if the paraproduct defined in \eqref{def:paraproduct} is bounded from $L^\infty(M,\mu) \times  \dot{L}^{p_2}_{\beta}(M,\mathcal{L},\mu)$ to $\dot{L}^{p_2}_{\beta}(M,\mathcal{L},\mu)$ for all $\beta\in(0,1)$ then it is bounded from $L^\infty(M,\mu) \times  \dot{L}^{p}_{\alpha}(M,\mathcal{L},\mu)$ to $\dot{L}^{p}_{\alpha}(M,\mathcal{L},\mu)$. We have
\[
	\norm{\Pi_g(f)}_{p,\alpha} \lesssim \norm{f}_{p,\alpha} \norm{g}_{\infty},
\]
and  $A(\alpha,p)$ holds.
\end{theorem}

We are going to prove the previous theorem as an application of the following extrapolation result (\cite{ACDH}, \cite[Theorem 3.13]{AM}).

\begin{proposition} \label{prop:extra-p>2}
Let $T$ be a linear operator and $S$ a sublinear operator. Let $p_2\in[2,{+\infty})$, and assume that $T$ is bounded on $L^{p_2}(M,\mu)$. Assume that $T$ satisfies the following off-diagonal estimates: There exists an integer $N\geq1$, an exponent $\bar p\in(p_2,{+\infty})$ and an exponent  $\tilde{N}>\frac{\nu}{2}$ such that for every pair of balls $B_1,B_2$ of radius $r=\sqrt{t}>0$, we have
\begin{equation} \label{eq:amontrer-f}
 \left\| T Q_t^{(N)} \right\|_{L^{p_2}(B_1) \to L^{p_2}(B_2)} \lesssim \left(1+\frac{d(B_1,B_2)}{r}\right)^{-\tilde{N}} \end{equation}
and
\begin{equation} \left( \aver{B} |T(P_{r^2}^{(N)} f) |^{\bar p} \, d\mu \right)^{1/\bar p} \lesssim  \left(\inf_{x\in B} {\mathcal M}[|S(f)|^{p_2}]\right)^{1/ p_2}.\label{eq:amontrer3} \end{equation}
If, for some $p\in(p_2,\bar{p})$, $S$ is bounded on $L^p(M,\mu)$, then $T$ is bounded on $L^p(M,\mu)$. 
\end{proposition}

\begin{rem} 
\begin{itemize}
\item 
The assumptions in \cite{ACDH}, \cite[Theorem 3.13]{AM} are stated in terms of $L^{p_2}$ off-diagonal estimates for $T(I-P_t^{(N)})$ instead of \eqref{eq:amontrer-f}.
As explained in the proof of Proposition $\ref{prop:extrap<2}$, the $L^{p_2}$ boundedness of $T$ allows us to deduce from \eqref{eq:amontrer-f} such $L^{p_2}$-off-diagonal estimates  for $T(I-P_t^{(N)})$. 
\item For $p\in(p_2,\bar p)$ as above, $T$ is also bounded on the weighted space $L^p(\omega)$ for every weight $\omega \in {\mathbb A}_{p\over p_2} \cap RH_{(\frac{\bar p}{p})'}$.
\end{itemize}
\end{rem}

 As we have already seen in Proposition \ref{prop:para2},  in order to prove Theorem \ref{thm:extrap>2} we only have to study the $L^p$- boundedness of the operator
$$ T:= \mathcal{L}^{\alpha / 2} \Pi_g^1 (\mathcal{L}^{-\alpha/2}),$$
with
$$ \Pi_g^1(f) := \int_0^{+\infty} (I-P_t) \left[P_t g \cdot Q_tf \right] \, \frac{dt}{t}.$$
We recall that the kernel $K_{\alpha,g} $ is defined as
$$ K_{\alpha,g} (s,t):=Q_s \mathcal{L}^{\alpha/2} (Q_t \mathcal{L}^{-\alpha/2} (\, . \,) \cdot P_t g),$$
hence
$$ T  = \int_0^{+\infty} \int_0^t  K_{\alpha,g} (s,t) \, \frac{ds}{s} \frac{dt}{t}.$$
As a direct application of Lemma \ref{lem:orthogonality}, we have the following reduction.

\begin{lemma} \label{lemme}
Define the quadratic functional 
$$ U(f) := \left( \int_0^{+\infty} \left| \int_s^{+\infty} \tilde K_{\alpha,g} (s,t)[ \tilde Q_t f] \, \frac{dt}{t} \right|^2 \, \frac{ds}{s} \right)^{1 / 2},$$
where $\tilde Q_s := (Q_s)^{1/2}$ and $\tilde K(s,t) := \tilde Q_s \mathcal{L}^{\alpha/2} (\tilde Q_t \mathcal{L}^{-\alpha/2} (\, . \,) \cdot P_t g)$, so that 
$$ K_{\alpha,g} (s,t) = \tilde Q_s \tilde K_{\alpha,g} (s,t) \tilde Q_t.$$
Then for $p\in(2,{+\infty})$, the boundedness of $U$ on $L^p(M,\mu)$ implies the boundedness of $T$ on $L^p(M,\mu)$, and we have
$$ \|T\|_{p \to p} \lesssim \|U\|_{p \to p}.$$
\end{lemma}

We are now going to prove Theorem \ref{thm:extrap>2-bis}, based on the extrapolation method in Lebesgue spaces of Proposition \ref{prop:extra-p>2}.

\begin{proof}[Proof of Theorem {\rm \ref{thm:extrap>2-bis}}]

According to Lemma \ref{lemme}, we only have to prove the boundedness of the square functional 
$$ U(f) := \left( \int_0^{+\infty} \left| \int_s^{+\infty} \tilde K_{\alpha,g} (s,t)[ \tilde Q_t f] \, \frac{dt}{t} \right|^2 \, \frac{ds}{s} \right)^{1 / 2},$$
which will be done by applying Proposition \ref{prop:extra-p>2}.

By Proposition \ref{prop:off1}, we already know that \eqref{eq:amontrer-f} holds for $\Pi_g^1$ and the same proof allows us to prove also \eqref{eq:amontrer-f} for the square function $U$ (which is even easier). It remains to check \eqref{eq:amontrer3}.

Fix a ball $B$ of radius $r$ and some integer $N\geq \PP$ satisfying $N\geq \nu+1$. If $\PP$ is large enough, then we may also consider
 $$\widehat{K}_{\alpha,g} (s,t):=\widehat{Q}_s \mathcal{L}^{\alpha/2} (\tilde Q_t \mathcal{L}^{-\alpha/2} (\, . \,) \cdot P_t g),$$
where $\widehat{Q}_s=(\tilde Q_s)^{1/2}$. (We may choose $D \in 4\N$ for convenience). Notice that then both $\widehat{K}_{\alpha,g}$ and $\widehat{Q}_s$ satisfy the same off-diagonal estimates as $K_{\alpha,g}$ and $\tilde Q_s$, respectively. By definition, we have
$$ \tilde K_{\alpha,g} = \widehat{Q}_s \widehat{K}_{\alpha,g}.$$
If $s\leq t\leq r^2$, then
\begin{equation} \label{eq:Qtr}
Q_t P_{r^2}^{(N)} = (t\mathcal{L})^{\PP} e^{-t \mathcal{L}} P_{r^2}^{(N)} = \left(\frac{2t}{r^2}\right)^{\PP} Q_{\frac{r^2}{2}} R_{r^2}^{(N)} e^{-t\mathcal{L}},
\end{equation}
where  $R_{r^2}^{(N)}e^{-\frac{r^2}{2}\mathcal{L}} = P_{r^2}^{(N)}$ as defined in Remark \ref{Pt-rem}, and $R_{r^2}^{(N)}$ satisfies the same off-diagonal estimates as $P_{r^2}^{(N)}$.
Consequently,
$$\tilde K_{\alpha,g} (s,t)[ \tilde Q_t P_{r^2}^{(N)} f] 
%= \left(\frac{2t}{r^2}\right)^{\PP } \tilde K_{\alpha,g} (s,r^2/2)[R_{r^2}^{(N)} e^{-t\mathcal{L}} f] 
= \left(\frac{2t}{r^2}\right)^{\PP }  \widehat{Q}_s \widehat{K}_{\alpha,g} (s,r^2/2) [\tilde Q_{\frac{r^2}{2}} R_{r^2}^{(N)} e^{-t\mathcal{L}} f].$$

Then, from Lemma \ref{lem:off} we kow that $\widehat{Q}_s$ satisfies $L^{p_2}$-$L^p$ off-diagonal estimates at scale $r$ with an extra factor $\left(\frac{r^2}{s}\right)^{\frac{\nu}{2}(\frac{1}{p_2}-\frac{1}{p})}$. Moreover, Theorem \ref{thm:KLpLqII} yields that  $\widehat{K}_{\alpha,g} (s,r^2/2)$ also satisfies $L^{p_2}$-$L^{p_2}$ off-diagonal estimates at scale $r$ with a factor $\left(\frac{s}{r^2}\right)^{\frac{1-\alpha}{2}}$.  Lemma \ref{lem:off} implies $L^{p_2}$-$L^{p_2}$ off-diagonal estimates at scale $r$ for $\tilde Q_{\frac{r^2}{2}}$,  $R_{r^2}^{(N)}$ and  $e^{-t\mathcal{L}}$. All of these off-diagonal estimates are of an order which can be chosen as large as we want, up to choosing $\PP$ sufficiently large. By composing all these estimates according to Lemma \ref{lem:comp-OD}, it follows
for a large enough $\PP$, 
\begin{align*}
\left( \aver{B} | \tilde K_{\alpha,g} (s,t)[ \tilde Q_t P_{r^2}^{(N)} f]|^{ \bar p} \, d\mu \right)^{1/{ \bar p}} &  \lesssim \left(\frac{t}{r^2}\right)^{\PP} \left(\frac{s}{r^2}\right)^{\frac{1-\alpha}{2}-\frac{\nu}{2}(\frac{1}{p_2}-\frac{1}{\bar p})}   \left(\inf_{x\in B} {\mathcal M}(|f|^{p_2}) \right)^{1/p_2} \|g\|_\infty.
\end{align*}
First applying Minkowski's inequality and then  integrating over $s\leq t\leq r^2$ gives for $1-\alpha> \nu(\frac{1}{p_2}-\frac{1}{\bar p})$
\begin{align*}
\left( \aver{B} \left(\int_0^{r^2} \left| \int_s^{r^2}  | \tilde K_{\alpha,g} (s,t)[ \tilde Q_t P_{r^2}^{(N)} f]| \, \frac{dt}{t} \right|^2 \, \frac{ds}{s} \right)^{\bar p/2} \, d\mu \right)^{1 /  \bar p} \lesssim  \left(\inf_{x\in B} {\mathcal M}(|f|^{p_2}) \right)^{1/p_2} \|g\|_\infty .
\end{align*} 
If $ r^2\leq s\leq t$, then similarly as above, Lemma \ref{lem:off} and Theorem \ref{thm:KLpLqII} yield for $\bar p>p_2$ and for large enough $\PP$ (with $\tilde N$ an exponent eventualy varying from a line to the next one)
\begin{align*}
& \left( \aver{B} | \tilde K_{\alpha,g} (s,t)[ \tilde Q_t P_{r^2}^{(N)} f]|^{ \bar p} \, d\mu \right)^{1/{ \bar p}} \\
 & \hspace{2cm} \lesssim \left(\frac{t}{s}\right)^{\frac{\nu}{2}(\frac{1}{p_2}-\frac{1}{\bar p})} \sum_{j\geq 0} 2^{-j\tilde{N}} \left( \aver{2^j  \tilde B} |  \widehat{K}_{\alpha,g} (s,t)[ \tilde Q_t P_{r^2}^{(N)} f]|^{p_2} \, d\mu \right)^{1/p_2} \\
 & \hspace{2cm} \lesssim  \left(\frac{s}{t}\right)^{\frac{1-\alpha}{2}-\frac{\nu}{2}(\frac{1}{p_2}-\frac{1}{\bar p})}  \left[ \sum_{\ell \geq 0} 2^{-\ell\tilde{N}} \left(\aver{2^\ell \tilde B} | \tilde Q_t P_{r^2}^{(N)} f | ^{2} \, d\mu \right)^{1/ 2}\right] \|g\|_\infty \\
 & \hspace{2cm} \lesssim  \left(\frac{s}{t}\right)^{\frac{1-\alpha}{2}-\frac{\nu}{2}(\frac{1}{p_2}-\frac{1}{\bar p})}  \left[ \sum_{\ell \geq 0} 2^{-\ell \tilde{N}} \left(\aver{2^\ell \tilde B} | \tilde Q_t f | ^{2} \, d\mu \right)^{1/ 2}\right] \|g\|_\infty,
\end{align*}
where $\tilde B= \frac{\sqrt{t}}{r}B$ is the dilated ball, and we used $L^2$ off diagonal estimates for $P_{r^2}^{(N)}$ in the last step.

By Minkowski's inequality, integrating over $s \in (0,t)$, and H\"older's inequality, we get for $1-\alpha> \nu(\frac{1}{p_2}-\frac{1}{\bar p})$
\begin{align*}
&\left( \int_{r^2}^{{+\infty}} \left| \int_{s}^{+\infty} \left( \aver{B} | \tilde K_{\alpha,g} (s,t)[\tilde Q_t P_{r^2}^{(N)} f]|^{ \bar p} \, d\mu \right)^{1/ { \bar p}} \,  \frac{dt}{t} \right|^2 \, \frac{ds}{s}\right)^{1 / 2}\\ 
& \hspace{2cm} \lesssim \left[ \sum_{\ell\geq 0} 2^{-\ell\tilde{N}}  \left(\int_{r^2}^{+\infty} \aver{2^\ell \tilde B} | \tilde Q_t f | ^2  \, \frac{d\mu dt}{t} \right)^{1/ 2}\right] \|g\|_\infty \\
& \hspace{2cm} \lesssim \left(\inf _{x\in B} {\mathcal M}[{\mathcal G}_{N/2}(f)^2](x)\right)^{1 / 2} \|g\|_\infty \\
& \hspace{2cm} \lesssim \left(\inf _{x\in B} {\mathcal M}[{\mathcal G}_{N/2}(f)^{p_2}](x)\right)^{1 / p_2} \|g\|_\infty,
\end{align*}		
where ${\mathcal G}_{N/2}$ is the conical square function associated to $\tilde Q_t$, see Proposition \ref{prop:square-function}.

\medskip
\noindent
If $ s\leq r^2 \leq t$ then by Lemma \ref{lem:off}, for $\bar p>p_2$ 
\begin{align*}
 & \left( \aver{B} | \tilde K_{\alpha,g} (s,t)[ \tilde Q_t P_{r^2}^{(N)} f]|^{\bar p} \, d\mu \right)^{1/{\bar p}} \\
 & \hspace{2cm} \lesssim \left(\frac{r}{\sqrt{s}}\right)^{\nu(\frac{1}{p_2}-\frac{1}{\bar p})} \sum_{j\geq 0} 2^{-j\tilde{N}} \left( \aver{2^j B} | \widehat{K}_{\alpha,g} (s,t)[ \tilde Q_t P_{r^2}^{(N)} f]|^{p_2} \, d\mu \right)^{1/p_2}.
 \end{align*}
By repeating the same argument as before, we obtain
\begin{align*}
& \left( \int_0^{{+\infty}} \left| \int_{s}^{+\infty} \left( \aver{B} | \tilde K_{\alpha,g} (s,t)[\tilde Q_t P_{r^2}^{(N)} f]|^{\bar p} \, d\mu \right)^{1/{\bar p}} \,  \frac{dt}{t} \right|^2 \, \frac{ds}{s}\right)^{1 / 2} \\
& \hspace{2cm} \lesssim \left(\inf _{x\in B} {\mathcal M}[{\mathcal G}_{N/2}(f)^{p_2}](x)\right)^{1 / p_2} \|g\|_\infty , 
\end{align*}
as soon as $1-\alpha>\nu(\frac{1}{p_2}-\frac{1}{\bar p})$.

\medskip
\noindent
Gathering the above estimates, we obtain that the square function $U$ satisfies for ${\bar p}>p_2$ with $1-\alpha> \nu(\frac{1}{p_2}-\frac{1}{\bar p})$
$$ \left( \aver{B} |U(P_{r^2}^{(N)} f) |^{\bar p} \, d\mu \right)^{1/ \bar p} \lesssim \|g\|_\infty \left(\inf_{x\in B} {\mathcal M}[|{\mathcal G}_{N/2}(f)|^{p_2}]\right)^{1/ p_2} + \|g\|_\infty \left(\inf_{x\in B} {\mathcal M}(|f|^{p_2}) \right)^{1/ p_2},$$
where ${\mathcal G}_{N/2}$ is the conical square version. Since the conical square function is bounded on every $L^p$-space (see Proposition \ref{prop:square-function}), we may then extrapolate by using  Proposition \ref{prop:extra-p>2}. We deduce that $U$ is bounded on $L^p$ for every $p\in(p_2,\bar p)$. This holds for every $\bar p>p_2$ and $\alpha\in (0,1)$ such that $1-\alpha> \nu(\frac{1}{p_2}-\frac{1}{\bar p})$ so we conclude that $U$ is bounded on $L^p$ for every $p>p_2$ such that $1-\alpha> \nu(\frac{1}{p_2}-\frac{1}{p})$, which then implies the $\dot L^p_\alpha$-boundedness of the paraproduct $\Pi_g$.
\end{proof}

\section{Boundedness of the paraproducts for $p\geq p_0$ under $(G_{p_0})$ and $(DG_{2})$ via extrapolation} \label{sec p>2}

In this section, we prove stronger results under the additional assumption of a De Giorgi property. The proofs are, as in the previous section, based on $L^p$ extrapolation techniques.

\begin{theorem} \label{thm:tent-extrap-large} Let $(M,d,\mu, {\mathcal E})$  be a doubling metric  measure   Dirichlet space with a ``carr\'e du champ'' satisfying \eqref{due}.  Let $2<p_0 \leq {+\infty}$ and  assume $(G_{p_0})$ with $({DG}_{2,\kappa})$ for some $\kappa\in(0,1)$. Then the paraproduct defined in \eqref{def:paraproduct} is bounded from $L^\infty(M,\mu) \times  \dot{L}^p_{\alpha}(M,\mathcal{L},\mu)$ to $\dot{L}^p_{\alpha}(M,\mathcal{L},\mu)$ for every $\alpha\in(0,1-\kappa)$ and $p\in(2,{+\infty})$. We have
\[
	\norm{\Pi_g(f)}_{p,\alpha} \lesssim \norm{f}_{p,\alpha} \norm{g}_{\infty}.
\]
Therefore  $A(\alpha,p)$ holds.
\end{theorem}

As a consequence, we obtain our main result of this section.

\begin{theorem} \label{thm:tent-extrap-large-3} Let $(M,d,\mu, {\mathcal E})$  be a doubling metric  measure   Dirichlet space with a ``carr\'e du champ'' satisfying \eqref{due}. Let $2<p_0 \leq {+\infty}$ and  assume $(G_{p_0})$ with $({DG}_{2,\kappa})$ for some $\kappa\in(0,1)$ (and also $\kappa<\frac{\nu}{p_0}$, else the result is implied by Theorem \ref{thm:extrap>2}). Then for $p\in(1,{+\infty})$ the paraproduct defined in \eqref{def:paraproduct} is bounded from $L^\infty(M,\mu) \times  \dot{L}^p_{\alpha}(M,\mathcal{L},\mu)$ to $\dot{L}^p_{\alpha}(M,\mathcal{L},\mu)$ for every $\alpha\in(0,\gamma_p)$ 
with
$$\gamma_p:=\left\{ \begin{array}{l}
1, \quad \textrm{if $p\leq p_0$} \vsp \\
1-\kappa\left(1 -\frac{p_0}{p}\right), \quad \textrm{if $p\geq p_0$.} 
\end{array}
\right. $$
We have
\[
	\norm{\Pi_g(f)}_{p,\alpha} \lesssim \norm{f}_{p,\alpha} \norm{g}_{\infty}.
\]
Therefore  $A(\alpha,p)$ holds.
\end{theorem}

We postpone the proof of Theorem \ref{thm:tent-extrap-large} 
to the end of this section, and we now prove Theorem \ref{thm:tent-extrap-large-3} as a consequence.

\begin{proof}[Proof of Theorem $\ref{thm:tent-extrap-large-3}$]  The case $p<p_0$ has already been studied in Theorem \ref{gaza}, so we only focus on the case $p\in[p_0,{+\infty})$.
Fix  $g\in L^\infty$. For $z$ a complex number with $\Re(z)\in(0,1)$,  define
$$ T^z:= \mathcal{L}^{z / 2} \Pi_g (\mathcal{L}^{-z/2}).$$

Theorem \ref{gaza} shows that $T^{\alpha}$ is $L^{p}$-bounded for every $\alpha \in(0,1)$ and every $p\in(2,p_0)$. Then by combining with imaginary powers of $\mathcal{L}$, which are $L^{p}$-bounded (see Proposition \ref{prop:imaginary}), we deduce that for every $\alpha\in(0,1)$ and $\beta\in \R$, $T^{\alpha+i\beta}$ is $L^{p}$-bounded and
$$ \sup_{\beta\in \R}  (1+|\beta|)^{-s}  \| T^{\alpha+i \beta} \|_{p\to p} \lesssim C^0_\alpha,$$
for some constant $C^0_\alpha$ and any $s>\nu$.

Moreover, Theorem \ref{thm:tent-extrap-large} shows that $T^{\alpha}$ is $L^{p}$-bounded for every $\alpha \in(0,1-\kappa)$ and every $p\in (2,{+\infty})$. Then by using  Proposition \ref{prop:imaginary} we deduce that, for every $\alpha\in(0,1-\kappa)$ and $\beta\in \R$, $T^{\alpha+i\beta}$ is $L^{p}$-bounded and
$$ \sup_{\beta\in \R} (1+|\beta|)^{-s} \| T^{\alpha+i \beta} \|_{p\to p} \lesssim C^1_\alpha,$$
for some constant $C^1_\alpha$ and any $s>\nu$.

We then conclude the proof by applying Stein's complex interpolation method (\cite[Theorem 1]{Stein}) to the family $(T^z)_z$.
\end{proof}

\begin{proof}[Proof of Theorem $\ref{thm:tent-extrap-large}$]

By interpolating assumption $(G_{p_0})$ with $L^2$-$L^2$ Davies-Gaffney estimates, \eqref{eq:gradient} holds for every $p_2=p_1\in(2,p_0)$.
We reproduce the same reasoning as done for Theorem \ref{thm:extrap>2-bis}, relying on the extrapolation result Proposition \ref{prop:extra-p>2}.

So as previously, according to Lemma \ref{lemme}, we only have to prove the boundedness of the quadratic functional 
$$ U(f) := \left( \int_0^{+\infty} \left| \int_s^{+\infty} \tilde K_{\alpha,g} (s,t)[ \tilde Q_t f] \, \frac{dt}{t} \right|^2 \, \frac{ds}{s} \right)^{1 / 2},$$
which will be done by applying Proposition \ref{prop:extra-p>2}.

Fix a ball $B$ of radius $r$, and consider $U[P_{r^2}^{(N)}f]$ for some large enough integer $N\geq \PP/2$.

\medskip
\noindent
If $s\leq t\leq r^2$, then as in \eqref{eq:Qtr}
$$Q_t P_{r^2}^{(N)} 
 = \left(\frac{2t}{r^2}\right)^{\PP} Q_{\frac{r^2}{2}} R_{r^2}^{(N)} e^{-t\mathcal{L}},$$ 
and consequently
$$ \tilde K_{\alpha,g} (s,t)[\tilde Q_t P_{r^2}^{(N)}f] = \left(\frac{2t}{r^2}\right)^{\PP} \tilde K_{\alpha,g} (s,r^2/2)[ R_{r^2}^{(N)} e^{-t\mathcal{L}}f].$$
Hence, combining what was done for Theorem \ref{thm:extrap>2-bis} and Theorem \ref{thm:KLpLqIII} gives 
$$
 \|\tilde K_{\alpha,g} (s,t)[ \tilde Q_t P_{r^2}^{(N)} f]\|_{L^{\infty}(B)}
		\lesssim   \left(\frac{t}{r^2}\right)^{\PP} \left(\frac{s}{r^2}\right)^{\frac{1-\alpha-\kappa'}{2}}   \left(\inf_{x\in B} {\mathcal M}(|f|^2)(x) \right)^{1/ 2} \|g\|_\infty. $$
By integrating over $s\leq t\leq r^2$, one obtains
\begin{align*}
\left( \int_0^{r^2} \left| \int_s^{r^2} \sup_B | \tilde K_{\alpha,g} (s,t)[  \tilde Q_t P_{r^2}^{(N)} f]| \, \frac{dt}{t} \right|^2 \, \frac{ds}{s} \right)^{1 / 2} \lesssim  \left(\inf_{x\in B} {\mathcal M}(|f|^2)(x) \right)^{1/2} \|g\|_\infty .
\end{align*} 

\medskip
\noindent
If $ r^2\leq t$, then  Theorem \ref{thm:KLpLqIII} (with $\kappa'\in(\kappa,1)$) similarly yields
$$
\sup_{B} | \tilde K_{\alpha,g} (s,t)[ e^{-r^2\mathcal{L}} f]|
		\lesssim \left(\frac{s}{t}\right)^{\frac{1-\alpha-\kappa'}{2}}  \left[ \sum_{\ell \geq 0} 2^{-\ell M} \left(\aver{2^\ell \tilde B_{\sqrt{t}}} | \tilde Q_t  f | ^2 \, d\mu \right)^{1/ 2}\right] \|g\|_\infty , $$
		where $\tilde B_{\sqrt{t}}= \frac{\sqrt{t}}{r}B$ is the dilated ball and $M$ a large enough integer.
By integrating for $s\leq t$ and using Cauchy-Schwarz inequality, we get as soon as $1-\alpha-\kappa'>0$
\begin{align*}
&\left( \int_0^{{+\infty}} \left| \int_{s}^{+\infty} \sup_B | \tilde K_{\alpha,g} (s,t)[e^{-r^2\mathcal{L}} f]|  \, \frac{dt}{t} \right|^2 \, \frac{ds}{s}\right)^{1 / 2}\\ &
\lesssim \left[ \sum_{\ell\geq 0} 2^{-\ell M}  \left(\int_{r^2}^{+\infty} \aver{2^\ell \tilde B} | \tilde Q_t f | ^2  \, \frac{d\mu dt}{t} \right)^{1/ 2}\right] \|g\|_\infty \\
& \lesssim \left(\inf _{x\in B} {\mathcal M}[|{\mathcal G}_{N/2}(f)|^2](x)\right)^{1 / 2} \|g\|_\infty , 
\end{align*}		
where ${\mathcal G}_{N/2}(f)$ is the conical square function associated with $\tilde Q_tf$, see Proposition \ref{prop:square-function}.

\medskip
\noindent
Conclusion: by combining the previous estimates we obtain that the square function $U$ satisfies
$$ \|U(P_{r^2}^{(N)} f) \|_{L^\infty(B)} \lesssim \|g\|_\infty \left(\inf_{x\in B} {\mathcal M}[|{\mathcal G}_{N/2}(f)|^2](x)\right)^{1/ 2} + \|g\|_\infty \left(\inf_{x\in B} {\mathcal M}(|f|^2)(x) \right)^{1/ 2},$$
as soon as $1-\alpha>\kappa$ (in which case there exists $\kappa'<\kappa$ with $1-\alpha-\kappa'>0$).

We can then apply the extrapolation result Proposition \ref{prop:extra-p>2}.  Since ${\mathcal G}_{N/2}(f)$ is bounded on $L^p$ according to Proposition \ref{prop:square-function}, we obtain that $U$ is bounded and therefore $T$ on $L^p(M,\mu)$ for every $p\in(2,{+\infty})$. All these computations require $1-\kappa>\alpha$, which is the main condition.
\end{proof}

\section{The case $p>2$ via oscillation} \label{sec:osci}

\begin{definition} \label{def:S} Let $\alpha>0$  and $\rho\in [1,\infty)$.
For $f\in L^1_{\loc}(M,\mu)$, $\alpha>0$ and $x\in M$, we consider the quadratic functional
$$ S_{\alpha}^\rho f(x):=\left(\int_{0}^{{+\infty}} \left[\frac{1}{r^{\alpha}} \rho\text{-}\osc_{B(x,r)}(f)\right]^2\frac{dr}{r}\right)^{{1/2}},
$$
where for a ball $B$, $ \rho\text{-}\osc_{B}$ denotes the $L^\rho$-oscillation defined by
$$ \rho\text{-}\osc_{B}(f):= \left(\aver{B} \left|f- \aver{B}f \, d\mu \right|^\rho \, d\mu \right)^{1/ \rho}.$$
\end{definition}

We are going to prove the two following results.

\begin{theorem} \label{thm:osci}
Let $(M,d,\mu, {\mathcal E})$  be a doubling metric  measure   Dirichlet space with a ``carr\'e du champ'' satisfying \eqref{due}. Assume   $(H^\eta)$ for some $\eta\in(0,1]$. Let $\alpha\in(0,\eta)$ and $p\in(1,{+\infty})$.
Then the paraproduct defined in \eqref{def:paraproduct} is bounded from $L^\infty(M,\mu) \times  \dot{L}^p_{\alpha}(M,\mathcal{L},\mu)$ to $\dot{L}^p_{\alpha}(M,\mathcal{L},\mu)$.  We have
\[
	\norm{\Pi_g(f)}_{p,\alpha} \lesssim \norm{f}_{p,\alpha} \norm{g}_{\infty}.
\]
It follows that  $A(\alpha,p)$ holds.
Moreover the space $\dot{L}^p_{\alpha}(M,\mathcal{L},\mu)$ is characterized by $S_\alpha$-functionals: for $1\leq \rho<\min(2,p)$ and $\alpha\in(0,\eta)$ we have
$$ \|f\|_{\dot L^p_\alpha} \simeq \|S^{\rho}_\alpha(f)\|_{p}.$$
In particular, $E(\alpha,p)$ holds.

\end{theorem}

\begin{theorem} \label{thm:osci2} Let $(M,d,\mu, {\mathcal E})$  be a doubling metric  measure   Dirichlet space with a ``carr\'e du champ''. 
Assume the combination $(G_{p_0})$ with $(P_{p_0})$ for some $p_0\in(2,{+\infty})$. Then the space $\dot{L}^p_{\alpha}(M,\mathcal{L},\mu)$ is characterized by $S_\alpha^\rho$-functionals: for $\rho<2$ closed enough to $2$, every $p\in(2,p_0)$ and $\alpha\in(0,1)$ we have
$$ \|f\|_{\dot L^p_\alpha} \simeq \|S^{\rho}_\alpha(f)\|_{p}.$$
In particular, $E(\alpha,p)$ holds.
\end{theorem}

\begin{rem} Under $(G_{p_0})$, in the considered range $p\in(2,p_0)$ and $\alpha\in(0,1)$, it is already known that the paraproducts are bounded in the Sobolev space and so $A(\alpha,p)$ holds (see Theorem \ref{gaza}).
\end{rem}

Let us observe that for two test functions $f,g$, every ball $B$ and every exponent $\rho\geq 1$, one has
$$ \rho\text{-}\osc_{B}(fg) \leq \rho\text{-}\osc_{B}(f)\|g\|_\infty + \|f\|_\infty\rho\text{-}\osc_{B}(g),$$
and for every Lipschitz function $F$ 
$$  \rho\text{-}\osc_{B}(F(f)) \lesssim \|F\|_{\textrm{Lip}} \rho\text{-}\osc_{B}(f).$$
Consequently, as soon as the Sobolev norm $\dot L^p_{\alpha}$ is characterized by a quadratic functional $S_\alpha^\rho$ for some $\rho\in[1,{+\infty}]$,  property $A(\alpha,p)$ is also satisfied and the following sharp chain rule.

\begin{coro} \label{cor:chainrule} Under the assumptions of Theorems $\ref{thm:osci}$ or $\ref{thm:osci2}$,  for  $p$ and $\alpha$ in their respective ranges, and every Lipschitz function $F$, the map $ f\rightarrow F(f)$ is bounded in $\dot{L}^p_{\alpha}(M,\mathcal{L},\mu)$ and
$$ \|F(f)\|_{\dot L^p_\alpha} \lesssim \|F\|_{\textrm{Lip}} \|f\|_{\dot L^p_\alpha}.$$
\end{coro}

We will see in Proposition \ref{prop:nec}, that such a characterisation of Sobolev norms (through quadratic functional) cannot hold in a systematic way, since some of them require the Poincar\'e inequality $(P_2)$.

Here the sharpness refers to the fact that we only require a Lipschitz control of the nonlinearity $F$. We refer the reader to Section \ref{sec:chainrule} for a chain rule under weaker assumptions on the ambient space $(M,d,\mu, {\mathcal E})$ but more regular nonlinearities $F$.

We are going to simultaneously prove Theorems  \ref{thm:osci} and \ref{thm:osci2}  in the two following sections: in Section \ref{bosci}  the statements concerning paraproducts  and in Section \ref{quadra} the statements concerning the functionals $S_\alpha^\rho$.  Theorem \ref{thm:osci} is the combination of Propositions \ref{prop:L2Lp-osci} and \ref{prop:osci}, whereas Theorem \ref{thm:osci2} follows from Proposition \ref{prop:poinc2}.

\subsection{Boundedness of paraproducts via oscillation}\label{bosci}

We first recall that according to Lemma \ref{lemme}, to prove the boundedness of the paraproduct it is enough to prove the $L^p$-boundedness of the square function
$$ U(f) := \left( \int_0^{+\infty} \left| \int_s^{+\infty} \tilde K_{\alpha,g} (s,t)[ \tilde Q_t f] \, \frac{dt}{t} \right|^2 \right)^{1 / 2},$$
where $\tilde Q_s := (Q_s)^{1/2}$ and $\tilde K(s,t) := \tilde Q_s \mathcal{L}^{\alpha/2} (\tilde Q_t \mathcal{L}^{-\alpha/2} (\, . \,) \cdot P_t g)$, so that
$$ K_{\alpha,g} (s,t) = \tilde Q_s \tilde K_{\alpha,g} (s,t) \tilde Q_t.$$

\begin{proposition} \label{prop:L2Lp-osci} 
Let $(M,d,\mu, {\mathcal E})$  be a doubling metric  measure   Dirichlet space with a ``carr\'e du champ'' satisfying \eqref{due}. Assume $(H^\eta)$ for some $\eta\in(0,1]$.
Then for every $\alpha<\lambda<\eta$, the kernel $\tilde K_{\alpha,g}$ satisfies for $s\leq t$ the pointwise estimate
$$ \tilde K_{\alpha,g}(s,t) [h](x_0) \lesssim \|g\|_{\infty}  \left(\frac{s}{t}\right)^{\frac{\lambda-\alpha}{2}}  {\mathcal M}(h)(x_0) \qquad \forall x_0\in M.$$
\end{proposition}

\begin{proof} Let $x_0 \in M$. We have 
\begin{align*}
	\tilde K_{\alpha,g} (s,t)[h] (x_0)= \left(\frac{s}{t}\right)^{-\alpha/2} \tilde Q_s (s\mathcal{L})^{\alpha/2} (\tilde Q_t (t\mathcal{L})^{-\alpha/2} h \cdot P_t g) (x_0).
\end{align*}

Consider $B_{\sqrt{s}}$ the ball of centre $x_0$ and radius $\sqrt{s}$. Then by linearity 
\begin{align*}
&\left| Q_s (s\mathcal{L})^{\alpha/ 2}  (Q_t (t\mathcal{L})^{-\alpha/2} h \cdot P_t g) (x_0)\right|  \\ \lesssim &\left|Q_s (s\mathcal{L})^{\alpha / 2}  \left[\left(Q_t (t\mathcal{L})^{-\alpha/2} h - \aver{B_{\sqrt{s}}}Q_t (t\mathcal{L})^{-\alpha/2} h \, d\mu \right) \cdot P_t g\right] (x_0)\right| \\
& + \left| \aver{B_{\sqrt{s}}}Q_t (t\mathcal{L})^{-\alpha/2} h \, d\mu \right| \left|Q_s (s\mathcal{L})^{\alpha / 2} [P_t g] (x_0)\right|,
\end{align*}
which gives us two terms $I$ and $II$.

The second term is the easiest, since $Q_s (s\mathcal{L})^{\alpha / 2} P_t = (\frac{s}{t})^{\PP+ \alpha/2} e^{-s\mathcal{L}} (t\mathcal{L})^{\PP+\alpha/2} P_t$, so due to the $L^\infty$-boundedness of $e^{-s\mathcal{L}} (t\mathcal{L})^{\PP+\alpha/2} P_t$, we deduce that
\begin{align*}
 II & \lesssim \left( \frac{s}{t} \right)^{\PP+\alpha/2} \left| \aver{B_{\sqrt{s}}}Q_t (t\mathcal{L})^{-\alpha/2} h \, d\mu \right|  \|g\|_{\infty} \\
 & \lesssim \left( \frac{s}{t} \right)^{\PP+\alpha/2} {\mathcal M}[h](x_0) \|g\|_{\infty},
\end{align*}
where we used Lemma \ref{lem:off} (item 1) in the last step.

For the first term $I$, we use the $L^{p_0}$-$L^{\infty}$ off-diagonal estimates for $Q_s$ (Lemma \ref{prop:kernel-est}) and we get
\begin{align}
I \lesssim \|g\|_{\infty} \sum_{\ell \geq 0} 2^{-\ell(\PP+\alpha/2)}  |2^\ell B_{\sqrt{s}}|^{-\frac{1}{p_0}} \left\|Q_t (t\mathcal{L})^{-\alpha/2} h - \aver{B_{\sqrt{s}}}Q_t (t\mathcal{L})^{-\alpha/2} h\right\|_{L^{p_0}(2^\ell B_{\sqrt{s}})}. \label{eq:oscil2}
\end{align}
Since $(H^\eta)$ self-improves into $(\overline{H}_{1,\infty}^\lambda)$ for $\lambda\in(\alpha,\eta)$ (see item 2 of Proposition \ref{prop:bcf1}), one has with $Q_t(t\mathcal{L})^{-\alpha/2}= 2^{\PP} e^{-\frac{t}{2}\mathcal{L}} Q_{t/2}^{(\PP-\alpha/2)}$ that for every integer $k\in\{0,..,\ell\}$
\begin{align*}
  \infty\text{-}\osc_{2^k B_{\sqrt{s}}} (Q_t (t\mathcal{L})^{-\alpha/2} h) & \lesssim 2^{k\lambda} \left(\frac{s}{t}\right)^{\lambda / 2} \sup_{j\geq 0} \left( \aver{B(x_0,2^j \sqrt{t})} |Q_{t/2}^{(\PP-\alpha/2)} h| \, d\mu \right) \\
  & \lesssim 2^{k\lambda} \left(\frac{s}{t}\right)^{\lambda / 2} \sup_{j\geq 0} \left(\aver{2^j \tilde B_{\sqrt{t}}} |h| \, d\mu\right) \\
  & \lesssim 2^{k\lambda} \left(\frac{s}{t}\right)^{\lambda / 2} \calM[h](x_0)
\end{align*}
where we have used Lemma \ref{prop:kernel-est} to  estimate pointwise the kernel of $Q_{t/2}^{(\PP-\alpha/2)}$.

Since
\begin{align*}
  &\left\|Q_t (t\mathcal{L})^{-\alpha/2} h - \aver{B_{\sqrt{s}}}Q_t (t\mathcal{L})^{-\alpha/2} h\right\|_{L^{p_0}(2^\ell B_{\sqrt{s}})} \\ 
  & \qquad \lesssim \sum_{k=0}^\ell  \left\|Q_t (t\mathcal{L})^{-\alpha/2} h - \aver{2^k B_{\sqrt{s}}}Q_t (t\mathcal{L})^{-\alpha/2} h\right\|_{L^{p_0}(2^k B_{\sqrt{s}})} \\
 & \qquad \lesssim \sum_{k=0}^\ell |2^\ell B_{\sqrt{s}}|^{1/p_0} \infty\text{-}\osc_{2^k B_{\sqrt{s}}} (Q_t (t\mathcal{L})^{-\alpha/2} h),
\end{align*}
it follows that
\begin{align*}
  \left\|Q_t (t\mathcal{L})^{-\alpha/2} h - \aver{B_{\sqrt{s}}}Q_t (t\mathcal{L})^{-\alpha/2} h\right\|_{L^{p_0}(2^\ell B_{\sqrt{s}})} \lesssim 2^{\ell \lambda} |2^\ell B_{\sqrt{s}}|^{1/ {p_0}} \left(\frac{s}{t}\right)^{\lambda / 2}  \calM [h](x_0). 
\end{align*}
Finally, since $\PP>\nu+1$ (so $\PP+\alpha/2>1>\lambda$)
\begin{align*}
I & \lesssim \left(\frac{s}{t}\right)^{{\lambda / 2}} \|g\|_{\infty}    \left( \sum_{\ell \geq 0} 2^{-\ell(\PP+\alpha/2)} \ell 2^{\ell \lambda} \right) \calM [h](x_0) \\
& \lesssim \left(\frac{s}{t}\right)^{{\lambda / 2}} \|g\|_{\infty} \calM [h](x_0).
\end{align*}
 Hence
$$ \tilde K_{\alpha,g}(s,t) (h)(x_0) \lesssim \left(\frac{s}{t}\right)^{\frac{\lambda-\alpha}{2}} \|g\|_{\infty} \calM [h](x_0).$$
\end{proof}

We can now conclude the proof of  the statements about paraproducts in Theorem \ref{thm:osci}. %and \ref{thm:osci2}.

\begin{proof}[Proof of Theorem $\ref{thm:osci}$] We use Proposition \ref{prop:L2Lp-osci}, so that we have the following pointwise bound of the square function $U$ (as soon as $\alpha<\lambda$):
\begin{align*}
U(f) & \lesssim  \|g\|_{\infty}  \left( \int_0^{+\infty} \left| \int_s^{+\infty}    \left(\frac{s}{t}\right)^{\frac{\lambda-\alpha}{2}}  {\mathcal M}[\tilde Q_t f] \, \frac{dt}{t} \right|^2 \right)^{1 / 2} \\
& \lesssim \|g\|_{\infty}  \left( \int_0^{+\infty} \left|  {\mathcal M}[\tilde Q_t f]   \right|^2 \, \frac{dt}{t} \right)^{1 / 2} .
\end{align*}
By using the Fefferman-Stein inequality (see Proposition \ref{prop:FS}) and the $L^p$-boundedness of the horizontal square functionals (see Proposition \ref{prop:square-function}), we deduce that $U$ is $L^p$-bounded, which implies (see Lemma \ref{lemme}) the $\dot L^p_\alpha$-boundedness of the paraproduct.
\end{proof}

\subsection{Characterisation of Sobolev norms via $S_\alpha$ }\label{quadra}

The following statement can be found in \cite[Section 2.1.1]{CRT} and \cite[Section 5.2]{BBR}. The proof works in our setting.

\begin{proposition} \label{lemma:salpha} Assume \eqref{due}. Suppose $p,\rho\in(1,{+\infty})$, $\alpha>0$, and let $f \in L^1_{\loc}(M,\mu)$.
If $S_\alpha^\rho (f)\in L^p(M,\mu)$, then $f\in \dot{L}^p_{\alpha}(M,\mathcal{L},\mu)$ and
$$ \|f\|_{\dot L^p_\alpha} \lesssim \|S_\alpha^\rho (f)\|_{p}.$$
\end{proposition}

The proof of the reverse inequality in  \cite[Section 2.1.2]{BBR} uses pointwise gradient estimates. We are now going to observe that 
the weaker assumption $(H^\eta)$ is in fact sufficient, as noted already in \cite[p.333]{CRT}.
Without loss of generality, we assume $N(\mathcal{L})=\{0\}$ in the following.

\begin{proposition} \label{prop:osci} Assume $(H^\eta)$ for some $\eta\in(0,1]$. Fix $\alpha\in(0,\eta)$. Then for every $p\in(1,{+\infty})$ with $\rho<\min(2,p)$ and every $f\in \dot{L}^p_{\alpha}(M,\mathcal{L},\mu)$, we have
$$ \|f\|_{\dot L^p_\alpha} \simeq \|S^{\rho}_\alpha(f)\|_{p}.$$
\end{proposition}

\begin{proof} Due to Proposition \ref{lemma:salpha}, it only remains to prove that
$$ \|S^{\rho}_\alpha(f)\|_{p} \lesssim \|f\|_{\dot L^p_\alpha}.$$
We first decompose the identity with the semigroup as  
\begin{align*}
 f& =-\int_0^{{+\infty}}\frac{ \partial}{\partial{t}}(e^{-t\mathcal{L}}f) dt  =\int_0^{{+\infty}}(t\mathcal{L}) e^{-t\mathcal{L}}f \, \frac{dt}{t} \\
  & =\sum_{n=-\infty}^{{+\infty}} \int_{2^n}^{2^{n+1} } (t\mathcal{L})e^{-t\mathcal{L}}f \, \frac{dt}{t},
\end{align*}
and define the piece at scale $2^{n}$ as
$$f_n:=\int_{2^n}^{2^{n+1} } (t\mathcal{L})e^{-t\mathcal{L}}f \, \frac{dt}{t}.$$
Then fix $x\in M$ and a scale $r>0$. We have
$$ \rho \text{-}\osc_{B(x,r)} (f_n) \leq \int_{2^n}^{2^{n+1} } \rho \text{-} \osc_{B(x,r)}[(t\mathcal{L})e^{-t\mathcal{L}}f] \, \frac{dt}{t}.$$
Using $(H^{\eta})$, which implies $(\overline{H}^{\lambda}_{\rho,\rho})$ for some $\lambda\in(\alpha,\eta)$ (see item 2 of Proposition \ref{prop:bcf1}) we know that if $r\lesssim \sqrt{t}$ then
$$ \rho \text{-} \osc_{B(x,r)}[(t\mathcal{L})e^{-t\mathcal{L}}f] \lesssim \left(\frac{r}{\sqrt{t}}\right)^\lambda {\mathcal M}_{\rho}[(t\mathcal{L})e^{-\frac{t}{2}\mathcal{L}}f](x).$$
So if $r\lesssim 2^{\frac{n}{2}}$, then we deduce by the Cauchy-Schwarz inequality that
\begin{equation} \rho \text{-} \osc_{B(x,r)} (f_n) \lesssim  \left(r 2^{-\frac{n}{2}}\right)^\lambda \left(\int_{2^n}^{2^{n+1} } \left|{\mathcal M}_{\rho}[(t\mathcal{L})e^{-\frac{t}{2}\mathcal{L}}f](x)\right|^2 \, \frac{dt}{t}\right)^{1/ 2}. \label{eq:osci1} \end{equation}
Moreover, if $2^{n/ 2} \lesssim r$, we use
$$ \rho \text{-} \osc_{B(x,r)}[(t\mathcal{L})e^{-t\mathcal{L}}f] \lesssim  {\mathcal M}_{\rho}[(t\mathcal{L})e^{-t\mathcal{L}}f](x)$$
which yields
\begin{equation} \rho \text{-}\osc_{B(x,r)} (f_n) \lesssim \left(\int_{2^n}^{2^{n+1} } \left|{\mathcal M}_{\rho}[(t\mathcal{L})e^{-t\mathcal{L}}f](x)\right|^2 \, \frac{dt}{t}\right)^{1/ 2}. \label{eq:osci2} \end{equation}
Then it follows that 
\begin{align*}
S_{\alpha}^{\rho} f(x)^2 &= \int_0^{{+\infty}} \left[\frac{1}{r^{\alpha}} \rho \text{-}\osc_{B(x,r)}(f) \right]^2\frac{dr}{r}
\\
& \lesssim \int_0^{{+\infty}} \left[\sum_{n\in {\mathbb Z}}\frac{1}{r^{\alpha}} \rho \text{-} \osc_{B(x,r)}(f_n) \right]^2\frac{dr}{r}.
\end{align*}
Using \eqref{eq:osci1} and \eqref{eq:osci2}, one has
\begin{align*}
S_{\alpha}^{\rho} f(x)^2  & \lesssim \int_0^{{+\infty}} \left[\sum_{2^n \leq r^2}\frac{1}{r^{\alpha}} \left(\int_{2^n}^{2^{n+1} } \left|{\mathcal M}_{\rho}[(t\mathcal{L})e^{-t\mathcal{L}}f](x)\right|^2 \, \frac{dt}{t}\right)^{1/ 2}\right]^2 \frac{dr}{r} \\
 & + \int_0^{{+\infty}} \left[\sum_{r^2\leq 2^n }\frac{1}{r^{\alpha}} \left(r 2^{-\frac{n}{2}}\right)^\lambda \left(\int_{2^n}^{2^{n+1} } \left|{\mathcal M}_{\rho}[(t\mathcal{L})e^{-\frac{t}{2}\mathcal{L}}f](x)\right|^2 \, \frac{dt}{t}\right)^{1/ 2}\right]^2 \, \frac{dr}{r}.
\end{align*}
Using Schur's lemma (or see \cite[p. 300]{CRT}),  for $\alpha< \lambda$,
\begin{align*}
S_{\alpha}^{\rho} f(x)^2  & \lesssim \sum_{n\in {\mathbb Z}}  2^{-n\alpha} \int_{2^n}^{2^{n+1} } \left|{\mathcal M}_{\rho}[(t\mathcal{L})e^{-t\mathcal{L}}f](x)\right|^2 \, \frac{dt}{t}  \\
 & +  \sum_{n\in {\mathbb Z}}  2^{-n\alpha} \int_{2^n}^{2^{n+1} } \left|{\mathcal M}_{\rho}[(t\mathcal{L})e^{-\frac{t}{2}\mathcal{L}}f](x)\right|^2 \, \frac{dt}{t},
\end{align*}
which implies
\begin{align*}
S_{\alpha}^{\rho} f(x)^2  & \lesssim \int_{0}^{{+\infty} } \left|{\mathcal M}_{\rho}[(t\mathcal{L})e^{-t\mathcal{L}}f](x)\right|^2 \, \frac{dt}{t^{1+\alpha}} .
\end{align*}
Then by  Proposition \ref{prop:FS} it follows that for every $p>\rho$ (since $2>\rho$) 
\begin{align*}
\| S_{\alpha}^{\rho} f \|_{L^p}  & \lesssim \left\| \left(\int_{0}^{{+\infty} } \left|{\mathcal M}_{\rho}[(t\mathcal{L})e^{-t\mathcal{L}}f]\right|^2 \, \frac{dt}{t^{1+\alpha}}\right)^{1/ 2} \right\|_{p} \\
& \lesssim \left\| \left(\int_{0}^{{+\infty} } \left|(t\mathcal{L})e^{-t\mathcal{L}}f\right|^2 \, \frac{dt}{t^{1+\alpha}}\right)^{1/ 2} \right\|_{p} \\
& \lesssim \|\mathcal{L}^{\alpha / 2} f \|_{p}.
\end{align*}
\end{proof}

The same proof holds when replacing the oscillation by a Poincar\'e inequality:

\begin{proposition}  \label{prop:poinc2}
Assume $(G_{p_0})$ with the Poincar\'e inequality $(P_{p_0})$ for some $p_0\in(2,{+\infty})$. Let $\alpha \in(0,1)$, $\rho\in(1,2)$ and $p\in[2,p_0)$. Then for every $f \in \dot{L}^p_{\alpha}(M,\mathcal{L},\mu)$, we have
$$ \|f\|_{\dot L^p_\alpha} \simeq \|S^{\rho}_\alpha(f)\|_{p}.$$
\end{proposition}

\begin{proof} First, using the combination $(G_{p_0})$ and $(P_{p_0})$ as detailed in the proof of \cite[Theorem 3.4]{BF2} with \cite[Remark 3.5]{BF2}, we know that we have the following inequality: for every $\rho\in(1,2)$, every ball $B_r$ of radius $r>0$ and $h=(t\mathcal{L})e^{-t\mathcal{L}}f$,
% \begin{equation} 
$$ \left( \aver{B_r} |h-\aver{B_r} h\, d\mu |^{p_0} \, d\mu \right)^{1/p_0} \lesssim r \left(\aver{2B_r} |\nabla h |^\rho \, d\mu \right)^{1/\rho} + r^2 \left\| \mathcal{L} h \right\|_{L^\infty(4B_r)}.$$
Writing ${\mathcal L}h=t{\mathcal L}^2 e^{-t\mathcal{L}}f = e^{-\frac{t}{2}\mathcal{L}} t{\mathcal L}^2 e^{-\frac{t}{2}\mathcal{L}}f$ with $L^\rho$-$L^\infty$ off-diagonal estimates of $e^{-\frac{t}{2}\mathcal{L}}$, we deduce that
$$ \left\| t \mathcal{L} h \right\|_{L^\infty(4B_r)} \lesssim \inf_{x\in B_r} {\mathcal M}_{\rho}[ (t\mathcal{L})^2e^{-\frac{t}{2}\mathcal{L}}f](x).$$

We then repeat the exact same proof as for Proposition \ref{prop:osci}, with the following estimate on the oscillation, 
$$ \rho \text{-}\osc_{B(x,r)}[(t\mathcal{L})e^{-t\mathcal{L}}f] \lesssim \left(\frac{r}{\sqrt{t}}\right) {\mathcal M}_{\rho}[ \sqrt{t} \nabla (t\mathcal{L})e^{-\frac{t}{2}\mathcal{L}}f](x) + \left(\frac{r}{\sqrt{t}}\right)^2 {\mathcal M}_{\rho}[ (t\mathcal{L})^2e^{-\frac{t}{2}\mathcal{L}}f](x).$$
%where we used any exponent $\rho\in(1,2)$ and where the last inequality comes as explained in \eqref{eq:p2harm}.
Hence, we have a pointwise estimate
\begin{align*}
S_{\alpha}^{\rho} f(x)  \lesssim \left(\int_{0}^{{+\infty} } \left|{\mathcal M}_{\rho}[\sqrt{t} \nabla (t\mathcal{L})e^{-t\mathcal{L}}f](x)\right|^2  +  \left|{\mathcal M}_{\rho}[ (t\mathcal{L})^2e^{-t\mathcal{L}}f](x)\right|^2 \, \frac{dt}{t^{1+\alpha}}\right)^{1/ 2}.
\end{align*}
The proof is then completed by taking the $L^p$-norm of both sides of the previous inequality and using  Proposition \ref{prop:FS} as well as the $L^p$-boundedness of the vertical square function (which is a consequence of the combination $(G_{p_0})$ with $(P_{p_0})$, see Proposition \ref{prop:square-function} (iii)) and of the horizontal square function.
\end{proof}

\begin{proposition} \label{prop:nec}
Assume \eqref{dnu}. Let $\rho,p\in(1,{+\infty})$ with $\rho\leq p$, $\nu<p$, and let $\alpha\in(\frac{\nu}{p},1)$. Assume that for every $f\in \dot{L}^p_{\alpha}(M,\mathcal{L},\mu)$, we have
$$ \|f\|_{\dot L^p_\alpha} \simeq \|S^{\rho}_\alpha(f)\|_{p}.$$
Then   $(H^{\alpha-\frac{\nu}{\rho}}_{p,\rho})$ holds, and also $(P_2)$.
\end{proposition}

\begin{proof}
Let $r\leq \sqrt{t}$, and let $B_r$,$B_{\sqrt{t}}$ be two concentric balls of respective radii $r,\sqrt{t}$. For every $x$ and $s>0$, denote the ball $B_s(x)=B(x,s)$. Then for $h=e^{-t\mathcal{L}}f$, we have for $s\in[r,2r]$
\begin{equation} \rho \text{-}\osc_{B_r(x)}(h) \lesssim \rho \text{-}\osc_{B_{s}(x)}(h). \label{eq:oscill} \end{equation}
So
\begin{align*}
 \rho \text{-}\osc_{B_r(x)} (h) \lesssim r^\alpha \left(\int_{r}^{2r} \left[s^{-\alpha} \rho \text{-}\osc_{B_s(x)}(h) \right]^2 \frac{ds}{s}\right)^{1/2} \lesssim r^\alpha S_\alpha^\rho(h) (x).
\end{align*}
Consequently, 
\begin{align*}
\left(\aver{B_{\sqrt{t}}} \rho \text{-}\osc_{B_r(x)} (h)^\rho \, d\mu(x) \right)^{1/\rho} & \lesssim  r^\alpha \left( \aver{B_{\sqrt{t}}}  |S_\alpha^\rho(h) (x)|^\rho \, d\mu(x) \right)^{1/\rho} \\
& \lesssim  r^\alpha \left( \aver{B_{\sqrt{t}}}  |S_\alpha^\rho(h) (x)|^p \, d\mu(x) \right)^{1/p} \\
& \lesssim  r^\alpha |B_{\sqrt{t}}|^{-1/p}  \|S_\alpha^\rho(h) \|_{p}.
\end{align*}
So using the assumption and the analyticity of the semigroup on $L^p$, we get
\begin{align*}
\left(\aver{B_{\sqrt{t}}} \rho \text{-}\osc_{B_r(x)}(e^{-t\mathcal{L}}f) ^\rho \, d\mu(x) \right)^{1/\rho} & \lesssim  r^\alpha |B_{\sqrt{t}}|^{-1/p}  \| \mathcal{L}^{\alpha/2} e^{-t\mathcal{L}}f \|_{p} \\
& \lesssim  \left(\frac{r}{\sqrt{t}}\right)^\alpha |B_{\sqrt{t}}|^{-1/p}  \| f \|_{p},
\end{align*}
which yields in particular (since $B_r \subset B_{\sqrt{t}}$)
\begin{align*}
\left(\aver{B_{r}} \rho \text{-}\osc_{B_r(x)}(e^{-t\mathcal{L}}f) ^\rho \, d\mu(x) \right)^{1/\rho} & \lesssim  \left(\frac{r}{\sqrt{t}}\right)^{\alpha-\frac{\nu}{\rho}} |B_{\sqrt{t}}|^{-1/p}  \| f \|_{p}.
\end{align*}
Since for $x\in B_r$, the two balls $B_r(x)$ and $B_r$ have equivalent measures, we deduce by doubling that
\begin{align*}
\rho \text{-}\osc_{2 B_r}(e^{-t\mathcal{L}}f)& \lesssim  \left(\frac{r}{\sqrt{t}}\right)^{\alpha-\frac{\nu}{\rho}} |B_{\sqrt{t}}|^{-1/p}  \| f \|_{p},
\end{align*}
for every $r<\sqrt{t}$, which is $(H^{\alpha-\frac{\nu}{\rho}}_{p,\rho})$.  Then  Proposition \ref{prop:bcf1}  yields $(P_2)$.
\end{proof}

\section{Chain rule and paralinearisation} \label{sec:chainrule}

This section is devoted to the proof of a chain rule in our abstract setting. That is, we show stability of Sobolev spaces with regard to the composition of functions with a regular map. We follow the same approach as in \cite{cm}, which relies on paraproducts. In the sequel, we establish a paralinearisation result. This is a deeper and more general result than the chain rule, but requires more regularity on the nonlinearity.

\begin{theorem}[Chain rule] \label{thm:chainrule} Let $(M,d,\mu, {\mathcal E})$  be a doubling metric  measure   Dirichlet space with a ``carr\'e du champ'' satisfying \eqref{due}.  Let $F \in C^2({\mathbb R})$ be a nonlinearity with $F(0)=0$.
 Let $\alpha\in(0,1)$ and $p\in(1,{+\infty}]$. For a function $f\in L^\infty(M,\mu) \cap \dot L^p_\alpha(M,\mathcal{L},\mu)$, we have
$$ F(f)\in L^\infty(M,\mu) \cap \dot L^p_\alpha(M,\mathcal{L},\mu)$$ in the following situations:
\begin{enumerate}
\item[i)] if $p\leq 2$ and $\alpha\in(0,1)$;
\item[ii)] if $2<p < p_0$, $\alpha\in(0,1)$ and under $(G_{p_0})$ for some $p_0>2$; 
\item[iii)] if $2<q$, $0<\alpha<1-\kappa$ and under $(G_{q})$ with $({DG}_{2,\kappa})$.
\end{enumerate}
More precisely, we have the following estimate: for every $L>0$ there exists a constant $C:=C(F,L)$ such that for every $f\in  L^\infty(M,\mu) \cap \dot L^p_\alpha(M,\mathcal{L},\mu)$ with $\|f\|_\infty\leq L$, there holds
$$  \| F(f) \|_{\dot L^p_\alpha} \leq C \|f\|_{\dot L^p_\alpha}.$$
\end{theorem}

\begin{rem} In Section \ref{sec:osci} and in \cite{CRT},  \cite{BBR}, under certain extra assumptions (in particular a Poincar\'e inequality), Sobolev norms are shown to be equivalent to the $L^p$-norm of some quadratic functional. Then the chain rule is a direct consequence, and holds for every Lipschitz map $F$.

Under the weaker assumptions of Theorem \ref{thm:chainrule} we do not expect to have such a characterisation in general (see also Proposition \ref{prop:nec}), and the paraproduct approach requires more regularity on $F$ in order to obtain the chain rule.
\end{rem}

\begin{proof} Consider first a more regular function $f\in (\calS^p + N(\mathcal{L})) \cap \dot L^p_\alpha \cap L^\infty$.
Fix a large enough integer $\PP$, and consider the approximation operators $P_t,Q_t$ and the paraproduct $\Pi$ associated with this parameter as defined in \eqref{def:paraproduct}. We represent the nonlinearity as
\begin{align*}
F(f) & = \lim_{t\to 0} F(P_t f) - \lim_{t\to {+\infty}} F(P_t f) + F(\mathsf{P}_{N(\mathcal{L})}(f)),
\end{align*}
where the limit is taken in $L^p(M,\mu)$. This is a consequence of Proposition \ref{prop:reproducing} and the fact that $F$ is Lipschitz, since then
\begin{align*}
\| F(f)-F(P_t f) \|_p & \lesssim \| f-P_t f\|_p \to 0, \quad t \to 0^+,
%\lesssim \int_0^t \|Q_s^{(\PP)} f \|_p \frac{ds}{s} \lesssim \int_0^t s \|Q_s^{(\PP-1)} \mathcal{L}f \|_p \frac{ds}{s}  \\
% & \lesssim \int_0^t s \|f_2\|_{p} \frac{ds}{s} \lesssim t \|f_2\|_{p},
\end{align*}
and similarly
\begin{align*}
\| F(\mathsf{P}_{N(\mathcal{L})}(f)) -F(P_t f) \|_p & \lesssim \| \mathsf{P}_{N(\mathcal{L})}(f)-P_t f\|_p \to 0, \quad t \to +\infty.
%\lesssim \int_t^{+\infty} \|Q_s^{(\PP)} f \|_p \frac{ds}{s} \\ 
%&\lesssim \int_t^{+\infty} \|Q_s^{(\PP+1)} f_1 \|_p \frac{ds}{s^2}  \lesssim \int_t^{+\infty}   \|f_1\|_{p} \frac{ds}{s^2} \lesssim t^{-1} \|f_1\|_{p}.
\end{align*}
From this decomposition, we deduce 
\begin{align*}
F(f)  & = -\int_0^{+\infty} \frac{d}{dt} F(P_t f) \, dt + F(\mathsf{P}_{N(\mathcal{L})}(f)) \\
 & = -\int_0^{+\infty} Q_t f \cdot F'(P_t f) \, \frac{dt}{t}+ F(\mathsf{P}_{N(\mathcal{L})}(f)).
\end{align*}
According to Proposition \ref{prop:kernel}, $\mathsf{P}_{N(\mathcal{L})}(f)$ is equal to $0$ or to a constant (depending if the ambient space is bounded or not), therefore
$$ F(\mathsf{P}_{N(\mathcal{L})}(f)) \in N(\mathcal{L}).$$
Consequently, in order to estimate $F(f)$ in the homogeneous Sobolev space, we only have to control the first term
\begin{equation} \bar F(f) := \int_0^{+\infty} Q_t f \cdot F'(P_t f) \, \frac{dt}{t}. \label{eq:chain} \end{equation}

The representation \eqref{eq:chain} does not exactly match the definition of a paraproduct. However, in the study of paraproducts in the previous sections, we only used the following three properties of the term $H(t,x) = P_t g (x)$:
\begin{enumerate}
\item Uniform boundedness $\sup_{t>0} \|H(t,\cdot)\|_{\infty} \lesssim \|g\|_{\infty}$;
\item $L^2$-$L^2$ (resp. $L^p$-$L^p$) gradient estimates of $\nabla H(t,\cdot)$ at the scale $\sqrt{t}$ in case i) and iii)(resp. ii));
\item $L^2$-$L^2$  (resp. $L^p$-$L^p$) global estimate for the square function $ \|\nabla H(t,\cdot)\|_{L^2(\frac{dt}{t})}$ in situation i) and iii) (resp. ii)).
\end{enumerate}
We refer the reader to Theorem $\ref{thm:tent-extrap-small}$ (whose proof relies on Theorem $\ref{thm:KLpLqII}$) for case i), to Theorem 
$\ref{thm:extrap>2}$ for case ii) and to Theorem $\ref{thm:tent-extrap-large}$ (whose proof relies on Theorem $\ref{thm:KLpLqIII}$) for case iii).

By  \eqref{eq:chain}, following the same proof as for the paraproduct, we will have shown that $\bar F(f)\in  \dot L^p_\alpha$ (and so $F(f)\in  \dot L^p_\alpha$)  as soon as we will have checked that the quantity $H(t,x):= F'(P_t f(x))$ satisfies properties $(a), (b)$ and $(c)$.
Since $f\in L^\infty(M,\mu)$, $P_t f$ is uniformly bounded, and since $F'$ is continuous, also $ F'(P_t f(x))$ is uniformly bounded, hence property $(a)$.
Due to the chain rule,
$$ \nabla H(t,x) = F''(P_t f(x)) \nabla P_t f,$$
and since also $ F''(P_t f(x)) $ is uniformly bounded, we deduce that $\nabla H(t,\cdot)$ satisfies the same Davies-Gaffney estimates as $\nabla P_t f$, hence property $(b)$ is checked. A similar reasoning holds also for property $(c)$.

In this way, repeating the same proof as for the paraproduct gives that $\bar F(f)\in \dot L^p_\alpha$.

Consequently, we get that for every $f\in (\calS^p + N(\mathcal{L})) \cap \dot L^p_\alpha \cap L^\infty$, one has $F(f)\in \dot L^p_\alpha$ and
\begin{equation} \|F(f)\|_{\dot L^p_\alpha} \lesssim \phi( \|f\|_{\dot L^p_\alpha\cap L^\infty}), \label{eq:ext} \end{equation}
where $\phi$ is some non-decreasing function. We already know that  $(\calS^p + N(\mathcal{L})) \cap \dot L^p_\alpha \cap L^\infty$ is dense in $\dot L^p_\alpha \cap L^\infty$. This allows us to extend the map $f \mapsto F(f)$ on the whole Banach space $\dot L^p_\alpha \cap L^\infty$: indeed for $(f_n)_n$ a Cauchy sequence, we easily check that $\bar F(f_n)$ (and so $(F(f_n))_n$) still is a Cauchy sequence in $\dot L^p_\alpha \cap L^\infty$ , since
$$ \bar F(f_n)- \bar F(f_m) = \int_0^{+\infty} Q_t (f_n-f_m) \cdot F'(P_t f_n) \, \frac{dt}{t} + \int_0^{+\infty} Q_t f_m \cdot [F'(P_t f_n)-F'(P_t f_m)] \, \frac{dt}{t}$$
and the two previous quantities can be bounded by the same reasoning as previously. Using that $F''$ is continuous and so is uniformly continuous on a bounded interval containing all the values of the sequence $(f_n(x))_n$, we let the reader check that the quantity $F'(P_t f_n)-F'(P_t f_m)$ still satisfies properties $(a), (b)$ and $(c)$, involving a control in terms of $\|f_n-f_m\|_{\dot L^p_\alpha \cap L^\infty}$.

In this way, $f \mapsto F(f)$ can be extended on the whole Banach space $\dot L^p_\alpha \cap L^\infty$ and \eqref{eq:ext} remains valid on the whole space.
\end{proof}

\begin{theorem}[Paralinearisation] \label{thm:paralinearisation} Let $(M,d,\mu, {\mathcal E})$  be a doubling metric  measure   Dirichlet space with a ``carr\'e du champ'' satisfying \eqref{due}. Assume uniform volume growth (also called a local Ahlfors regularity): there exist constants $c_1,c_2$ such that for every $x\in M$ and every radius $r\in(0,1]$, one has
\begin{equation} c_1 \leq \frac{|B(x,r)|}{r^\nu} \leq c_2. \label{eq:growth} \end{equation}
Let $F \in C^3(\R)$ be a nonlinearity with $F(0)=0$, and let $\alpha\in(0,1)$, $p\in(1,{+\infty})$ with $\alpha p >\nu$. Let  $f\in L^\infty(M,\mu) \cap \dot L^p_\alpha(M,\mathcal{L},\mu)$.
 Then  there exists $\PP_0:=\PP_0(\nu,p)$ such that for $\PP\geq \PP_0$, we have the paralinearisation
$$ F(f)-\Pi_{F'(f)}(f) \in L^\infty(M,\mu) \cap \dot L^p_\alpha(M,\mathcal{L},\mu) \cap \dot L^p_{\alpha+\rho}(M,\mathcal{L},\mu)$$
in the following situations:
\begin{enumerate}
\item[i)] if $p\leq 2$ (and $\nu<2$), $\alpha\in(0,1)$, $0<\rho< \min\{1-\alpha,\alpha-\frac{\nu}{p}\}$; 
\item[ii)] if $p>\nu$, $0<\alpha<1-\frac{\nu}{p}$, $0<\rho< \min\{1-\frac{\nu}{p}-\alpha,\alpha-\frac{\nu}{p}\}$ and under $(G_p)$.
\end{enumerate}
\end{theorem}

\begin{rem} We let the reader check the following (easy) extension (also valid for Theorem $\ref{thm:chainrule}$): consider a regular function $F:M\times \R \rightarrow \R$ such that both $F(x,\cdot)$ and $\nabla_x F(x,\cdot)$ satisfy the assumptions of Theorem \ref{thm:paralinearisation}. 
Then the result still holds with the following paralinearisation formula:
$$ x\mapsto F(x,f(x))-\Pi_{\partial_t F(x,f(x))}(f)(x) \in L^\infty \cap \dot L^p_\alpha \cap \dot L^p_{\alpha+\rho}.$$
\end{rem}

\begin{proof}
Using \eqref{eq:chain}, one may write
$$ F(f) = \Pi_{F'(f)}(f) + R$$
with the remainder
$$ R:= \int_0^{+\infty} Q_t f \cdot \left[F'(P_t f) - P_t F'(f)\right] \, \frac{dt}{t} + F(\mathsf{P}_{N(\mathcal{L})}(f)).$$
As previously, the second term is bounded and belongs to any Sobolev space (since it is equal to a constant). So we only have to focus on the first part and as previously, we are going to check that the quantity $H(t,x):=F'(P_t f(x)) - P_t [F'(f)] (x)$ satisfies more ``regular'' properties than $(a)$, $(b)$ and $(c)$.
Using the mean value theorem, one obtains
\begin{align*}
|H(t,x)| & \leq \left|F'(P_t f(x)) -F'(f(x))\right| + \left| F'(f(x)) - P_t [F'(f)] (x)\right| \\
 & \leq \|F''\|_{\infty} |(1-P_t)[f] (x)| + |(1-P_t)[F'(f)] (x)|.
\end{align*}
Then for the function $h=f$ or $h=F'(f)$ belonging to $\dot L^p_\alpha$ (due to the previous Theorem applied to $F'$), we have
\begin{align*}
\|(1-P_t)h\|_{\infty } & \leq \int_0^t \|Q_s h\|_{\infty} \frac{ds}{s} \nonumber \\
 & \lesssim \left(\int_0^t s^{\alpha / 2} \|(s\mathcal{L})^{-\alpha/2}Q_s\|_{p \to \infty} \frac{ds}{s}\right)  \| h\|_{\dot L^p_\alpha} \nonumber \\
 & \lesssim \left(\int_0^t s^{\alpha / 2} s^{-\nu/2p} \frac{ds}{s} \right) \| h\|_{\dot L^p_\alpha} \label{eq:uniform} \\
& \lesssim t^{\frac{\alpha}{2} -\frac{\nu}{2p}} \| h\|_{\dot L^p_\alpha}, \nonumber 
\end{align*}
as soon as $\alpha >\frac{\nu}{p}$.
So with implicit constants depending on $f$, we deduce that 
$$ \| H(t,\cdot)\|_{\infty} \lesssim t^{\frac{\alpha}{2} -\frac{\nu}{2p}},$$
instead of  $(a)$, which is better for small $t\lesssim 1$.

Similarly, we have
\begin{align*}
 \nabla H(t,\cdot) & = F''(P_t f) \nabla P_t f - \nabla P_t [F'(f)]   \\
  & = \left(F''(P_t f) \nabla P_t f   - F''(f) \nabla P_t f \right) + \left( F''(f) \nabla P_t f - \nabla P_t [F'(f)]  \right).
  \end{align*}
 As previously, the first term satisfies properties $(b)$ and $(c)$ with the extra coefficient $t^{\frac{\alpha}{2} -\frac{\nu}{2p}}$. The second term is more difficult: 
we aim to take advantage of the fact that $f,F'(f)\in L^\infty \cap \dot L^p_\alpha \subset \dot L^\infty_s$, with any exponent $0<s<\alpha-\frac{\nu}{p}$ (see Lemma \ref{lem:fin}). Let us write
$$ II_t(\phi_1,\phi_2):=\left( F''(f) \nabla P_t \phi_1  - \nabla P_t [\phi_2]  \right).$$
\begin{itemize}
\item For the diagonal part, we use the global $L^2$-boundedness, shown in Lemma \ref{lem:1} below,
$$ \| \sqrt{t} \nabla P_t \mathcal{L}^{-s/2} \|_{2 \to 2} \lesssim t^{s/2}.$$
Therefore, we have for every ball $B$ of radius $\sqrt{t}$
\begin{align*} 
 |B|^{-{1/2}} \| II_t( \Eins_B f, \Eins_B F'(f))\|_{L^2(B)} & \lesssim t^{s/2}(\|\mathcal{L}^{s/ 2}f\|_{\infty} + \|\mathcal{L}^{s/ 2}F'(f)\|_{\infty}) \\
  & \lesssim t^{\frac{s-1}{2}}(\|f\|_{\dot L^p_\alpha } + \|F'(f)\|_{\dot L^p_\alpha}).
 \end{align*}
\item For the off-diagonal part, we use Lemma \ref{lem:2} below to obtain $L^2$-$L^2$ off-diagonal estimates: for every ball $B,B_1$ of radius $\sqrt{t}$ with $\sqrt{t} \leq d(B,B_1)$ 
\begin{align*}
 \|II_t( \mathcal{L}^{-s/2}  \Eins_{B_1} \mathcal{L}^{s/2} f, \mathcal{L}^{-s/2} \Eins_{B_1} \mathcal{L}^{s/2} F'(f))\|_{L^2(B)} & \\
 & \hspace{-6cm} = \| F''(f) \nabla (P_t -1)\mathcal{L}^{-s/2}  (\Eins_{B_1} \mathcal{L}^{s/2} f) + \nabla (P_t-1)\mathcal{L}^{-s/2}  [\Eins_{B_1} \mathcal{L}^{s/2} F'(f)] \|_{L^2(B)} \\
  & \hspace{-6cm} \lesssim t^{\frac{s-1}{2}} \left(1+\frac{d(B,B_1)^2}{t} \right)^{-M} \left[ \|\mathcal{L}^{s/2} f\|_{L^2(B_1)} +\|\mathcal{L}^{s/2}  F'(f)\|_{L^2(B_1)}  \right],
\end{align*}
where $M$ can be chosen arbitrarily large.
\end{itemize}
This proves that $H(t,\cdot)$ satisfies  $(b)$ with an extra factor $t^{s/2}$.
By the same reasoning we obtain that $H(t,\cdot)$ satisfies  $(c)$ with an extra factor $t^{s/2}$, which yields the $L^2-L^2$ global estimate for the square function $ \| t^{-s/2} \nabla H(t,\cdot)\|_{L^2((0,1],\frac{dt}{t})}$ in situations i) and ii).

So finally, for $t\leq 1$ (which corresponds to the situation where the previous inequalities are improvements), we obtain that the quantity $H(t,\cdot)$ satisfies Properties $(a)$, $(b)$ and $(c)$ with an extra factor $t^{s/2}$, with $s<\alpha-\frac{\nu}{p}$.

Then coming back to the proof of boundedness of the paraproduct, this gain allows to prove that the remainder term
$$ R \in L^\infty \cap \dot L^p_\alpha \cap \dot L^p_{\alpha+s},$$
as soon as  $s>0$ and $\alpha+s$ in the range allowed by the proof ($\alpha+s<1$ in case i) and $\alpha+s<1-\frac{\nu}{p}$ in case ii)).
\end{proof}

\begin{lemma}[Sobolev embedding] \label{lem:fin}  Let $(M,d,\mu, {\mathcal E})$  be a doubling metric  measure   Dirichlet space with a ``carr\'e du champ'' satisfying \eqref{due} and the uniform volume growth \eqref{eq:growth}. Then for $\alpha>0$, $p>1$ with $\alpha p>\nu$, we have
$$  L^\infty(M,\mu) \cap \dot L^p_\alpha(M,\mathcal{L},\mu) \subset \dot L^\infty_s(M,\mathcal{L},\mu),$$
for any exponent $0<s<\alpha-\frac{\nu}{p}$.
\end{lemma}

\begin{proof} Let $f\in L^\infty \cap \dot L^p_\alpha$. Then
$$ \mathcal{L}^{s\over 2} f = \int_0^1 \mathcal{L}^s (t\mathcal{L}) e^{-t\mathcal{L}}f \, \frac{dt}{t} + \mathcal{L}^{s\over 2} e^{-\mathcal{L}} f .$$
For the second term, using the $L^\infty$-boundedness of $\mathcal{L}^{s\over 2} e^{-\mathcal{L}} f$ (due to the decay of its kernel, see Lemma \ref{prop:kernel-est}) we have
\begin{align*}
\| \mathcal{L}^{s\over 2} e^{-\mathcal{L}} f \|_\infty \lesssim \|f\|_\infty.
\end{align*}  
For the second term, we use that
\begin{align*}
\| \mathcal{L}^s (t\mathcal{L}) e^{-t\mathcal{L}}f\|_\infty & \lesssim \|t\mathcal{L}^{1+\frac{s-\alpha}{2}}e^{-t\mathcal{L}}\|_{p \to \infty} \|f\|_{\dot L^p_\alpha} \\
 & \lesssim t^{\frac{\alpha-s}{2}} t^{-\frac{\nu}{2p}}  \|f\|_{\dot L^p_\alpha},
\end{align*}
where we used the pointwise estimate of the kernel of $(t\mathcal{L})^{1+\frac{s-\alpha}{2}}e^{-t\mathcal{L}}$ (due to Lemma  \ref{prop:kernel-est}) with the uniform control of the volume \eqref{eq:growth}. We then conclude by integrating this estimate.
\end{proof}

\begin{lemma} \label{lem:1} Let $\epsilon\in(0,1)$. Under $(G_p)$ for $p\geq 2$  we have
$$ \| \sqrt{t} \nabla \mathcal{L}^{-\frac{\epsilon}{2}} P_t \|_{p \to p} \lesssim t^{\epsilon/2}.$$
\end{lemma}

\begin{proof} We decompose
$$ \sqrt{t} \nabla \mathcal{L}^{-\frac{\epsilon}{2}} P_t = \int_0^{+\infty} \sqrt{t} \nabla e^{-s\mathcal{L} } P_t \frac{ds}{s^{1-\frac{\epsilon}{2}}}.$$
Then we use that for $s\leq t$, by $(G_p)$  we have
$$ \| \sqrt{t} \nabla e^{-s\mathcal{L} }P_t \|_{p \to p} =\| \sqrt{t} \nabla P_t e^{-s\mathcal{L} }\|_{p \to p} \leq \| \sqrt{t} \nabla P_t \|_{p \to p} \| e^{-s\mathcal{L} }\|_{p \to p} \lesssim 1.$$
For $s\geq t$, $(G_p)$ yields
$$ \| \sqrt{t} \nabla e^{-s\mathcal{L} }P_t \|_{p \to p} \leq \| \sqrt{t} \nabla e^{-s\mathcal{L} } \|_{p \to p} \| P_t \|_{p \to p}  \lesssim \left(\frac{t}{s}\right)^{1/ 2}.$$
We conclude the proof by integrating these inequalities.
\end{proof}

\begin{lemma} \label{lem:2} Let $\epsilon\in(0,1)$. Under $(G_p)$ for $p\geq 2$, we have for all balls $B_1,B_2$ of radius $\sqrt{t}$ with $d(B_1,B_2)\geq \sqrt{t}$
$$ \| \sqrt{t} \nabla \mathcal{L}^{-\frac{\epsilon}{2}} (P_t-I) \|_{L^p(B_1) \to L^p(B_2)} \lesssim t^{\frac{\epsilon}{2}} \left(1+\frac{d^2(B_1,B_2)}{t}\right)^{-M},$$
where $M$ can be chosen arbitrarily large (depending on $P_t$).
\end{lemma}

\begin{proof} For $\epsilon=1$, this corresponds to off-diagonal estimates for the Riesz transform, see \cite[Lemma 3.1]{ACDH}. The exact same proof still holds for $\epsilon\in(0,1]$.
\end{proof}

\appendix

\section{About the $p$-independence of $(H^\eta_{p,p})$}\label{AppA}

In this appendix, we study in more detail the $p$-independence of the property $(H^\eta_{p,p})$ for $p \in [1,+\infty]$ and $\eta\in(0,1]$ and prove the two last statements of Proposition \ref{prop:bcf1}.

All of this appendix is valid in a more general setting than the one presented in the introduction. It is enough to consider a metric measure space $(M,d,\mu)$ satisfying $(VD)$, endowed with a  semigroup  $(e^{-tL})_{t>0}$ acting on $L^p(M,\mu)$, $1\le p\le +\infty$.
For $1\le p\le +\infty$, let us write the $L^p$-oscillation for $u\in L^p_{loc}(M,\mu)$ and a ball $B$ a ball by
$$ p\text{-}\osc_B(f) := \left(\aver{B} |f-\aver{B} f \,d\mu|^p \,d\mu\right)^{1/p}$$
if $p<+\infty$, and
$$ \infty\text{-}\osc_B(f) := \esssup_{B}| f - \aver{B}f\,d\mu |.$$
Recall that we denote by $\calM$ the Hardy-Littlewood maximal operator, 
and by $\calM_p$ the operator defined by $\calM_p(f):=[\calM(|f|^p)]^{1/p}$, $f \in L^1_{\loc}(M,\mu)$, $p \in [1,+\infty)$. We set $\calM_{\infty}(f):=\|f\|_{\infty}$, $f \in L^\infty(M,\mu)$. \\

In \cite{Du}, gradient estimates for the heat semigroup are studied in the Riemannian setting, but the proofs rely only on the finite propagation speed  property, therefore extend to the setting of a metric measure space with a ``carr\'e du champ". More precisely, it is proved that, under \eqref{d} and \eqref{UE}, the condition
\begin{equation} \sup_{t>0} \ \sup_{x\in M } \  |B(x,\sqrt{t})|^{1-\frac{1}{q}} \|\sqrt{t} |\nabla p_t(x,\cdot)|\|_{{q}} < +\infty \label{eq:ap0}\end{equation}
is independent of $q\in[1,+\infty]$ and is in particular equivalent to Gaussian pointwise estimates for the gradient of the heat kernel.
Since for $q=p'$
$$\sup_{x\in M } \  \|\sqrt{t} |\nabla p_t(x,\cdot)|\|_{{q}} = \|\sqrt{t} |\nabla e^{-tL}|\|_{p\to \infty},$$
this property can be thought of, at least in the polynomial volume growth situation  $V(x,r)\simeq r^\nu$, as follows: the quantity $\|\sqrt{t} |\nabla e^{-tL}|\|_{p\to \infty}$ does not depend on the exponent $p\in[1,+\infty]$.

Even if the full version of this result in \cite{Du} is really non-trivial, it appears that a localised counterpart is indeed very easy: more precisely, the property
\begin{equation} \sup_{t>0} \   \sqrt{t} |\nabla e^{-tL} f (x)| \lesssim \calM_p(f)(x)  \label{eq:ap1} 
\end{equation}
 is $p$-independent.
This fact directly follows by writing $\nabla e^{-tL} = \left(\nabla e^{-\frac{t}{2}L}\right) e^{-\frac{t}{2}L}$ with a semigroup $e^{-\frac{t}{2}L}$ satisfying all $L^{p}$-$L^{q}$ off-diagonal estimates (since the heat kernel satisfies pointwise Gaussian estimates), so that for every $p,q\in[1,+\infty]$ with $p< q$, we have
$$ \calM_{q}(e^{-tL}f)(x) \lesssim \calM_{p}(f)(x).$$
The estimate for $p\geq q$ follows from H\"older's inequality. 
In other words, the localised property \eqref{eq:ap1} is much easier to  prove than the full ``global'' version \eqref{eq:ap0}.

The inequality $(H^\eta_{p,p})$ is the H\"older counterpart of the   $L^p$ - $L^\infty$ Lipschitz  regularity property  of the semigroup \eqref{eq:ap0}. Following the previous observation (and the results of \cite{Du}, which can be extended to the situation of H\"older regularity instead of gradient estimates), it is natural to study the $p$-independence of $(H^\eta_{p,p})$ and to do so, we recall the localised versions of $(H^\eta_{p,p})$ (already introduced in the introduction).

\begin{definition} 
Let $(M,d,\mu, L)$ as above satisfying $(VD)$ and \eqref{UE}.
Let $p,q\in[1,+\infty]$ and $\eta \in (0,1]$.  
We shall say that \eqref{gplocal} is satisfied, if 
for all $0<r\leq \sqrt{t}$, every ball $B_r$ of radius, and every function $f\in L^{p}_{\textrm{loc}}(M,\mu)$, 
\begin{equation} \label{gplocal-bis}
q\text{-}\osc_{B_r}(e^{-tL} f) \lesssim \left( \frac{r}{\sqrt{t}} \right)^\eta  \inf_{z\in B_{\sqrt{t}}} \calM_p(f)(z).
 \tag{$\overline{H}_{p,q}^\eta$}
\end{equation}
\end{definition}

Note that $(\overline{H}_{\infty,\infty}^\eta)=(H_{\infty,\infty}^\eta)$.\\

With the help of this definition, we can prove the following ``almost" $p$-independence of $(H^\eta_{p,p})$. 

\begin{theorem}
 Let $(M,d,\mu, L)$ be as above and satisfying \eqref{d} and \eqref{UE}.
Let $\eta\in(0,1]$. The property $(\overline{H}_{p,p}^\eta)$ is independent of $p\in[1,+\infty]$. The property ``$(H_{p,p}^\lambda)$ for every $\lambda<\eta$'' is independent of $p\in[1,+\infty]$.
\end{theorem}

The above theorem will be a direct consequence of self-improvement properties of $(H_{p,p}^\eta)$ and $(\overline{H}_{p,p}^\eta)$, which read as follows.

\begin{proposition} \label{proposition2} Let $(M,d,\mu, L)$ be as above and satisfying \eqref{d} and \eqref{UE}. Let $p,q\in[1,+\infty]$ and $\eta \in (0,1]$. 
Then 
\begin{itemize}
\item[(i)]  $(\overline{H}_{p,p}^\eta)$ $\Longrightarrow$ $(\overline{H}_{1,\infty}^\eta)$  $\Longrightarrow$ $(\overline{H}_{q,q}^\eta)$;
\item[(ii)] $(\overline{H}_{p,p}^\eta)$  $\Longrightarrow$ $({H}_{p,p}^\eta)$;
\item[(iii)] For every $\lambda\in[0,\eta)$, $(H_{p,p}^\eta)$  $\Longrightarrow$ $(\overline{H}_{p,p}^\lambda)$.
\end{itemize}
\end{proposition}

\begin{rem}
As a consequence of Proposition \ref{proposition2}, the property: ``there exists $\eta>0$ such that $({H}_{p,p}^\eta)$ holds" is independent of $p\in [1,+\infty]$.
\end{rem}

\begin{rem} \label{rem:sub2} All results of Appendix \ref{AppA} remain true in the context of sub-Gaussian estimates.
\end{rem}

\begin{proof}[Proof of Proposition $\ref{proposition2}$]
Let us start with (i).
First, we follow \cite[Proposition 3.1]{BCF1} (which relies on a Meyers argument to improve oscillations estimates), and the same proof allows us to improve $(\overline{H}_{p,p}^\eta)$ into $(\overline{H}_{p,\infty}^\eta)$.
Then, if $q\geq p$, we obtain from Jensen's inequality
$$ \inf_{z\in B_{\sqrt{t}}} \calM_p(f)(z)  \leq \inf_{z\in B_{\sqrt{t}}} \calM_q(f)(z), $$
therefore
$$(\overline{H}_{p,\infty}^\eta) \Longrightarrow (\overline{H}_{q,\infty}^\eta) \Longrightarrow (\overline{H}_{q,q}^\eta).$$
Now let us focus on the case $q<p$.
Consider $t>0$ and set $s=\frac{t}{2}$. Let $B_r$ be a ball of radius $r<\sqrt{t}$
 and $B_{\sqrt{t}}=\frac{\sqrt{t}}{r}B_r$ the dilated ball of radius $\sqrt{t}$.
If $r<\sqrt{s}$, we apply $(\overline{H}_{p,\infty}^\eta)$ to $e^{-sL}f$, which yields
\begin{equation} 
\esssup_{x,y \in B_r} \left| e^{-2sL}f(x) - e^{-2sL}f(y)\right| \lesssim \left( \frac{r}{\sqrt{s}} \right)^\eta \inf_{z\in B_{\sqrt{s}}} \calM_p(e^{-sL}f)(z). \label{eq:mod1} 
\end{equation}
Using \eqref{UE} together with $t=2s$, we then obtain 
$$ \esssup_{x,y \in B_r} \left| e^{-tL}f(x) - e^{-tL}f(y)\right| \lesssim \left( \frac{r}{\sqrt{t}} \right)^\eta \inf_{z\in B_{\sqrt{t}}} \calM(f)(z)  ,$$
which is $(\overline{H}_{1,\infty}^\eta)$.
The case $\sqrt{s}\leq r \leq \sqrt{t}$ is a direct consequence of \eqref{UE}, since we have $r\simeq \sqrt{t}$ and so
\begin{align*} 
\esssup_{x,y \in B_r} \left| e^{-tL}f(x) - e^{-tL}f(y)\right| & \leq 2 \|e^{-tL} f \|_{L^\infty(B_r)} \lesssim \|e^{-tL} f \|_{L^\infty(B_{\sqrt{t}})} 
 \lesssim \inf_{z\in B_{\sqrt{t}}} \calM(f)(z),
\end{align*}
which yields $(\overline{H}_{1,\infty}^\eta)$.

\noindent
Now for (ii).
Assume $(\overline{H}_{p,p}^\eta)$ for some $p\in[1,+\infty]$. First, note that for $t=2s$
\begin{align*}
 \inf_{z\in B_{\sqrt{s}}} \calM_p(e^{-sL}f)(z) & \leq  |B_{\sqrt{s}}|^{-1/p} \|e^{-sL} f\|_p + \sup_{x\in B_{\sqrt{s}}} |e^{-sL}f(x)| 
    \lesssim |B_{\sqrt{t}}|^{-1/p} \|f\|_p ,
\end{align*}
where we used \eqref{UE}. By applying the above estimate to  \eqref{eq:mod1}, we can obtain $(H_{p,\infty}^\eta)$ from  $(\overline{H}_{p,p}^\eta)$ with the same reasoning as in the proof of part (i). $(H_{p,p}^\eta)$ then easily follows. 

\noindent
Let us finally prove (iii).
Assume $(H_{p,p}^\eta)$ for some $\eta\in (0,1]$ and $p\in[1,+\infty]$. Let $B_r$, $B_{\sqrt{t}}$ be a pair of concentric balls with respective radii $r$ and $\sqrt{t}$, where $0<r\leq \sqrt{t}$.
Then we know that
$$ p\text{-}\osc_{B_r}(e^{-tL} f)  \lesssim \left( \frac{r}{\sqrt{t}} \right)^\eta  |B_{\sqrt{t}}|^{-1/p}\|f\|_p.$$
Let us split
$ \displaystyle f = \sum_{\ell \geq 0} f \Eins_{S_\ell (B_{\sqrt{t}})}$,
and define for $\ell\geq 0$
$$ I(\ell):=   p\text{-}\osc_{B_r}\left[e^{-tL}(f \Eins_{S_\ell (B_{\sqrt{t}})})\right],$$ 
where $S_\ell (B_{\sqrt{t}})$ stands for the dyadic annuli 
$$ S_\ell (B_{\sqrt{t}}):= 2^{\ell+1}B_{\sqrt{t}} \setminus 2^\ell B_{\sqrt{t}}.$$
We have  $(H_{p,p}^\lambda)$ for every $\lambda\in[0,\eta]$, therefore, for $\ell\leq 1$, 
\begin{align*}
 I(\ell)  \lesssim \left( \frac{r}{\sqrt{t}} \right)^\lambda \left( \aver{4 B_{\sqrt{t}}} |f|^p d\mu\right)^{1/p} 
   \lesssim \left( \frac{r}{\sqrt{t}} \right)^\lambda \inf_{z\in B_{\sqrt{t}}}\calM_p(f)(z).
\end{align*}
For $\ell\geq 2$, we similarly have
\begin{equation}
 I(\ell)  \lesssim \left( \frac{r}{\sqrt{t}} \right)^\eta  2^{\ell \frac{\nu}{p}} \left( \aver{2^\ell B_{\sqrt{t}}} |f|^p d\mu\right)^{1/p}. \label{eq:inter1}
\end{equation}
Moreover, using again \eqref{UE}, we have
\begin{align}
 I(\ell)  \leq 2 \left(\aver{B_r} \left| e^{-tL}(f \Eins_{S_\ell (B_{\sqrt{t}})})  d\mu \right|^{p} d\mu \right)^{1/p} 
  \lesssim  e^{-c 4^\ell} \left( \aver{2^\ell B_{\sqrt{t}}} |f|^p d\mu\right)^{1/p}, \label{eq:inter2}
\end{align} 
 which yields
\begin{align*}
\left(\aver{B_r} \left| e^{-tL}(f \Eins_{S_\ell (B_{\sqrt{t}})})  d\mu \right|^{p} d\mu \right)^{1/p} & \leq \|  e^{-tL}(f \Eins_{S_\ell (B_{\sqrt{t}})}) \|_{L^\infty(B_r)} \\
 & \leq \|  e^{-tL}(f \Eins_{S_\ell (B_{\sqrt{t}})}) \|_{L^\infty(B_{\sqrt{t}})}  \lesssim  e^{-c 4^\ell} \left( \aver{2^\ell B_{\sqrt{t}}} |f|^p d\mu\right)^{1/p}.
 \end{align*} 
 By interpolating between \eqref{eq:inter1} and \eqref{eq:inter2}, we get for every $\lambda\in[0,\eta)$, with $c_\lambda$ a constant depending on $\lambda$,
$$
 I(\ell)  \lesssim \left( \frac{r}{\sqrt{t}} \right)^\lambda e^{-c_\lambda 4^\ell} \left( \aver{2^\ell B_{\sqrt{t}}} |f|^p d\mu\right)^{1/p}.
$$
By summing over $\ell \geq 0$, we obtain
\begin{align*}
\left(\aver{B_r} \left| e^{-tL}f - \aver{B_r}e^{-tL}f\, d\mu \right|^{p} d\mu \right)^{1/p}   \leq \sum_{\ell \geq 0} I(\ell)  \lesssim \left( \frac{r}{\sqrt{t}} \right)^\lambda \inf_{z\in B_{\sqrt{t}}} \calM_p(f)(z),
\end{align*}
which is $(\overline{H}_{p,p}^\lambda)$.
\end{proof}

\end{document}